\documentclass[reqno]{amsart}
\usepackage{graphicx}
\usepackage{amscd}
\usepackage{amsmath}
\usepackage{amsfonts}
\usepackage{amssymb}
\usepackage{cite}
\usepackage{xcolor}
\usepackage{longtable}
\usepackage{float}
\def\R {\mathbb{R}}
\def\C {\mathcal{C}}
\newtheorem{proposition}{Proposition}[section]
\newtheorem{theorem}[proposition]{Theorem}
\newtheorem{corollary}[proposition]{Corollary}
\newtheorem{lemma}[proposition]{Lemma}
\theoremstyle{definition}
\newtheorem{definition}[proposition]{Definition}
\newtheorem{remark}[proposition]{Remark}
\numberwithin{equation}{section}
\newtheorem{exAPP}{Example}[section]

\begin{document}

\title[ nonlinear  multi-term fractional differential equations]
{Initial-boundary value  problems to semilinear\\ multi-term fractional differential equations}

\author[S.V. Siryk and N. Vasylyeva]
{Sergii V. Siryk and Nataliya Vasylyeva}
\address{CONCEPT Lab, Istituto Italiano di Tecnologia
\newline\indent
Via Morego 30, 16163, Genova, Italy}
\email[S. Siryk]{accandar@gmail.com}

\address{Institute of Applied Mathematics and Mechanics of NASU
\newline\indent
G.Batyuka st. 19, 84100 Sloviansk, Ukraine;
\newline\indent and
Dipartimento di Matematica, Politecnico di Milano
\newline\indent
Via Bonardi 9, 20133 Milano, Italy
}
\email[N.Vasylyeva]{nataliy\underline{\ }v@yahoo.com}


%%%%%%%%%%%%%%%%%%%%%%%%%%%%%%%%%%%%%%%%%%%%%%%%%%%%%%%%%%%%%%%%%%%%%%
\begin{abstract}
For $\nu,\nu_i,\mu_j\in(0,1)$, we analyze the semilinear integro-differential equation on the one-dimensional domain $\Omega=(a,b)$ in the unknown $u=u(x,t)$
\[
\mathbf{D}_{t}^{\nu}(\varrho_{0}u)+\sum_{i=1}^{M}\mathbf{D}_{t}^{\nu_{i}}(\varrho_{i}u)
-\sum_{j=1}^{N}\mathbf{D}_{t}^{\mu_{j}}(\gamma_{j}u)
-\mathcal{L}_{1}u-\mathcal{K}*\mathcal{L}_{2}u+f(u)=g(x,t),
\]
where $\mathbf{D}_{t}^{\nu},\mathbf{D}_{t}^{\nu_{i}}, \mathbf{D}_{t}^{\mu_{j}}$ are Caputo fractional derivatives, $\varrho_0=\varrho_0(t)>0,$ $\varrho_{i}=\varrho_{i}(t)$, $\gamma_{j}=\gamma_{j}(t)$, $\mathcal{L}_{k}$ are uniform elliptic operators with time-dependent smooth coefficients, $\mathcal{K}$ is a summable convolution kernel. Particular cases of this equation are the recently proposed advanced models of oxygen transport through capillaries. Under certain structural conditions on the nonlinearity $f$ and  orders $\nu,\nu_i,\mu_j$, the global existence and uniqueness of classical and strong solutions to the related initial-boundary value problems are established via the so-called continuation arguments method. The crucial point is searching suitable a priori estimates of the solution in the fractional H\"{o}lder  and Sobolev spaces. The problems are also studied from the numerical point of view.
\end{abstract}
%%%%%%%%%%%%%%%%%%%%%%%%%%%%%%%%%%%%%%%%%%%%%%%%%%%%%%%%%%%%%%%
\maketitle
%%%%%%%%%%%%%%%%
\section{Introduction}
\label{s1}

 \noindent Let $\Omega=(a,b)\subset\R$ be a segment, with a
boundary $\partial\Omega=\{a\}\cup\{b\}$. For an arbitrary fixed
time $T>0$, we denote
$$
\Omega_{T}=\Omega\times (0,T)\qquad \text{and}\qquad
\partial\Omega_{T}=\partial\Omega\times [0,T].
$$
We consider the semilinear equation in the unknown function $u=u(x,t):\Omega_{T}\to\R$,
\begin{equation}\label{1.1}
\mathbf{D}_{t}u-\mathcal{L}_{1}u-\mathcal{K}*\mathcal{L}_{2}u+f(u)=g(x,t),
\end{equation}
subject either to the Dirichlet boundary condition (\textbf{DBC})
\begin{equation}\label{1.2}
u=\psi(x,t)\quad\text{on}\quad \partial\Omega_{T},
\end{equation}
or to the Neumann boundary condition (\textbf{NBC})
\begin{equation}\label{1.2*}
\frac{\partial u}{\partial x}=\psi_{1}(x,t)\qquad\text{on} \quad \partial\Omega_{T},
\end{equation}
where the functions $g,f,\psi,\psi_{1},\mathcal{K}$ are prescribed.

\noindent The equation is  supplemented with the initial condition
\begin{equation}\label{1.3}
u(x,0)=u_{0}\quad \text{in}\quad \bar{\Omega}
\end{equation}
for some given initial datum $u_0$.

\noindent Here the $*$ denotes the usual time-convolution product on $(0,t)$, namely
\[
(\mathfrak{h}_{1}*\mathfrak{h}_{2})(t)=\int\limits_{0}^{t} \mathfrak{h}_{1}(t-s)\mathfrak{h}_{2}(s)ds,
\]
while the symbol $\mathbf{D}_{t}$ stands for the linear combinations of  Caputo fractional derivatives with respect to time, defined as
\begin{equation}\label{1.4}
\mathbf{D}_{t}u=\mathbf{D}_{t}^{\nu}(\varrho_0 u)+\sum_{i=1}^{M}\mathbf{D}_{t}^{\nu_{i}}(\varrho_i u)-
\sum_{j=1}^{N}\mathbf{D}_{t}^{\mu_{j}}(\gamma_j u)
\end{equation}
for any fixed $\nu\in(0,1)$ $\nu_{i}, \mu_i\in(0,\nu)$, and given positive functions $\varrho_{0}=\varrho_{0}(t),$ $\varrho_{i}=\varrho_{i}(t)$, $\gamma_{j}=\gamma_{j}(t),$ $i=1,...,M,$ and  $j=1,...,N$. We agreed that if $N=0$ or $M=0$ then the corresponding sum is missing from the above representation.
 Here $\mathbf{D}_{t}^{\theta}$ denotes  the Caputo fractional derivative of order $\theta$ with respect to time.
Let us recall the definition of the Caputo fractional derivative in the case of $\theta\in(0,1]$,
\[
\mathbf{D}_{t}^{\theta}u(x,t)=
\begin{cases}
\frac{1}{\Gamma(1-\theta)}\frac{\partial}{\partial t}\int\limits_{0}^{t}\frac{u(x,s)-u(x,0)}{(t-s)^{\theta}}ds\quad\text{if}\quad \theta\in(0,1),\\
\frac{\partial u}{\partial t}(x,t)\qquad\qquad\qquad\qquad\,\text{ if}\quad \theta=1,
\end{cases}
\]
with $\Gamma$ being the Euler Gamma-function. An equivalent definition of this derivative  in the case of absolutely continuous functions reads
\[
\mathbf{D}_{t}^{\theta}u(x,t)=
\begin{cases}
\frac{1}{\Gamma(1-\theta)}\int\limits_{0}^{t}(t-s)^{-\theta}\frac{\partial u}{\partial s}(x,s)ds
\quad\text{if}\quad \theta\in(0,1),\\
\frac{\partial u}{\partial t}(x,t)\qquad\qquad\qquad\qquad\,\text{ if}\quad \theta=1.
\end{cases}
\]
Coming to the operators involved, $\mathcal{L}_{i}$ are linear elliptic operators  of the second order with time-dependent coefficients, namely,
\begin{align*}
\mathcal{L}_{1}u&=a_{2}\frac{\partial^{2}u}{\partial x^{2}}+a_{1}\frac{\partial u}{\partial x}+a_{0} u,\\
\mathcal{L}_{2}u&=b_{2}\frac{\partial^{2}u}{\partial x^{2}}+b_{1}\frac{\partial u}{\partial x}+b_{0} u,\\
\end{align*} 
where 
$
a_i=a_i(x,t),$ $b_i=b_i(x,t).
$

The evolution equations with fractional  derivatives play an important role in the modleing of the so-called anomalous phenomena arising in Biology, Geophysics, Chemistry and Physics (see e.g. \cite{BG,BT,CDV,LT,MS,MMVBBDK} and references therein). It occurs that for certain processes the order of the time-fractional derivatives from the corresponding model equation does not remain constant. A possible method to control these phenomena is to exploit the multi-term time-fractional diffusion-wave equation, see e.g. \cite{L1}. In particular, partial case of equation \eqref{1.1} ($M=0$, $N=1$, $\varrho_0=1$ and $\gamma_1=\text{constant}$) describes oxygen delivery through capillaries \cite{MGSKA,SR}.

Published works related to the multi-term  fractional diffusion/wave equations, i.e. equations similar to \eqref{1.1} with the operator 
\begin{equation}\label{1.6}
\mathbf{D}_{t}u=\sum_{i=1}^{N}q_i\mathbf{D}_{t}^{\nu_i}u,
\end{equation}
with $q_i$ being positive, and $0\leq \nu_1<\nu_2<...<\nu_{M}$, are quite limited in spite of rich literature on their single-term version.
Exact solution of linear multi-term fractional diffusion equations with $q_i$ being positive constants on bounded domains are constructed employing eigenfunction expansion in \cite{DB,CLA,JLTB,MGSKA,ZLLA}. We quote \cite{SR,MGSKA,JDM,ZLLA} where certain numerical solutions are built to the corresponding initial-boundary value problems to evolution equations with $\mathbf{D}_{t}$ given via \eqref{1.6}. Abstract multi-term time-fractional  equations in Banach spaces are discussed in \cite{LMMZL}. Well-posedness along with a maximum principle and the long-time asymptotic behavior of the solution for the initial-boundary value problems to the these equation are studied in \cite{KPV1,LLY,LY,LSY} (see also references therein). In fine, we refer to \cite{LHY}, where initial-boundary value problems to this equation with the $x$-dependent coefficients $q_i$ are analyzed.

The principal distinction of equation \eqref{1.1} from the equations in the aforementioned  works is related to the representation of the  operator $\mathbf{D}_{t}$ (see \eqref{1.4}) as a linear combination of the multi-term fractional derivatives. Therefore, for  certain $\rho_0$ and $\gamma_{i}$, $\mathbf{D}_{t}u$  can be rewritten in the form 
\[
\frac{\partial}{\partial t}\int_{0}^{t}\mathcal{N}(t-\tau)[u(x,\tau)-u(x,0)]d\tau
\]
 with the kernel $\mathcal{N}$ being either a negative function or a function alternating in sign. 
Indeed, choosing $M=0$, $N=1$, and 
$$\gamma_{1}=1+\varrho_{0},\quad \varrho_{0}\equiv C_{\varrho}>0,$$
 and appealing to  \cite[Lemma 4]{JK}, we end up with the equality
\[
\mathbf{D}_{t}^{\nu}(\varrho_{0}u)-\mathbf{D}_{t}^{\nu_{1}}(\gamma_1 u)=\frac{\partial}{\partial t}[\mathcal{N}* (u-u(x,0))],
\]
where the kernel
\[
\mathcal{N}=C_{\varrho}\frac{t^{-\nu}}{\Gamma(1-\nu)}-(1+C_{\varrho})\frac{t^{-\mu_1}}{\Gamma(1-\mu_1)},
\]
 is \textit{negative } for $t>e^{-\gamma}$, $\gamma$ being the Euler-Mascheroni constant. It is worth noting that the nonnegativity of the kernel $\mathcal{N}$ plays a crucial role in the previous investigations of fractional partial differential equations and related initial/initial-boundary value problems. This assumption is removed in our research.
Moreover, equation \eqref{1.1} contains fractional derivatives calculated from the product of two functions: the desired solution $u$ and the prescribed coefficients $\varrho_{i}$, $\varrho_0$, $\gamma_{i}$. This peculiarity provides additional difficulties  since the typical Leibniz rule does not work in the case of fractional derivatives.

The first result on the global classical solvability to the linear version \eqref{1.1} with the operator $\mathbf{D}_{t}$ given by \eqref{1.4}, where the coefficients $\varrho_{i}=\varrho_i(x,t),$ $\gamma_{i}=\gamma_{i}(x,t)$ are alternating sign, is discussed in \cite{PSV6}. To the authors' best knowledge, there are no works in the literature addressing to one-valued global solvability to the quasilinear equation \eqref{1.1} in the case of $N\geq 1$ and positive $\gamma_{i}$. The aim of the present paper is to fill this gap, providing a well-posedness result along with the regularity of solutions in fractional H\"{o}lder and Sobolev classes for any fixed time $T$, in the case of ``power law'' memory kernel \cite{Pi}, i.e. satisfying for every $t\in[0,T]$ the bound
\[
|\mathcal{K}(t)|\leq C t^{-\beta}
\]
for some positive constant $C$ and $\beta\in[0,1)$. Indeed, boundary problems with kernels of this kind do have a practical interest: many viscoelastic materials have rapidly decreasing memory in small time values, and are therefore better described by kernels with singularities at the origin.

The technique of this paper heavily rely on the fact that we work in a one-dimensional domain. Bedsides, the main tools in our analysis are a priori estimates in fractional Sobolev and H\"{o}lder spaces. Our analysis is complemented by numerical simulations.

%%%%%%%%%%%%%%%%%%%%%%%%%%%%%%%%%%%%%%%%%%%%%%%%%%%%%%%%%%%%%%%%%%%%%%

\subsection*{Outline of the paper}

The paper is organized as follows: in Section \ref{s2}, we introduce the notation and the functional settings. The main assumptions are discussed in Section \ref{s3}. The principal results (Theorems \ref{t4.1} and \ref{t4.2}) are stated in Section \ref{s4}. Theorem \ref{t4.1} concerns to the classical global solvability of \eqref{1.1}-\eqref{1.4}, while Theorem \ref{t4.2} touches the existence and uniqueness of strong solutions to these problems. In Section \ref{s5}, we recall some definitions together with some auxiliary technical results from  fractional calculus, playing a key role in the course of this study. Section \ref{s6} is devoted to obtain a priori estimates in the fractional Sobolev and H\"{o}lder spaces. The proof of Theorem \ref{t4.1} is carried out in Section \ref{s7}. To this end, we exploit the continuation method, treating the family of problems depending on the parameter $\lambda\in[0,1]$,
\[
\mathbf{D}_{t}u-\mathcal{L}_1u-\mathcal{K}
*\mathcal{L}_{2}u=\lambda[f(u_0)-f(u)]+g(x,t)-f(u_0),
\] 
subject to the conditions \eqref{1.2}-\eqref{1.3}. At last, in Section \ref{s8}, we prove Theorem \ref{t4.2} via construction of a strong solution as a limit of approximate smooth solutions. In the final Section \ref{s9}, we study \eqref{1.1}-\eqref{1.4} from the numerical side.

%%%%%%%%%%%%%%%%%%%%%%%%%%%%%%%%%%%%%%%%%%%%%%%%%%%%%%%%%%%%%%%%%%%%%%
 
%%%%%%%%%%%%%%%%%%%%%%%%%%%%%%%%%%%%%%%%%%%%%%%%%%%%%%%%%%%%%%%%%%%%%%

\section{Functional Spaces and Notation}
\label{s2}
\noindent Throughout this work, the symbol $C$ will denote a generic positive constant, depending only on the structural quantities of the problem. We will carry out our analysis in the framework of the fractional H\"{o}lder spaces. To this end, in what follows we take two arbitrary (but fixed) parameters
\[
\alpha\in(0,1) \quad\text{and}\quad \theta\in(0,1). 
\] 
For any nonnegative integer $l$, any Banach space $(\mathbf{X},\|\cdot\|_{\mathbf{X}}),$ and any $p\geq 1,$ $s\geq 0$, we consider the usual spaces
\[
\C^{s}([0,T],\mathbf{X}),\quad\C^{l+\alpha}(\bar{\Omega}),\quad W^{s,p}(\Omega),\quad L_{p}(\Omega),\quad W^{s,p}((0,T),\mathbf{X}).
\]
Recall that for noninteger $s$,  $W^{s,p}$ is called the Sobolev-Slobodeckii space (for its definition and properties see, e.g.,   in \cite[Chapter 1]{AF}, \cite[Chapter 1]{Gr}).

Denoting for $\beta\in(0,1)$
\begin{align*}
\langle v\rangle^{(\beta)}_{x,\Omega_T}&=\sup\Big\{\frac{|v(x_1,t)-v(x_2,t)|}{|x_1-x_2|^{\beta}}:\quad x_{2}\neq x_{1},\quad x_1,x_2\in\bar{\Omega}, \quad t\in[0,T]\Big\},\\
\langle v\rangle^{(\beta)}_{t,\Omega_T}&=\sup\Big\{\frac{|v(x,t_1)-v(x,t_2)|}{|t_1-t_2|^{\beta}}:\quad t_{2}\neq t_{1},\quad x\in\bar{\Omega}, \quad t_1,t_2\in[0,T]\Big\},
\end{align*}
 we assert the following definition.
\begin{definition}\label{d2.1}
 A function $v=v(x,t)$ belongs to the class $\C^{l+\alpha,\frac{l+\alpha}{2}\theta}(\bar{\Omega}_{T})$, for $l=0,1,2,$ if the function $v$ and its corresponding derivatives are continuous and the norms here below are finite:
\begin{equation*}
\|v\|_{\C^{l+\alpha,\frac{l+\alpha}{2}\theta}(\bar{\Omega}_{T})}=
\begin{cases}
\|v\|_{\C([0,T],
\C^{l+\alpha}(\bar{\Omega}))}+\sum_{|j|=0}^{l}\langle
D_{x}^{j}v\rangle^{(\frac{l+\alpha-|j|}{2}\theta)}_{t,{\Omega}_{T}},
\qquad\qquad\qquad\qquad\quad\,
l=0,1,\\
\\
\|v\|_{\C([0,T],
\C^{2+\alpha}(\bar{\Omega}))}+\|\mathbf{D}_{t}^{\theta}v\|_{\C^{\alpha,\frac{\alpha}{2}\theta}
(\bar{\Omega}_{T})}+\sum_{|j|=1}^{2}\langle
D_{x}^{j}v\rangle^{(\frac{2+\alpha-|j|}{2}\theta)}_{t,{\Omega}_{T}},\quad l=2.
\end{cases}
\end{equation*}
\end{definition}
\noindent In a similar way, for $l=0,1,2,$ we introduce the space $\C^{l+\alpha,\frac{l+\alpha}{2}\theta}(\partial\Omega_{T})$. 

The properties of these spaces have been discussed in \cite[Section 2]{KPV3}. It is worth noting that, in the limiting case $\theta=1$, the class  $\C^{l+\alpha,\frac{l+\alpha}{2}\theta}$ coincides with the usual parabolic H\"{o}lder space $H^{l+\alpha,\frac{l+\alpha}{2}}$ (see e.g.\cite[(1.10)-(1.12)]{LSU}).

Finally, exploiting \cite[Proposition 1]{Y2}, we introduce the space $\mathcal{H}^{s}((0,T),\mathbf{X})$ for $s\in(0,1)$.
\begin{definition}\label{d2.2}
For $s\in(0,1)$ we define the space  
\[
\mathcal{H}^{s}((0,T),\mathbf{X})=
\begin{cases}
W^{s,2}((0,T),\mathbf{X}),\qquad\qquad\qquad\qquad s\in(0,1/2),\\
\{v\in W^{1/2,2}((0,T),\mathbf{X}),\quad \int_{0}^{T}\|v\|^{2}_{\mathbf{X}}\frac{dt}{t}<+\infty\},\quad s=1/2,\\
\{v\in W^{s,2}((0,T),\mathbf{X}),\quad v|_{t=0}=0\},\qquad \qquad s\in(1/2,1),\\
\end{cases}
\]
subject toh the  norms
\[
\|v\|_{\mathcal{H}^{s}((0,T),\mathbf{X})
}=
\begin{cases}
\|v\|_{W^{s,2}((0,T),\mathbf{X})},\qquad\qquad\qquad s\in(0,1),\ \ s\neq 1/2,\\
\bigg(\|v\|^{2}_{W^{1/2,2}((0,T),\mathbf{X})}+\int_{0}^{T}\|v\|^{2}_{\mathbf{X}}\frac{dt}{t}\bigg)^{1/2},\quad s=1/2.
\end{cases}
\]
\end{definition}
\noindent Setting $v(x,0)=v_{0}$ and taking into account \cite[Propositions 3 and 7]{Y2}, we arrive to the following norm equivalence
\begin{equation}\label{2.1}
C^{-1}\|v-v_{0}\|_{\mathcal{H}^{s}((0,T),\mathbf{X})}\leq \|\mathbf{D}_{t}^{s}v\|_{L_{2}((0,T),\mathbf{X})}
\leq C\|v-v_{0}\|_{\mathcal{H}^{s}((0,T),\mathbf{X})},
\end{equation}
for all $(v-v_{0})\in \mathcal{H}^{s}((0,T),\mathbf{X})$ and $s\in(0,1)$.
%%%%%%%%%%%%%%%%%%%%%%%%%%%%%%%%%%%%%%%%%%%%%%%%%%%%%%%%%%%%%%%%%%%%%%

%%%%%%%%%%%%%%%%%%%%%%%%%%%%%%%%%%%%%%%%%%%%%%%%%%%%%%%%%%%%%%%%%%%%%%

\section{General Assumptions}
\label{s3}
\noindent First, we state our general hypothesis on the structural terms of the model. To this end,
 introducing
$$
\omega_{1-\nu}(t)=\frac{t^{-\nu}}{\Gamma(1-\nu)},
$$
we define the positive values $\nu^{*}$ and $T^{*}$ such that the kernels
\[
\mathcal{N}(t;\nu,\mu_{j})=\omega_{1-\nu}(t)-\omega_{1-\mu_{j}}(t),\quad j=1,2,...,N,
\]
are nonnegative for all $t\in[0,T^{*}]$ and $0<\mu_j<\nu\leq \nu^{*}<1$.

\begin{description}
 \item[h1 (Conditions on the fractional order of the derivatives)] We assume that
\begin{equation*}
\nu\in
\begin{cases}
(0,\nu^{*})\quad\text{if}\quad N\geq 1,\\
(0,1)\quad\text{if}\quad N=0,
\end{cases}
\quad\text{and}\quad  
\nu_i,\mu_{j}\in \Big(0,\frac{\nu(2-\alpha)}{2}\Big),\quad i=1,2,...,M,\, j=1,2,...,N, 
\end{equation*}
\[
0<\mu_1<...<\mu_{N}<\nu,\quad 0<\nu_1<...<\nu_{M}<\nu,\quad \nu_{i}\neq \mu_{j}\quad \text{for all} \quad i=1,2,...,M,\, j=1,2,...,N, 
\]
%\smallskip
    \item[h2 (Ellipticity conditions)]  There are positive constants 
		 $\delta_{i},$ $i=0,1,2,3,$ such that
     \begin{align*} 
a_2(x,t)&\geq\delta_0>0,\quad \varrho_{0}(t)\geq \delta_{1}>0,\\
 \varrho_{i}(t)&\geq \delta_{2}>0,\quad \gamma_{j}(t)\geq \delta_{3}>0,\quad i=1,2,...,M,\, j=1,2,...,N, 
    \end{align*}
   for any $(x,t)\in\bar{\Omega}_{T}$ and $t\in[0,T]$.	
       %\smallskip
    \item[h3 (Regularity of the coefficients)] We require
  \begin{align*}
		a_{k},b_{k}&\in
    \C^{\alpha,\frac{\alpha\nu}{2}}(\bar{\Omega}_{T}),\quad k=0,1,2,\quad \frac{\partial a_2}{\partial x}, \frac{\partial b_2}{\partial x}\in\C(\bar{\Omega}_{T}),\\
 		\varrho_{0},\varrho_i,\gamma_{j}&\in \C^{1}([0,T]),\quad i=1,2,...,M,\, j=1,2,...,N,
\end{align*}
and
\[
\frac{\partial \varrho_{i}}{\partial t}, \frac{\partial \varrho_{0}}{\partial t}, \frac{\partial\gamma_{j}}{\partial t}\geq 0,
\]
for all $t\in[0,T]$ and $i=1,...,M,$ $j=1,2,...,N.$

Besides, in the case of  $N\geq 1$ the representation holds
\[
\varrho_0=\varrho+\sum_{i=1}^{N}\gamma_{i}
\]
with the  positive function $\varrho$ having the properties of the function $\varrho_{0}$.
%    \smallskip
		\item[h4 (Condition on the kernel)] The summable kernel $\mathcal{K}$ fulfills the estimate
$$
|\mathcal{K}(t)|\leq \frac{C}{t^\beta}, \quad \beta\in(0,1-\nu]
$$
for any $t\in[0,T]$.
%\smallskip
				     \item[h5 (Conditions on the given functions)] We require that the given functions possess  the regularity:
						
						\noindent\textbf{(i)} either
     \begin{align*}
	 u_{0}(x)&\in C^{2+\alpha}(\bar{\Omega}),
\qquad\qquad g\in\C^{\alpha,\frac{\alpha\nu}{2}}(\bar{\Omega}_{T}),\\
    \psi(x,t)&\in\C^{2+\alpha,\frac{2+\alpha}{2}\nu}(\partial\Omega_{T}),
\quad
    \psi_1(x,t)\in\C^{1+\alpha,\frac{1+\alpha}{2}\nu}(\partial\Omega_{T}),
\end{align*}

%\smallskip
\noindent \textbf{(ii)} or 
$$\psi,\psi_1,g\equiv 0,\qquad u_0\in
\begin{cases}
 W^{2,2}(\Omega)\cap\overset{0}{W}\,^{1,2}(\Omega)\quad\text{in the case of \bf{DBC}},\\
 W^{2,2}(\Omega)\qquad\qquad\quad\text{in the case of \bf{NBC}}.
\end{cases}$$
    %\smallskip
        \item[h6 (Conditions on the nonlinearity)]  
		The function $f(u)$ satisfies the one of two conditions:
		
		\noindent\textbf{(i)}
		either $f(u)$ is  the local Lipshits, i.e. for every $\rho>0$ there exists a positive constant $C_{\rho}$ such that
		\[
		|f(u_1)-f(u_2)|\leq C_{\rho}|u_1-u_2|
		\]
		for any $u_1,u_2\in[-\rho,\rho]$; and 
		
		\noindent there is a positive constant $L$ such that
		\begin{equation}\label{4.1}
		|f(u)|\leq L(1+|u|)\quad \text{for any}\quad u\in\R;
		\end{equation}
	
				\noindent\textbf{(ii)}
				or
		\begin{equation}\label{4.2}
		\begin{cases}
		f\in\C^{1}(\R),\\
		|f(u)|\leq L_1(1+|u|^{r}),\\
		uf(u)\geq -L_2+L_3|u|^{r+1},\\
		f'(u)\geq -L_{4},
		\end{cases}
		\end{equation}
		for   some nonnegative constants $r$ and $L_{i}$, $i=1,2,3,4$.
		\item[h7 (Compatibility conditions)] The following compatibility conditions hold for every
		 $x\in\partial\Omega$ at the initial time
     $t=0$, 
    \begin{equation*}
\psi(x,0)=u_{0}(x)\quad\text{and}\quad
\mathbf{D}_{t}\psi|_{t=0}=\mathcal{L}_{1}u_{0}(x)|_{t=0}-f(u_0)+g(x,0)
    \end{equation*}
		if the \textbf{DBC} \eqref{1.2} holds, and
		\[
		\frac{\partial u_0}{\partial x}(x)=\psi_{1}(x,0),
		\]
		if the \textbf{NBC} \eqref{1.2*} holds.
		       \end{description}
\begin{remark}\label{r3.1}
Thanks to Lemma 4.1 in \cite{KPV1}, the equality is true 
\[
(\mathcal{K}*\mathcal{L}_{2}u)(x,0)=0
\]
for any $u\in \C^{2+\alpha,\frac{2+\alpha}{2}\nu}(\partial\Omega_{T})$ and  any $x\in\partial\Omega$. That explains the absence of the memory term $(\mathcal{K}*\mathcal{L}_{2}u)$  in the compatibility conditions \textbf{h7}.
\end{remark}
\begin{remark}\label{r3.2}
It is worth noting that the existence of $\nu^{*}$ and $T^{*}$ in assumption \textbf{h1} is provided by  \cite[Lemma 4]{JK}. Indeed, this lemma establishes the existence of the pair $(\nu_{\gamma},T_{\gamma})$, $0<\nu_{\gamma}<1$ and $T_{\gamma}=e^{-\gamma}$ ($\gamma\approx 0.577$ is the Euler-Mascheroni constant), such that the function $\omega_{1-\nu}(t)$ is strictly increasing for all $\nu\in(0,\nu_{\gamma})$ and each $t\in[0,T_{\gamma}]$. Thus, this assertion tells us that the kernels  $\mathcal{N}(t;\nu,\mu_{j})$ are positive if $0<\mu_{j}<\nu<\nu_{\gamma}$ and $t\in[0,T_{\gamma}]$. Hence, we can select $\nu^{*}=\nu_{\gamma}$ and $T^{*}=T_{\gamma}$ in \textbf{h1}. 

Nevertheless, if $T^{*}<T_{\gamma}$, then the value $\nu^{*}$ can be chosen greater than $\nu_{\gamma}$. Unfortunately, an analytical proof of such conjecture as well as explicit  values of  $\nu^{*}$ and $T^{*}$ seem to be out of reach. This is the point of the story where the Numerics steps in. 
Indeed, let us examine case $\mu_{j}=\frac{\nu}{j+1},$ $j=1,2,3,$ and $T^{*}\leq 0.11$. We find numerically $\nu_{j}^{*}=\nu^{*}(T^{*},\mu_j)$ and $\hat{\nu}_{\gamma}^{*}=\hat{\nu}_{\gamma}(T^{*})$, which provide for all $t\in[0,T^{*}]$:
\[
\mathcal{N}\bigg(t;\nu,\frac{\nu}{j+1}\bigg)\geq 0\quad \text{for}\quad \frac{\nu}{j+1}<\nu\leq \nu_{j}^{*},\quad \text{and}
\]
\[
\omega_{1-\nu}(t)\quad\text{is strictly  increasing for } \quad \nu<\hat{\nu}^{*}_{\gamma}.
\]
Then,  setting $\nu^{*}=
\underset{j\in\{1,2,3,\}}{\min}\, \nu_{j}^{*}$, we  ensure the fulfillment of  assumption \textbf{h1} in the considered case. 

In particular, our  numeric calculations (presented with Figure \ref{fig:1} and Table \ref{tab:table1}) demonstrate that if  $T^{*}=0.1$, then  $\nu^{*}=
0.7200$, while  $\hat{\nu}^{*}_{\gamma}=0.5614$.
\begin{figure}[htp]
  %\centering
\includegraphics[scale=0.5]{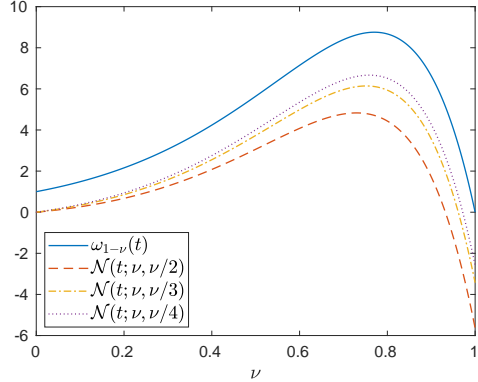}
\centerline{\text{(a)}  $T^{*}=0.01$ }
               \label{fig:a}
 %\centering
%\hfill
%\hspace{4ex}
        \includegraphics[scale=0.5]{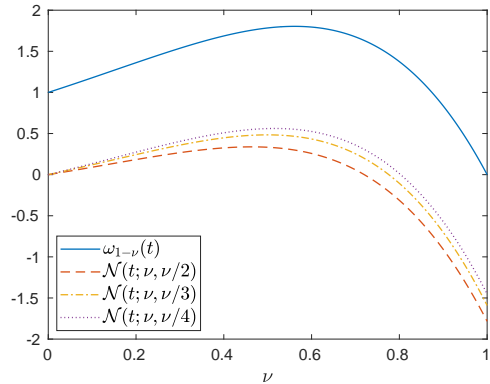}
				\centerline{\text{(b)}  $T^{*}=0.1$ }
               \label{fig:b}
								 \caption{The behavior of the function $\omega_{1-\nu}(t)$ and the kernels $\mathcal{N}(t;\nu,\frac{\nu}{j+1})$ for $j=1,2,3,$  with (a) 
								$T^{*}=0.01$, (b) $T^{*}=0.1$. 
   }
    \label{fig:1}
\end{figure}
	\begin{table}[htp]
  \begin{center}
    \caption{Values of $\hat{\nu}^{*}_{\gamma}$ and $\nu^{*}_{j}$, $j=1,2,3,$ for the corresponding value $T^{*}$}
    \label{tab:table1}
    \begin{tabular}{c|c|c|c|c|c|c|c|c|c|c|c} 
		%\hline
      %\multicolumn{4}{c|}{$|\nu-\nu_{\delta}(\lambda_{i_{j_0}},\tilde{t}_{j_{0}})|$} & $\nu$\\ % <-- Combining two cells with alignment c| and content 12.
      \hline
     $T^{*}$ & $0.01$ & $0.02$ &  $0.03$ & $0.04$ & $0.05$ & $0.06$ &  $0.07$ & $0.08$ & $0.09$ & $0.1$ & $0.11$ \\
             \hline						
    	$\hat{\nu}^{*}_{\gamma}$& 0.7703 & 0.7303 & 0.7003 & 0.6751 & 0.6527 & 0.6322 & 0.6131 & 0.5950 & 0.5779 & 0.5614 & 0.5455\\								
								\hline						
    	$\nu^{*}_{1}$& 0.9321 & 0.8984 & 0.8703 & 0.8451 & 0.8218 & 0.7999 & 0.7789 & 0.7586 & 0.7390 & 0.7200 & 0.7013\\
						\hline						
    	$\nu^{*}_{2}$& 0.9616 & 0.9361 & 0.9133 & 0.8919 & 0.8714 & 0.8515 & 0.8321 & 0.8131 & 0.7944 & 0.7759 & 0.7577\\
							\hline						
    	$\nu^{*}_{3}$& 0.9716 & 0.9505 & 0.9306 & 0.9116 & 0.8929 & 0.8746 & 0.8564 & 0.8385 & 0.8206 & 0.8023 & 0.7851\\	
    \end{tabular}
  \end{center}
\end{table}

\end{remark}
\begin{remark}\label{r3.3}
We remark that the examples of the nonlinearity $f(u)$ and the kernel $\mathcal{K}$ satisfying assumptions \textbf{h6} and \textbf{h4} are given in \cite[Example 3.3]{KPV3}, \cite[Remark 3.2]{KPV2} and \cite[(1.4)]{GPM}, respectively.
\end{remark}

%%%%%%%%%%%%%%%%%%%%%%%%%%%%%%%%%%%%%%%%%%%%%%%%%%%%%%%%%%%%%%%%%%%%%%

%%%%%%%%%%%%%%%%%%%%%%%%%%%%%%%%%%%%%%%%%%%%%%%%%%%%%%%%%%%%%%%%%%%%%%

\section{Main Results}
\label{s4}

\noindent Now we are in the position to state our first main result related with a classical solvability of \eqref{1.1}-\eqref{1.3}.

\begin{theorem}\label{t4.1}
Let $T>0$ be arbitrarily given, and let assumptions \textbf{h1-h4}, \textbf{h5 (i)} and \textbf{h7} hold. Moreover, we assume that
$f(u)$ meets requirement \textbf{h6(i)} if $N\geq 1$, while in the case of $N=0$, $f(u)$ satisfies  \textbf{h6}.
Then, equation \eqref{1.1} with the initial condition \eqref{1.3}, subjects either to the \textbf{DBC} \eqref{1.2} or to the \textbf{NBC} \eqref{1.2*}, admits a unique classical solution $u=u(x,t)$ satisfying regularity
\[
u\in\C^{2+\alpha,\frac{2+\alpha}{2}\nu}(\bar{\Omega}_{T}), \quad \mathbf{D}_{t}^{\nu_{i}}u,\, \mathbf{D}_{t}^{\mu_{j}}u\in\C^{\alpha,\frac{\alpha\nu}{2}}(\bar{\Omega}_{T}),\, i=1,2,...,M,\, j=1,2,...,N.
\] 
\end{theorem}

\begin{remark}\label{r4.1}
Actually, with inessential modifications in the proof, the very same result holds for the boundary value problem \eqref{1.1}, \eqref{1.3}
subject to either the mixed boundary conditions, e.g.
\[
u(a,t)=\psi(a,t),\quad\frac{\partial u}{\partial x}(b,t)=\psi_{1}(b,t),\qquad t\in[0,T],
\]
or the boundary conditions of the third kind.
\end{remark}
Before stating our next result, we specify the notation of a strong solution to \eqref{1.1}-\eqref{1.3}.
\begin{definition}\label{d4.1}
A function $u$ is called a strong solution to problems \eqref{1.1}-\eqref{1.3} if

\noindent$\bullet$ $u\in\C(\bar{\Omega}_{T})\cap L_2((0,T),W^{2,2}(\Omega)),$ $\mathbf{D}_{t}^{\nu} u, \mathbf{D}_{t}^{\nu_{i}} u,\, \mathbf{D}_{t}^{\mu_{j}} u\in L_{2}(\Omega_{T})$;

\noindent$\bullet$ the boundary and initial conditions \eqref{1.2}-\eqref{1.3} hold;

\noindent$\bullet$ for any fixed $T>0$ and any $\phi\in L_2(\Omega_{T})$,
\[
\int_{\Omega}\int_{0}^{T}(\mathbf{D}_{t}u-\mathcal{L}_{1}u-\mathcal{K}*\mathcal{L}_{2}u+f(u))\phi dxdt=\int_{\Omega}\int_{0}^{T} g\phi dxdt.
\]
\end{definition} 

Now we are ready to assert next results concerning to the existence of a strong solution to \eqref{1.1}-\eqref{1.3}.
\begin{theorem}\label{t4.2}
Let assumptions \textbf{h1-h4}, \textbf{h5(ii)} hold. Let the nonlinearity $f(u)$ satisfy  \textbf{h6(i)}  if $N\geq 1$, while  $f(u)$ meets requirement \textbf{h6} if $N=0$. Besides, in the case of \textbf{NBC}, we require that 
\[
\frac{\partial u_0}{\partial x}(a)=\frac{\partial u_0}{\partial x}(b)=0.
\]
Then problem \eqref{1.1}-\eqref{1.3} admits a unique strong solution $u=u(x,t)$.
\end{theorem}
\begin{remark}\label{r4.2}
In general, assumption \textbf{h3} on the coefficient $\varrho_0$ can be relaxed under additional requirements on the orders $\nu_{i}$ and $\mu_{j}$ and coefficients $\varrho_{i}$. Indeed, the results of Theorems \ref{t4.1} and \ref{t4.2} hold if instead of the condition
\[
\varrho_{0}(t)=\varrho(t)+\sum_{j=1}^{N}\gamma_{j}(t)
\]
we will require that
\begin{align*}
N&\leq M,\quad 0<\mu_1<\nu_1<\mu_2<\nu_2<...<\nu_{N-1}<\mu_{N}<\nu_{N}<...<\nu_{M}<\nu,\\
\varrho_{i}(t)&=\hat{\varrho}_{i}(t)+\gamma_{i}(t),\quad i=1,2,...,N,
\end{align*}
with $\hat{\varrho}_{i}(t)$ having the properties of the function $\varrho_{i}(t)$.

Moreover, if $M=N=0,$ 
then  Theorems \ref{t4.1} and \ref{t4.2} hold if $\varrho=\varrho(x,t)\in\C^{1}(\bar{\Omega}_{T})$.
\end{remark}
\begin{remark}\label{r4.3}
Theorems \ref{t4.1} and \ref{t4.2} hold if  $\mathbf{D}_{t}u$ in \eqref{1.4} is replaced by 
\[
\mathbf{D}_{t}u=\varrho(t)\mathbf{D}_{t}^{\nu}u+\sum_{i=1}^{M}\varrho_{i}(t)\mathbf{D}_{t}^{\nu_i}u-\sum_{j=1}^{N}\gamma_{j}(t)\mathbf{D}_{t}^{\mu_j}u.
\]
\end{remark}
\begin{remark}\label{r4.4}
It is worth noting that our assumptions on the kernel $\mathcal{K}$ include the case $\mathcal{K}=0$, telling that the multi-term subdiffusion equation:
\[
\mathbf{D}_{t}u-\mathcal{L}_1u+f(u)=g(x,t)
\]
fits in our analysis and described by the theorems above.
\end{remark}

Finally, we remark that in our work we do not consider equation \eqref{1.1} if $\nu=1$, since (thanks to presence the first order derivative in time in \eqref{1.4}) this case is examined with simpler approach.

The remaining part of the paper is devoted to the proof of Theorems \ref{t4.1} and \ref{t4.2}.
%%%%%%%%%%%%%%%%%%%%%%%%%%%%%%%%%%%%%%%%%%%%%%%%%%%%%%%

%%%%%%%%%%%%%%%%%%%%%%%%%%%%%%%%%%%%%%%%%%%%%%%%%%%%%%%
\section{Technical Results}
\label{s5}

\noindent In this section we describe some properties of fractional derivatives and integrals,  along with  several technical assertions which will be used in the course of our analysis. First, we define fractional Riemann-Liouville integrals and derivatives.

Throughout this work, for any $\theta>0$ we denote (as we did before)
\begin{equation}\label{5.1}
\omega_{\theta}(t)=\frac{t^{\theta-1}}{\Gamma(\theta)},
\end{equation} 
and define the fractional Riemann-Liouville integral and the derivative of order $\theta$, respectively, of a function $v=v(x,t)$ with respect to time $t$ as
\begin{equation*}
I_{t}^{\theta}v(x,t)=(\omega_{\theta}* v)(x,t),\quad
\partial _{t}^{\theta}v(x,t)=\frac{\partial^{\lceil\theta\rceil}}{\partial
t^{\lceil\theta\rceil}}(\omega_{\lceil\theta\rceil-\theta}*
v)(x,t),
\end{equation*}
where $\lceil\theta\rceil$ is the ceiling function of $\theta$
(i.e.\ the smallest integer greater than or equal to $\theta$).

Clearly, for $\theta\in(0,1)$ we have
\begin{equation*}
\partial _{t}^{\theta}v(x,t)=\frac{\partial}{\partial
t}(\omega_{1-\theta}* v)(x,t).
\end{equation*}
Accordingly,  the Caputo fractional derivative of the
order $\theta\in(0,1)$ to the function $v(x,t)$ can be represented as 
\begin{equation}\label{5.2}
\mathbf{D}_{t}^{\theta}v(x,t)
=\partial _{t}^{\theta}v(x,t)-\omega_{1-\theta}(t)v(x,0),
 \end{equation}
if the both derivatives exist (see  \cite[(2.4.8)]{KST}).

In the first claim, which subsumes Propositions 4.1 and 4.2 in \cite{KPV3}, we recall some important relations for the fractional derivatives and integrals.
\begin{proposition}\label{p5.1}
The following holds.
\begin{enumerate}
 \item[{\bf (i)}] Let $\theta,\theta_1\in(0,1),$ $t\in[0,T].$ Given any function $w=w(t)\in\C^{\theta_1}([0,T])$
\[
I_{t}^{\theta}\partial_{t}^{\theta}w(t)=w(t).
\]
If in addition $\theta<2\theta_1$ and $p\geq 2$ is any even integer, it is also true that
\begin{align*}
\partial_{t}^{\theta}w^{p}(t)&\leq \partial_{t}^{\theta}w^{p}(t)+(p-1)w^{p}(t)\omega_{1-\theta}(t)\\
&
\leq p w^{p-1}(t)\partial_{t}^{\theta}w(t).
\end{align*}
If  $w$ is a nonnegative function then these bounds hold for any integer odd $p$.
\item[{\bf (ii)}] For any given positive numbers $\theta_{1}$ and
    $\theta_{2}$, and any function $k\in L_1(0,T),$ the following relations are fulfilled
    \begin{equation*}
\omega_{\theta_{1}}*\omega_{\theta_{2}}=\omega_{\theta_{1}+\theta_{2}}(t),\,
1*\omega_{\theta_{1}}=\omega_{1+\theta_{1}}(t), \
\omega_{\theta_1}(t)\geq C T^{\theta_1-1},\, \omega_{\theta_{1}}*
k\leq C\omega_{\theta_{1}+1}\leq C\omega_{\theta_{1}},
    \end{equation*}
		 for any $t\in[0,T]$.
The positive constant $C$ depends only on
    $T$,
  $\theta_{1}$ and $\|k\|_{L_{1}(0,T)}$.
\end{enumerate}
\end{proposition}
Our next assertion is similar to point (i) of Proposition \ref{p5.1} stated to the convolution with the kernel
\begin{equation}\label{5.3}
\mathcal{N}_{\theta}(t)=\mathcal{N}(t;\theta_1,\theta_2)=\omega_{1-\theta_{1}}(t)-
\omega_{1-\theta_{2}}(t).
\end{equation}
\begin{corollary}\label{c5.0}
Let $\theta^{*}$ and $T^{*}\leq T$ be the numbers such that the kernel $\mathcal{N}_{\theta}(t)$ is positive for $0<\theta_2<\theta_1\leq\theta^{*}\leq 1$ and $t\in[0,T^{*}]$.
  Then for any given function $w\in\C^{\theta_3}([0,T]),$ $\theta_3\in(\theta_1,1)$ and any even integer $p\geq 2$ the inequalities are fulfilled
\begin{align}\label{5.2**}\notag
\frac{d}{dt}(\mathcal{N}_{\theta}*w^{p})(t)&\leq \frac{d}{dt}(\mathcal{N}_{\theta}*w^{p})(t)+(p-1)w^{p}(t)\mathcal{N}_{\theta}(t)\\
&
\leq p w^{p-1}(t)\frac{d}{dt}(\mathcal{N}_{\theta}*w)(t),\quad \forall t\in[0,T^{*}].
\end{align}
If additionally $w$ is nonnegative  then theses bounds hold for any integer odd $p$.
\end{corollary}
\begin{proof}
It is worth noting that this claim is a simple consequence of the following inequalities: 
\begin{align}\label{5.2*}\notag
\frac{d}{dt}(\mathcal{N}_{\theta}*w\,w_1)(t)&=w\frac{d}{dt}(\mathcal{N}_{\theta}*w_1)(t)+w_1(t)\frac{d}{dt}(\mathcal{N}_{\theta}*w)(t)
-w(t)w_1(t)\mathcal{N}_{\theta}(t)\\\notag
&
+\int_{0}^{t}[w(t)-w(s)][w_1(t)-w_1(s)]
\frac{d}{d(t-s)}\mathcal{N}_{\theta}(t-s)ds,\\
\frac{d\mathcal{N}_{\theta}}{dt}(t)&\leq 0,
\end{align}
which hold for any $w_1\in\C^{\theta_3}([0,T])$ and any $t\in[0,T^{*}]$.

Indeed, substituting 
\[
w_{1}(t)=
\begin{cases}
w\quad\text{if}\quad p=2,\\
w^{2}\quad\text{if}\quad p=3,
\end{cases}
\]
in the first equality in \eqref{5.2*},
we derive the relations:
\begin{align*}
\frac{d}{dt}(\mathcal{N}_{\theta}*w^{2})(t)&=2w\frac{d}{dt}(\mathcal{N}_{\theta}*w)(t)-\mathcal{N}_{\theta}(t)w^{2}(t)
+
\int_{0}^{t}[w(s)-w(t)]^{2}\frac{d}{d(t-s)}\mathcal{N}_{\theta}(t-s)ds,\\
\frac{d}{dt}(\mathcal{N}_{\theta}*w^{3})(t)&=w\bigg[\frac{d}{dt}(\mathcal{N}_{\theta}*w^{2})(t)
+w(t)\frac{d}{dt}(\mathcal{N}_{\theta}*w)(t)
-\mathcal{N}_{\theta}(t)w^{2}(t)\bigg]\\
&
+
\int_{0}^{t}[w(s)-w(t)]^{2}[w(s)+w(t)]\frac{d}{d(t-s)}\mathcal{N}_{\theta}(t-s)ds.
\end{align*}
Then, taking into account the positivity of the kernel $\mathcal{N}_{\theta}(t)$ for $\theta_i$ and $t$ meeting requirements of this corollary, and appealing to the second inequality in \eqref{5.2*}, we arrive at the desired estimates if $p=2,3$ (in the last case we also used the positivity of $w$).

Finally, keeping in mind the obtained inequalities and exploiting the induction, we end up with \eqref{5.2**} for $p>3$. Thus, in order to complete the verification of Corollary \ref{c5.0}, we are left to prove \eqref{5.2*}. As for the second inequality in \eqref{5.2*}, the straightforward calculations provide
\[
\frac{d\mathcal{N}_{\theta}}{dt}(t)=-\theta_1t^{-1}\mathcal{N}_{\theta}(t)-(\theta_1-\theta_{2})t^{-1}\omega_{1-\theta_2}(t)\leq 0\quad\text{for}\quad t\in[0,T^{\star}].
\]
Coming to the verification of the first equality in \eqref{5.2*}, we take advantage of the definition of a derivative and have
\[
\frac{d}{dt}(\mathcal{N}_{\theta}*ww_1)(t)=\underset{\varepsilon\to 0}{\lim}\frac{1}{\varepsilon}\bigg[
\int_{0}^{t+\varepsilon}\mathcal{N}_{\theta}(t+\varepsilon-s)w_1(s)w(s)ds
-
\int_{0}^{t}
\mathcal{N}_{\theta}(t-s)w_1(s)w(s)ds
\bigg].
\]
Then, exploiting the easily verified equality
\[
w(s)w_1(s)=[w(s)-w(t)][w_{1}(s)-w_{1}(t)]-w(t)w_1(t)+w(t)w_1(s)+w(s)w_1(t),
\]
we end up with the equality
\begin{align*}
\frac{d}{dt}(\mathcal{N}_{\theta}* w_1w)(t)=&
\underset{\varepsilon\to 0}{\lim}\frac{1}{\varepsilon}\Big[\int_{0}^{t+\varepsilon}\mathcal{N}_{\theta}(t+\varepsilon-s)[w_1(s)-w_1(t)][w(s)-w(t)]ds
\\&-\int_{0}^{t}\mathcal{N}_{\theta}(t-s)[w_1(s)-w_1(t)][w(s)-w(t)]ds\Big]\\
&+
w_{1}(t)
\underset{\varepsilon\to 0}{\lim}\frac{1}{\varepsilon}\Big[\int_{0}^{t+\varepsilon}\mathcal{N}_{\theta}(t+\varepsilon-s)w(s)ds-\int_{0}^{t}\mathcal{N}_{\theta}(t-s)w(s)ds\Big]\\
&
+
w(t)
\underset{\varepsilon\to 0}{\lim}\frac{1}{\varepsilon}\Big[\int_{0}^{t+\varepsilon}\mathcal{N}_{\theta}(t+\varepsilon-s)w_{1}(s)ds-\int_{0}^{t}\mathcal{N}_{\theta}(t-s)w_{1}(s)ds\Big]\\
&
-
w_{1}(t)w(t)
\underset{\varepsilon\to 0}{\lim}\frac{1}{\varepsilon}\Big[\int_{0}^{t+\varepsilon}\mathcal{N}_{\theta}(t+\varepsilon-s)ds-\int_{0}^{t}\mathcal{N}_{\theta}(t-s)ds\Big].
\end{align*}
Finally, taking into account the smoothness of the functions $w$ and $w_1$, we obtain the desired equality. This finishes the proof of Corollary \ref{c5.0}.
\end{proof}
\begin{corollary}\label{c5.1}
Let $0<\theta_2<\theta_1<\theta\leq 1$ and $w=w(t)\in\C([0,T])$ be a positive function, and let for any fixed $T>0$
\begin{align}\label{5.6}\notag
0&<T_{1}=T_{1}(\theta_1,\theta_2)<\min\Big\{T, \Big(\frac{\theta_1\Gamma(1+\theta_1-\theta_2)}{\theta_2}\Big)^{\frac{1}{\theta_1-\theta_2}}\Big\},\\
0&<T_{2}=T_{2}(\theta_1,\theta_2,\theta)<\min\Big\{T, \Big(\frac{\theta_1\Gamma(1+\theta-\theta_2)}{\theta_2\Gamma(1+\theta-\theta_1)}\Big)^{\frac{1}{\theta_1-\theta_2}}\Big\}.
\end{align}
 Then the inequalities hold
\begin{align*}
\theta_1 I_{t}^{1}w(t)-\theta_2I_{t}^{1+\theta_1-\theta_2}w(t)&\geq 0\quad \text{for each}\quad t\in[0,T_{1}],\\
\theta_1 I_{t}^{1+\theta-\theta_1}w(t)-\theta_2I_{t}^{1+\theta-\theta_2}w(t)&\geq 0\quad \text{for each}\quad t\in[0,T_{2}].
\end{align*}
\end{corollary}
\begin{proof}
It is apparent that the second estimate is proved with the similar arguments as the first. Thus, we restrict ourselves with the verification of the  first inequality. It is worth noting that this bound
follows from  the definition of the fractional Riemann-Liouville integral and performing straightforward calculations. Namely, appealing to \eqref{5.1}, \eqref{5.6} and assumptions on $\theta_i$,  we easily conclude that
\[
\theta_1-\theta_2\omega_{1+\theta_1-\theta_2}(t)=\theta_1-\frac{\theta_2 t^{\theta_1-\theta_2}}{\Gamma(1+\theta_1-\theta_2)}
\geq
\theta_1-\frac{\theta_2 T_{1}^{\theta_1-\theta_2}}{\Gamma(1+\theta_1-\theta_2)}\geq 0.
\]
Thus, collecting these inequalities with the positivity of $w$, we end up with the desired estimate
\[
\theta_1 I_{t}^{1}w(t)-\theta_2I_{t}^{1+\theta_1-\theta_2}w(t)=\int_{0}^{t}[\theta_1-\theta_2\omega_{1+\theta_1-\theta_2}(\tau)]w(t-\tau)d\tau\geq 0,\quad t\in[0,T_{1}],
\]
which completes the proof of this corollary.
\end{proof}
Our next result deals with the fractional differentiation of a product, i.e. $\mathbf{D}_{t}^{\theta}(w_1w_2)$. We remark that the similar result under stronger assumption on the function $w_2$ is established in \cite[Corollary 3.1]{KPSV4}.

\begin{proposition}\label{p5.1*}
Let $\theta\in(0,1),$ and $w_1\in\C^{1}([0,T]),$ $w_2\in\C([0,T]).$ 
\begin{enumerate}
\item[{\bf (i)}] If  $\mathbf{D}_{t}^{\theta}w_2$ belongs  either to $\C([0,T])$ or to
$ L_2(0,T)$, 
then  there is the equality
\begin{equation*}
\mathbf{D}_{t}^{\theta}(w_1 w_2)=w_1(t)\mathbf{D}_{t}^{\theta} w_2(t)+w_2(0)\mathbf{D}_{t}^{\theta} w_1(t)
+\frac{\theta}{\Gamma(1-\theta)}\mathfrak{I}_{\theta}(t)
\end{equation*}
with
$$
\mathfrak{I}_{\theta}(t)=\mathfrak{I}_{\theta}(t;w_1,w_2)=
\int\limits_{0}^{t}\frac{[w_1(t)-w_1(s)][w_2(s)-w_2(0)]}{(t-s)^{1+\theta}}ds.
$$
Besides,  the estimates hold
\begin{align*}
\|\mathbf{D}_{t}^{\theta}(w_1 w_2)\|_{\C([0,T])}&\leq C[\|\mathbf{D}_{t}^{\theta}w_2\|_{\C([0,T])}+\|w_2\|_{\C([0,T])}],\\
\|\mathfrak{I}_{\theta}(t)\|_{ L_2(0,T)}&\leq C\|w_2-w_2(0)\|_{L_2(0,T)},\\
\|\mathbf{D}_{t}^{\theta}(w_1 w_2)\|_{ L_2(0,T)}&\leq C[\|\mathbf{D}_{t}^{\theta}w_2\|_{ L_2(0,T)}+\|w_2-w_2(0)\|_{L_2(0,T)}],
\end{align*}
where the positive constant $C$ depends only on $T$, $\theta$ and the norm of the function $w_1$.
 \item[{\bf (ii)}] If $\partial_{t}^{\theta}w_2\in\C([0,T])$,
 then for any $\theta_{1}\geq \theta$ and all $t\in[0,T]$ the equality holds
\begin{align*}
I_{t}^{\theta_{1}}(w_1\partial_{t}^{\theta}w_2)(t)&=I_{t}^{\theta_1-\theta}(w_1w_2)(t)-w_2(0)[I_{t}^{\theta_1-\theta}w_1(t)-I_{t}^{\theta_1}(w_1\omega_{1-\theta})(t)]\\
&
-\theta I_{t}^{1+\theta_1-\theta}(\mathcal{W}(w_1)w_2)(t)
\end{align*}
with
\[
\mathcal{W}(w_1)=\mathcal{W}(w_1;t,\tau)=\int\limits_{0}^{1}\frac{\partial w_1}{\partial z}(z)ds,\quad z=st+(1-s)(\tau),\quad 0<\tau<t<T.
\]
  	\end{enumerate}
\end{proposition}
\begin{proof}
As for the representation of $\mathbf{D}_{t}^{\theta}(w_1w_2)$ stated in (i) of this claim, it follows from the definition of Caputo fractional derivative and smoothness of $w_1$ and $w_2$. Namely, we have
\begin{align*}
\mathbf{D}_{t}^{\theta}(w_1w_2)&=w_1(t)\mathbf{D}_{t}^{\theta}w_2+w_{2}(0)\mathbf{D}_{t}^{\theta}w_1
+w'_1(t)I_{t}^{1-\theta}(w_2-w_2(0))(t)\\
&
-
\frac{1}{\Gamma(1-\theta)}\int_{0}^{t}[w_2(s)-w_2(0)]\frac{\partial}{\partial t}\frac{[w_1(t)-w_1(s)]}{(t-s)^{\theta}}ds.
\end{align*}
Performing differentiation in the last integral arrives at the desired representation. 

Concerning the regularity of $\mathbf{D}_{t}^{\theta}(w_1w_2)$ and $\mathfrak{I}_{\theta}(t)$, they are simple consequence of the representation to $\mathbf{D}_{t}^{\theta}(w_1w_2)$, and of the properties of $w_1,w_2,$ and the Young inequality to convolution.

In fine, point (ii) of this proposition follows from  \eqref{5.2} and point (i) of this claim, if one takes into account \cite[Proposition 2.2]{KST} and semigroup property of the fractional Riemann-Liouville integral.
\end{proof}
We now state and prove some inequalities that will be needed to obtain a priori estimates of solutions to \eqref{1.1}-\eqref{1.3} in Section \ref{s6.5}.

First, for each positive fixed $T_3$ and $T$, $0<T_3<T$, we introduce the function
\begin{equation}\label{5.2***}
\xi=\xi(t)\in\C_{0}^{\infty}(\R_{+}),\quad \xi\in[0,1],\quad 
\xi=\begin{cases}
1,\quad t\in[0,T_3/2],\\
0,\quad t\geq 3T_3/4.
\end{cases}
\end{equation}
\begin{lemma}\label{l5.0}
Let $T,T_3$ be arbitrary fixed, $T>T_3>0$, $\theta\in(0,1)$ and $\theta_1\in(0,\theta]$. Then for any $w\in\C^{2+\alpha,\frac{2+\alpha}{2}\theta}(\bar{\Omega}_{T_3})$ and $w_1\in\C^{1}([0,T])$ the equalities hold:
\begin{enumerate}
\item[{ (i)}] $\|\xi w\|_{\C([0,T],\C^{1}(\bar{\Omega}))}\leq C\|w\|_{\C([0,T_{3}],\C^{1}(\bar{\Omega}))};$

	\smallskip
\item[{ (ii)}]$\sum_{j=1}^{2}\big[\big\|\xi\frac{\partial^{j}w}{\partial x^{j}}\big\|_{\C^{\alpha,\frac{2+\alpha-j}{2}\theta}(\bar{\Omega}_{T})}
	+\big\|\mathcal{K}*\xi\frac{\partial^{j}w}{\partial x^{j}}\big\|_{\C^{\alpha,\frac{\theta\alpha}{2}}(\bar{\Omega}_{T})}\big]
	\leq C\|w\|_{\C^{2+\alpha,\frac{2+\alpha}{2}\theta}(\bar{\Omega}_{T_3})};
	$
	
	\smallskip
	\item[{ (iii)}] $\langle\xi w\rangle^{(\theta/2)}_{t,\Omega_{T}}\leq C[\langle w\rangle^{(\theta/2)}_{t,\Omega_{T_{3}}}+\|w\|_{\C(\bar{\Omega}_{T_3})}];$
	
	\smallskip
	\item[{ (iv)}] $\|\mathfrak{I}_{\theta_1}(t;\xi,w_1w)\|_{\C^{\alpha,\frac{\theta\alpha}{2}}(\bar{\Omega}_{T})}\leq
	C\|w_1\|_{\C^{1}([0,T])}[\|w\|_{\C([0,T_3],\C^{\alpha}(\bar{\Omega}))}+\langle w\rangle^{(\theta/2)}_{t,\Omega_{T_{3}}}];$
	
	\smallskip
	\item[{ (v)}] $\|\mathbf{D}_{t}^{\theta_1}(\xi ww_1)\|_{\C^{\alpha,\frac{\theta\alpha}{2}}(\bar{\Omega}_{T})}\leq C[1+\|w_1\|_{\C^{1}([0,T])}][\|\mathbf{D}_{t}^{\theta_1}w\|_{\C^{\alpha,\frac{\theta\alpha}{2}}(\bar{\Omega}_{T_{3}})}+\|w\|_{\C^{\alpha,\frac{\theta\alpha}{2}}(\bar{\Omega}_{T_{3}})}+\|w_0\|_{\C^{\alpha}(\bar{\Omega})}];$
	
\noindent $\|\mathbf{D}_{t}^{\theta_1}(\xi ww_1)-\xi\mathbf{D}_{t}^{\theta_1}(ww_1)\|_{\C^{\alpha,\frac{\theta\alpha}{2}}(\bar{\Omega}_{T})}\leq C[1+\|w_1\|_{\C^{1}([0,T])}][\langle w\rangle_{t,\Omega_{T_{3}}}^{(\frac{\theta}{2})}+\|w\|_{\C([0,T_3],\C^{\alpha}(\bar{\Omega}))}+\|w_0\|_{\C^{\alpha}(\bar{\Omega})}],$
	where $w_0=w(x,0);$
		
	\smallskip
	\item[{ (vi)}] $\|\mathbf{D}_{t}^{\theta_1}(\xi w)\|_{L_2(0,T)}+\|\xi w\|_{L_2((0,T),W^{2,2}(\Omega))}\leq C[\|\mathbf{D}_{t}^{\theta_1}(w)\|_{L_2(0,T_3)}+\|w\|_{L_2((0,T_3),W^{2,2}(\Omega))}]$.
	\end{enumerate}
	Here the positive constant $C$ depends only on $T,\theta,\theta_1$ and the norm of the function $\xi$.
\end{lemma}
\begin{proof}
It is worth noting that the points (v) and (vi) are simple consequences of Proposition \ref{p5.1*} and the inequality in (iv) of this claim. Thus, to complete the proof of this lemma, we are left to verify estimates in (i)-(iv).

The definition of the function $\xi(t)$ and the regularity of $w(x,t)$ arrive at the inequality
\[
\|\xi w\|_{\C([0,T],\C^{2+\alpha}(\bar{\Omega}))}\leq C\|w\|_{\C([0,T_3],\C^{2+\alpha}(\bar{\Omega}))},
\]
which in turn proves the point (i) of this claim. Besides, this bound tells us that the verification of the estimate in (ii) will immediately follow from the inequality
\begin{equation}\label{10.1}
\sum_{j=1}^{2}\bigg\langle\xi\frac{\partial^{j}w}{\partial x^{j}}\bigg\rangle_{t,\Omega_{T}}^{(\frac{2+\alpha-j}{2}\theta)}
\leq C 
\sum_{j=1}^{2}\bigg[\bigg\langle\frac{\partial^{j}w}{\partial x^{j}}\bigg\rangle_{t,\Omega_{T_{3}}}^{(\frac{2+\alpha-j}{2}\theta)}+
\bigg\|\frac{\partial^{j}w}{\partial x^{j}}\bigg\|_{\C(\bar{\Omega}_{T_{3}})}
\bigg].
\end{equation}
For simplicity consideration, we first assume $0<t_1<t_2<T$, and set $\Delta t=t_2-t_1$. Then, we discuss three options to the arrangement of  $t_1$ and $t_2$:
\begin{align}\label{10.5}\notag
&\text{either}\quad 0<t_1<t_2\leq 3T_{3}/4,\\\notag
&\text{or}\quad 0<t_1\leq 3T_{3}/4<t_2\leq T,\\
&\text{or}\quad 0<3T_{3}/4<t_1<t_2\leq T.
\end{align}
In the first case, we have
\[
\bigg|\xi(t_2)\frac{\partial^{j}w}{\partial x^{j}}(x,t_2)-\xi(t_1)\frac{\partial^{j}w}{\partial x^{j}}(x,t_1)\bigg|\leq|\xi(t_2)-\xi(t_1)|\bigg\|\frac{\partial^{j}w}{\partial x^{j}}\bigg\|_{\C(\bar{\Omega}_{T_3})}+
\xi(t_1)\bigg|\frac{\partial^{j}w}{\partial x^{j}}(x,t_2)-\frac{\partial^{j}w}{\partial x^{j}}(x,t_1)\bigg|.
\]
Collecting this estimate with smoothness of the functions $w$ and $\xi$, we immediately obtain  \eqref{10.1}.

Coming to the second case, i.e. $t_1\leq \frac{3T_3}{4}<t_2$,  the easily verified inequalities 
\begin{align*}
\Delta t&>\frac{3T_3}{4}-t_1,\\
\xi(t_2)\frac{\partial^{j}w}{\partial x^{j}}(x,t_2)-\xi(t_1)\frac{\partial^{j}w}{\partial x^{j}}(x,t_1)&=
\xi\bigg(\frac{3T_3}{4}\bigg)\frac{\partial^{j}w}{\partial x^{j}}\bigg(x,\frac{3T_3}{4}\bigg)-\xi(t_1)\frac{\partial^{j}w}{\partial x^{j}}(x,t_1),
\end{align*}
provide the desired bound \eqref{10.1}.

In the last case in \eqref{10.5}, thanks to $\xi(t)=0$ for $t>\frac{3T_3}{4}$,  we have
\[
\xi(t_2)\frac{\partial^{j}w}{\partial x^{j}}(x,t_2)-\xi(t_1)\frac{\partial^{j}w}{\partial x^{j}}(x,t_1)=0,
\]
which means that \eqref{10.1} holds.

As a result, gathering  all estimates, we complete the proof of  \eqref{10.1} and, besides,
\begin{equation*}\label{10.2}
\sum_{j=1}^{2}\bigg\|\xi\frac{\partial^{j}w}{\partial x^{j}}\bigg\|_{\C^{\alpha,\frac{2+\alpha-j}{2}\theta}(\bar{\Omega}_{T})}
\leq C\|w\|_{\C^{2+\alpha,\frac{2+\alpha}{2}\theta}(\bar{\Omega}_{T_{3}})}.
\end{equation*}
Therefore, in order  to complete the verification of the  estimate in (ii), we first take advantage of the representation
\[
\mathcal{K}*\xi\frac{\partial^{j}w}{\partial x^{j}}=\begin{cases}
\int_{0}^{t}\mathcal{K}(t-s)\xi(s)\frac{\partial^{j}w}{\partial x^{j}}(x,s)ds,\qquad 0<t<\frac{3T_3}{4},\\
\,\\
\int_{0}^{\frac{3T_3}{4}}\mathcal{K}(t-s)\xi(s)\frac{\partial^{j}w}{\partial x^{j}}(x,s)ds,\qquad t\geq\frac{3T_3}{4},
\end{cases}
\]
and  \cite[Lemma 4.1]{KPV1}, then, performing standard calculations, we conclude that
\[
\bigg\|\mathcal{K}*\xi\frac{\partial^{j}w}{\partial x^{j}}\bigg\|_{\C^{\alpha,\frac{\alpha\theta}{2}}(\bar{\Omega}_{T})}
\leq C\|\mathcal{K}\|_{L_1(0,T)}\bigg\|\frac{\partial^{j}w}{\partial x^{j}}\bigg\|_{\C^{\alpha,\frac{\alpha\theta}{2}}(\bar{\Omega}_{T_{3}})}.
\]
Thus, the proof of the estimate in (ii) is finished.

Concerning the inequalities in (iii), they are obtained with the  arguments leading to \eqref{10.1}.

At this point we examine the bound in (iv). Here we restrict ourselves with the consideration of the case $\theta_1=\theta$. Another case is verified with the similar manner.

In light of the definition  of $\xi(t)$ (see \eqref{5.2***}) and $\mathfrak{I}_{\theta}(t)$, we deduce
\begin{align}\label{10.3}\notag
\mathfrak{I}_{\theta}(t)&=\begin{cases}
0,\qquad\qquad t\leq T_3/2,\\
\,\\
\int_{0}^{t}\frac{[\xi(t)-\xi(\tau)][w(x,\tau)w_1(\tau)-w(x,0)w_1(0)]d\tau}{(t-\tau)^{1+\theta}},\quad t\in(T_3/2, 3T_3/4),\\
\,\\
\int_{0}^{3T_3/4}\frac{[\xi(t)-\xi(\tau)][w(x,\tau)w_1(\tau)-w(x,0)w_1(0)]d\tau}{(t-\tau)^{1+\theta}},\quad t\geq 3T_3/4,
\end{cases}\\
&
=\begin{cases}
0,\qquad\qquad t\leq T_3/2,\\
\,\\
\int_{0}^{3T_3/4}\frac{\mathcal{W}(\xi;t,\tau)[w(x,\tau)-w(x,0)]w_1(\tau)d\tau}{(t-\tau)^{\theta}}
+\int_{0}^{3T_3/4}\frac{\mathcal{W}(\xi;t,\tau)[w_1(\tau)-w_{1}(0)]w(x,0)d\tau}{(t-\tau)^{\theta}}
,\quad t\in(T_3/2, 3T_3/4),\\
\, \\
\int_{0}^{t}\frac{\mathcal{W}(\xi;t,\tau)[w(x,\tau)-w(x,0)]w_1(\tau)d\tau}{(t-\tau)^{\theta}}
+\int_{0}^{t}\frac{\mathcal{W}(\xi;t,\tau)[w_1(\tau)-w_{1}(0)]w(x,0)d\tau}{(t-\tau)^{\theta}}
,\quad t\geq 3T_3/4,
\end{cases}
\end{align}
where the function $\mathcal{W}(\cdot)$ is defined in point (ii) of Proposition \ref{p5.1*}.

Taking into account these equalities and performing standard technical  calculations, we end up with the estimate
\begin{equation}\label{10.4}
\|\mathfrak{I}_{\theta}(t)\|_{\C([0,T],\C^{\alpha}(\bar{\Omega}))}\leq C\|\xi\|_{\C^{1}([0,T])}\|w_1\|_{\C^{1}([0,T])}
[\|w\|_{\C([0,T_3],\C^{\alpha}(\bar{\Omega}))}+\langle w\rangle_{t,\Omega_{T_3}}^{(\theta/2)}].
\end{equation}
As for the H\"{o}lder regularity of $\mathfrak{I}_{\theta}(t)$ with respect to time, we first assume that
\begin{equation}\label{10.6}
\Delta t=t_2-t_1<T_3/8,
\end{equation}
otherwise the bound of $\langle\mathfrak{I}_{\theta}(t)\rangle_{t,\Omega_{T}}^{(\alpha\theta/2)}$ follows from \eqref{10.4}.

Next we consider again  \eqref{10.5} to the location  of $t_1,t_2$. If $t_1,t_2\in[0,3T_3/4]$, then   using \eqref{10.3}, we can write
\[
\mathfrak{I}_{\theta}(t_2)-\mathfrak{I}_{\theta}(t_1)\equiv\sum_{i=1}^{3}J_{i},
\] 
where we set
\begin{align*}
J_1&=\int_{0}^{t_1}\frac{\mathcal{W}(\xi;t_2,t_2-\tau)[w(x,t_2-\tau)w_1(t_2-\tau)-w(x,t_1-\tau)w_1(t_1-\tau)]d\tau}{\tau^{\theta}},\\
J_2&=\int_{0}^{t_1}\frac{[\mathcal{W}(\xi;t_2,t_2-\tau)-\mathcal{W}(\xi;t_1,t_1-\tau)][w(x,t_1-\tau)w_1(t_1-\tau)-w(x,0)w_1(0)]d\tau}{\tau^{\theta}},\\
J_3&=\int_{t_1}^{t_2}\frac{\mathcal{W}(\xi;t_2,t_2-\tau)[w(x,t_2-\tau)w_1(t_2-\tau)-w(x,0)w_1(0)]d\tau}{\tau^{\theta}}.
\end{align*}
According to the properties of  $\xi,w$ and $w_1$, we immediately deduce
\[
|J_1|+|J_2|\leq C\Delta t^{\theta/2}\|\xi\|_{\C^{2}([0,T])}\|w_1\|_{\C^{1}([0,T])}
[\|w\|_{\C(\bar{\Omega}_{T_3})}+\langle w\rangle_{t,\Omega_{T_{3}}}^{(\theta/2)}].
\]
As for $J_3$, we have
\begin{align*}
|J_3|&\leq C \|\xi\|_{\C^{1}([0,T])}\|w_1\|_{\C^{1}([0,T])}
[\|w\|_{\C(\bar{\Omega}_{T_3})}+\langle w\rangle_{t,\Omega_{T_{3}}}^{(\theta/2)}]
\int_{t_1}^{t_2}\tau^{-\theta}(t_2-\tau)^{\theta/2}[1+(t_2-\tau)^{1-\theta/2}]d\tau\\
&
\leq 
C\Delta t^{\theta/2}\|\xi\|_{\C^{2}([0,T])}\|w_1\|_{\C^{1}([0,T])}
[\|w\|_{\C(\bar{\Omega}_{T_3})}+\langle w\rangle_{t,\Omega_{T_{3}}}^{(\theta/2)}].
\end{align*}
It is worth noting that the positive constant $C$ depends only on $T$ and $\theta$. Collecting these estimates arrives at the bound
\begin{equation}\label{10.7}
\langle\mathfrak{I}_{\theta}(t)\rangle_{t,\Omega_{T}}^{(\theta/2)}\leq C \|\xi\|_{\C^{2}([0,T])}\|w_1\|_{\C^{1}([0,T])}
[\|w\|_{\C(\bar{\Omega}_{T_3})}+\langle w\rangle_{t,\Omega_{T_{3}}}^{(\theta/2)}].
\end{equation}
Concerning the case of $t_{1}\leq\frac{3T_3}{4}<t_2<T$, assumption \eqref{10.6} tells us that
\[
t_2<\frac{7T_3}{8}<T_3.
\]
Hence, recasting the arguments leading to \eqref{10.7}, we obtain the desired estimate in the case of the second option in \eqref{10.3}. Finally, analyzing the case $\frac{3T_3}{4}\leq t_1<t_2\leq T$, two possibilities occur:
\begin{enumerate}
\item[{ (i)}] either $t_1-\frac{3T_3}{4}\leq\frac{T_3}{8};$
\item[{ (ii)}] or  $t_1-\frac{3T_3}{4}>\frac{T_3}{8}.$
\end{enumerate}
It is apparent that the option (i) is studied with the similar arguments leading to \eqref{10.7}. As for the case (ii), exploiting \eqref{10.3}, we have
\[
\mathfrak{I}_{\theta}(t_2)-\mathfrak{I}_{\theta}(t_1)\equiv i_1+i_2,
\]
where
\begin{align*}
i_1&=\int_{0}^{3T_3/4}[w(x,\tau)w_1(\tau)-w(x,0)w_1(0)]\frac{\mathcal{W}(\xi;t_2,\tau)-\mathcal{W}(\xi;t_1,\tau)}{(t_2-\tau)^{\theta}}d\tau,\\
i_2&=\int_{0}^{3T_3/4}[w(x,\tau)w_1(\tau)-w(x,0)w_1(0)]\mathcal{W}(\xi;t_2,\tau)[(t_2-\tau)^{-\theta}-(t_1-\tau)^{-\theta}]d\tau.
\end{align*}
According to  the regularity of $\xi(t)$, we arrive at
\[
|i_1|\leq C\Delta t\|w_1\|_{\C^{1}([0,T])}\|\xi\|_{\C^{2}(\bar{\Omega}_{T})}[\|w\|_{\C(\bar{\Omega}_{T_3})}+\langle w\rangle_{t,\Omega_{T_{3}}}^{(\theta/2)}]
\int_{0}^{3T_3/4}[\tau+\tau^{\theta/2}](t_2-\tau)^{-\theta}d\tau.
\]
To estimate the term $i_2$, we use the mean value theorem and the fact that $t_1-\frac{3T_3}{4}>\frac{T_3}{8}$. Thus, we achieve
\[
|i_2|\leq C\frac{\Delta t}{T_3^{2\theta}}\|\xi\|_{\C^{1}([0,T])}\|w_1\|_{\C^{1}([0,T])}[\|w\|_{\C(\bar{\Omega}_{T_3})}+\langle w\rangle_{t,\Omega_{T_{3}}}^{(\theta/2)}].
\]
As a result, gathering this inequality with the estimate of $i_1$, we end up with bound \eqref{10.7} in the third case in \eqref{10.5}. Finally, \eqref{10.4} and \eqref{10.7} completes the proof of point (iv) in this lemma.
\end{proof}

Next, we state and prove inequalities which are generalized the bounds in \cite[Lemma 4.2]{KPV2} and will be used later in this art. 
\begin{lemma}\label{l5.1}
Let $w_2=w_{2}(x,t)\in\C^{2+\alpha,\frac{2+\alpha}{2}\theta}(\bar{\Omega}_{T}),$ $\theta\in(0,1)$, and let $w_1=w_1(t)\in\C^{1}([0,T])$ be a positive function. We assume that
\[
\text{either}\quad w_2|_{\partial\Omega_{T}}=0\quad\text{or}\quad \frac{\partial w_2}{\partial x}|_{\partial\Omega_{T}}=0.
\]
Then for any $0<\theta_2<\theta_1\leq\theta$ and any integer even $p\geq 2$ the inequalities hold:
\begin{enumerate}
 \item[{\bf (i)}] \begin{align*}
	&-pI_{t}^{\theta}\bigg(\int_{\Omega}\partial_{t}^{\theta_1}(w_1w_2)\frac{\partial}{\partial x}\bigg(\frac{\partial w_2}{\partial x_2}\bigg)^{p-1}dx\bigg)(t)\\
	&
	\geq 
	(p-1)I_{t}^{\theta}\Big(w_1\omega_{1-\theta_1}\int_{\Omega}\bigg(\frac{\partial w_2}{\partial x_2}\bigg)^{p}dx\Big)(t)
	+I_{t}^{\theta-\theta_1}\bigg(w_1\int_{\Omega}\bigg(\frac{\partial w_2}{\partial x_2}\bigg)^{p}dx\bigg)(t)\\
	&
	-w_{1}^{p}(0)\int_{\Omega}\bigg(\frac{\partial w_2}{\partial x_2}(x,0)\bigg)^{p}dx[I_{t}^{\theta-\theta_1}(w_1^{1-p})(t)
	-I_{t}^{\theta}(w_{1}^{1-p}\omega_{1-\theta_{1}})(t)]
	\\
	&
	-
	\theta_1I_{t}^{1-\theta_1+\theta}\bigg(w_1^{p}\mathcal{W}(w_1^{1-p})
	\int_{\Omega}\bigg(\frac{\partial w_2}{\partial x_2}\bigg)^{p}dx\bigg)(t) \quad \text{for}\quad \forall t\in[0,T],
	\end{align*}
 \item[{\bf (ii)}]
\begin{align*}
	&-pI_{t}^{\theta}\bigg(\int_{\Omega}\frac{\partial}{\partial t}(\mathcal{N}_{\theta}*w_1w_2)\frac{\partial}{\partial x}\bigg(\frac{\partial w_2}{\partial x_2}\bigg)^{p-1}dx\bigg)(t)
\\
	&\geq 
	(p-1)I_{t}^{\theta}\Big(w_1\mathcal{N}_{\theta}\int_{\Omega}\bigg(\frac{\partial w_2}{\partial x_2}\bigg)^{p}dx\Big)(t)
	+I_{t}^{\theta-\theta_1}\bigg(w_1\int_{\Omega}\bigg(\frac{\partial w_2}{\partial x_2}\bigg)^{p}dx\bigg)(t)\\
	&
		-w_{1}^{p}(0)\int_{\Omega}\bigg(\frac{\partial w_2}{\partial x_2}(x,0)\bigg)^{p}dxI_{t}^{\theta-\theta_1}\bigg(w_1^{1-p}+I_{t}^{\theta_1}(\omega_{1-\theta_1}w_1^{1-p})\bigg)(t)\\
	&
	-
	I_{t}^{\theta-\theta_2}\bigg(w_1\int_{\Omega}\bigg(\frac{\partial w_2}{\partial x_2}\bigg)^{p}dx\bigg)(t)
	-\theta_1I_{t}^{1-\theta_1+\theta}\bigg(w_1^{p}\mathcal{W}(w_1^{1-p})\int_{\Omega}\bigg(\frac{\partial w_2}{\partial x_2}\bigg)^{p}dx\bigg)(t)\\
	&
	+
	\theta_2 I_{t}^{1-\theta_2+\theta}\bigg(w_1^{p}\mathcal{W}(w_1^{1-p})\int_{\Omega}\bigg(\frac{\partial w_2}{\partial x_2}\bigg)^{p}dx\bigg)(t)\quad \text{for}\quad \forall t\in[0,T^{*}],
		\end{align*}
where $\mathcal{N}_{\theta}$ is given with \eqref{5.3} and $\mathcal{W}(\cdot)$ is defined in (ii) of Proposition \ref{p5.1*}.
\end{enumerate}
\end{lemma}
\begin{proof}
First we consider the case of homogeneous Dirichlet boundary condition.
We preliminarily observe that the estimate at the point (ii) is the same as the one for the bound at (i) (where instead of Proposition \ref{p5.1}, one should use Corollary \ref{c5.0})  of this claim. For these reason, we are left to tackle the inequality in (i).

To this end, similar to the proof of Lemma 4.2 \cite{KPV2}, we first construct the mollification of the function $w_2$. For any positive $\mathfrak{d}<|b-a|$, we introduce a cut function $\eta\in\C_{0}^{\infty}(\R)$ taking values in $[0,1]$:
\[
\begin{cases}
\eta=1,\quad\text{if}\quad x\in(a-\frac{\mathfrak{d}}{4},b+\frac{\mathfrak{d}}{4}),\\
\eta=0,\quad\text{if}\quad x\in\R\backslash(a-\frac{\mathfrak{d}}{2},b+\frac{\mathfrak{d}}{2}),
\end{cases}
\] 
and then define the even extension $W_{\mathfrak{d}}(x,t)$ of the function $w_2(x,t)$ in the segment $(a-\mathfrak{d},b+\mathfrak{d})$ as
\[
W_{\mathfrak{d}}(x,t)=
\begin{cases}
-w_2(2a-x,t),\quad\text{if}\quad x\in(a-\mathfrak{d},a),\\
w_2(x,t),\qquad\quad\quad\text{if}\quad x\in[a,b],\\
-w_2(2b-x,t),\quad\text{if}\quad x\in(b,b+\mathfrak{d}).
\end{cases}
\]
Then we build the zero extension of $W_{\mathfrak{d}}$
outside the segment $(a-\mathfrak{d},b+\mathfrak{d})$ as
\[
W_{\eta}=W_{\eta}(x,t)=\eta(x)W_{\mathfrak{d}}(x,t).
\]
It is easily verify that $W_{\eta}\in\C^{2+\alpha,\frac{2+\alpha}{2}\theta}(\bar{\R}_{T}),$ and
\[
\|W_{\eta}\|_{\C^{2+\alpha,\frac{2+\alpha}{2}\theta}(\bar{\R}_{T})}\leq C\|w_{2}\|_{\C^{2+\alpha,\frac{2+\alpha}{2}\theta}(\bar{\Omega}_{T})}.
\]
Setting the mollifier $J_{\varepsilon}(x)$ satisfying the properties:
\[
J_{\varepsilon}(x)\in\C_{0}^{\infty}(\R),\quad J_{\varepsilon}(x)=0\quad\text{if}\quad |x|\geq\varepsilon,
\quad
\int_{-\infty}^{+\infty}J_{\varepsilon}(x)dx=1,
\]
we define the molification of the function $W_{\eta}$ as
\[
W_{\eta,\varepsilon}=\int_{\R} J_{\varepsilon}(x-y)W_{\eta}(y,t)dy.
\]
Clearly,
\[
W_{\eta,\varepsilon}\in\C([0,T],\C^{\infty}(\R)),\quad \mathbf{D}_{t}^{\theta_i}W_{\eta,\varepsilon}\in\C([0,T],\C^{\infty}(\R)),\, 0<\theta_{i}\leq\theta,\, i=1,2,
\]
and
\[
W_{\eta,\varepsilon}=0\qquad\text{on}\quad\partial\Omega_{T}.
\]
The last equality tells us that
\[
\partial_{t}^{\theta_2}W_{\eta,\varepsilon}=\partial_{t}^{\theta_1}W_{\eta,\varepsilon}=\partial_{t}^{\theta}W_{\eta,\varepsilon}=0\qquad
\text{for}\quad (x,t)\in\partial\Omega_{T}.
\]
Following standard approximation arguments (see e.g. \cite[Chapter 1]{AF}), we have the uniform convergences on $\bar{\Omega}_{T}$ as $\varepsilon\to 0$
\begin{align*}
\frac{\partial^{i}W_{\eta,\varepsilon}}{\partial x^{i}}&\to \frac{\partial^{i}w_{2}}{\partial x^{i}},\quad i=0,1,2,\quad\text{and}\\
\mathbf{D}_{t}^{\theta_i}W_{\eta,\varepsilon}&\to\mathbf{D}_{t}^{\theta_i}w_2,\quad 0<\theta_2<\theta_1\leq\theta<1,\\
\frac{\partial}{\partial x}\bigg(\frac{\partial W_{\eta,\varepsilon}}{\partial x}\bigg)^{p-1}&\to 
\frac{\partial}{\partial x}\bigg(\frac{\partial w_{2}}{\partial x}\bigg)^{p-1},\quad p\geq 2.
\end{align*}
Moreover, exploiting \cite[Corollary 3.1]{KPSV4} and standard technical calculations, we arrive at
\begin{align*}
\partial_{t}^{\theta_i}(w_1W_{\eta,\varepsilon})&\to \partial_{t}^{\theta_i}(w_1w_{2}),\quad i=1,2,\\
\mathbf{D}_{t}^{\theta_i}(w_1W_{\eta,\varepsilon})&\to \mathbf{D}_{t}^{\theta_i}(w_1w_{2}).
\end{align*}

Now, we begin to prove the first inequality of this lemma for $W_{\eta,\varepsilon}$. Namely, integration by parts together with Propositions \ref{p5.1} and \ref{p5.1*} yield
\begin{align*}
&-pI_{t}^{\theta}\bigg(\int_{\Omega}\partial_{t}^{\theta_1}(w_1W_{\eta,\varepsilon})\frac{\partial}{\partial x}\bigg(\frac{\partial W_{\eta,\varepsilon}}{\partial x}\bigg)^{p-1}dx\bigg)(t)=
pI_{t}^{\theta}\bigg(w_{1}^{1-p}\int_{\Omega}\partial_{t}^{\theta_1}\frac{\partial}{\partial x}(w_1W_{\eta,\varepsilon})\bigg(\frac{\partial w_1 W_{\eta,\varepsilon}}{\partial x}\bigg)^{p-1}dx\bigg)(t)\\
&
\geq
I_{t}^{\theta}\bigg(w_{1}^{1-p}\int_{\Omega}\partial_{t}^{\theta_1}\bigg(\frac{\partial}{\partial x}(w_1W_{\eta,\varepsilon})\bigg)^{p}dx\bigg)(t)
+
(p-1)I_{t}^{\theta}\bigg(w_1\omega_{1-\theta_1}\int_{\Omega}\bigg(\frac{\partial W_{\eta,\varepsilon}}{\partial x}\bigg)^{p}dx\bigg)(t)\\
&
=
(p-1)I_{t}^{\theta}\bigg(w_1\omega_{1-\theta_1}\int_{\Omega}\bigg(\frac{\partial W_{\eta,\varepsilon}}{\partial x}\bigg)^{p}dx\bigg)(t)
+
I_{t}^{\theta-\theta_1}\bigg(w_1\int_{\Omega}\bigg(\frac{\partial W_{\eta,\varepsilon}}{\partial x}\bigg)^{p}dx\bigg)(t)\\
&
-
w_1^{p}(0)\int_{\Omega}\bigg(\frac{\partial W_{\eta,\varepsilon}}{\partial x}(x,0)\bigg)^{p}dxI_{t}^{\theta-\theta_1}(w_{1}^{1-p})(t)
-\theta_1I_{t}^{1-\theta_1+\theta}\bigg(\mathcal{W}(w_1^{1-p})w_1^{p}\int_{\Omega}\bigg(\frac{\partial W_{\eta,\varepsilon}}{\partial x}\bigg)^{p}dx\bigg)(t)\\
&
+w_1^{p}(0)\int_{\Omega}\bigg(\frac{\partial W_{\eta,\varepsilon}}{\partial x}(x,0)\bigg)^{p}dx I_{t}^{\theta}(w_1^{1-p}\omega_{1-\theta_1})(t)
\equiv
\mathcal{I}\bigg(t,w_1,\frac{\partial W_{\eta,\varepsilon}}{\partial x}\bigg).
\end{align*}
Finally, taking into account of the positivity of $w_1$ and even $p$ to control the last term in the right-hand side of the last inequality, we arrive at the desired bound to the function $W_{\eta,\varepsilon}$. The conclusion can be easily drawn from the uniform convergences
\[
I_{t}^{\theta}\bigg(\int_{\Omega}\partial_{t}^{\theta_1}(w_1W_{\eta,\varepsilon})\frac{\partial}{\partial x}\bigg(\frac{\partial W_{\eta,\varepsilon}}{\partial x}\bigg)^{p-1}dx\bigg)(t)
\to
I_{t}^{\theta}\bigg(\int_{\Omega}\partial_{t}^{\theta_1}(w_1w_{2})\frac{\partial}{\partial x}\bigg(\frac{\partial w_{2}}{\partial x}\bigg)^{p-1}dx\bigg)(t),
\]
and 
\[
\mathcal{I}\bigg(t,w_1,\frac{\partial W_{\eta,\varepsilon}}{\partial x}\bigg)\to 
\mathcal{I}\bigg(t,w_1,\frac{\partial w_{2}}{\partial x}\bigg).
\]

The case of $\frac{\partial u}{\partial x}=0$ on $\partial\Omega_{T}$ is similar and left to the reader.
This completes the proof of this lemma.
\end{proof}
For the reader's convenience, we now recall the global classical solvability of the linear version of problem \eqref{1.1}-\eqref{1.3}. This result, stated as a lemma, is proved in our previous work \cite[Theorem 4.1, Remark 4.4]{PSV6}, and will be a key point in our analysis in Sections \ref{s6}-\ref{s7}.
\begin{lemma}\label{l5.2}
Let $\partial\Omega\in\C^{2+\alpha}$, $f(u)\equiv 0$, and let $\nu\in(0,1],$  while $\nu_{i},\mu_{j}$ meet requirement \textbf{h1}. For any fixed $T>0$, under assumptions \textbf{h2}-\textbf{h4}, and  \textbf{h5(i)}, \textbf{h7},  then the conclusions of Theorem \ref{t4.1} hold.
Besides, this solution fulfills the estimate
\begin{align*}
&\|u\|_{\C^{2+\alpha,\frac{2+\alpha}{2}\nu}(\bar{\Omega}_{T})}+\sum_{i=1}^{M}\|\mathbf{D}_{t}^{\nu_i}u\|_{\C^{\alpha,\frac{\alpha}{2}\nu}(\bar{\Omega}_{T})}+\sum_{j=1}^{N}\|\mathbf{D}_{t}^{\mu_j}u\|_{\C^{\alpha,\frac{\alpha}{2}\nu}(\bar{\Omega}_{T})}\\
&
\leq C[\|u_0\|_{\C^{2+\alpha}(\bar{\Omega})}+\|g\|_{\C^{\alpha,\frac{\alpha}{2}\nu}(\bar{\Omega}_{T})}
+\|\psi\|_{\C^{2+\alpha,\frac{2+\alpha}{2}\nu}(\partial\Omega_{T})}]
\end{align*}
within the \textbf{DBC} \eqref{1.2}, or
\begin{align*}
&\|u\|_{\C^{2+\alpha,\frac{2+\alpha}{2}\nu}(\bar{\Omega}_{T})}+\sum_{i=1}^{M}\|\mathbf{D}_{t}^{\nu_i}u\|_{\C^{\alpha,\frac{\alpha}{2}\nu}(\bar{\Omega}_{T})}+\sum_{j=1}^{N}\|\mathbf{D}_{t}^{\mu_j}u\|_{\C^{\alpha,\frac{\alpha}{2}\nu}(\bar{\Omega}_{T})}\\
&
\leq C[\|u_0\|_{\C^{2+\alpha}(\bar{\Omega})}+\|g\|_{\C^{\alpha,\frac{\alpha}{2}\nu}(\bar{\Omega}_{T})}
+\|\psi_1\|_{\C^{1+\alpha,\frac{1+\alpha}{2}\nu}(\partial\Omega_{T})}]
\end{align*}
within \textbf{NBC} \eqref{1.2*}. 
The generic constant $C$ is independent of the right-hand sides of \eqref{1.1}-\eqref{1.3}.
\end{lemma}

We conclude this preliminary section with the inequalities for every $v\in W^{1,2}(\Omega)$ that will play a key role in the proof of Theorem \ref{t4.1} in Subsection \ref{s6.2}: 
\[
\|v\|_{L_{\infty}(\Omega)}\leq C\|v\|^{1/3}_{W^{1,2}(\Omega)}\|v\|^{2/3}_{L_{1}(\Omega)},
\]
and
\begin{equation}\label{i.1}
\|v\|_{L_{\infty}(\Omega)}\leq \varepsilon \|v\|_{W^{1,2}(\Omega)}+C\varepsilon^{-1/2}\|v\|_{L_{1}(\Omega)}.
\end{equation}
It is worth noting that the first inequality is the bound (2.19) in \cite{LU}, while estimate \eqref{i.1} is a simple consequence of the first and Young inequality.

%%%%%%%%%%%%%%%%%%%%%%%%%%%%%%%%%%%%%%%%%%%%%%%%%%%%%%%%%%%%%%%%%%%%%%
 
%%%%%%%%%%%%%%%%%%%%%%%%%%%%%%%%%%%%%%%%%%%%%%%%%%%%%%%%%%%%%%%%%%%%%%

\section{A Priori Estimates}
\label{s6}

\noindent In this section, we provide a priori estimates for the classical  solutions to the following family of problems for $\lambda\in[0,1],$
\begin{equation}\label{6.1}
\begin{cases}
\mathbf{D}_{t}u-\mathcal{L}_{1}u-\mathcal{K}*\mathcal{L}_2u+\lambda f(u)=g(x,t)\quad\text{in}\quad \Omega_{T},\\
u(x,0)=u_{0}(x)\qquad \text{in}\qquad \bar{\Omega},\\
u(x,t)=0,\qquad\text{on}\qquad \partial\Omega_{T}.
\end{cases}
\end{equation}
These estimates are the crucial point in the proof of Theorems \ref{t4.1} and \ref{t4.2}.
\begin{lemma}\label{l6.1}
Let the assumptions of Theorem \ref{t4.1} hold with $\psi=\psi_1\equiv 0$. 

\noindent If $u\in\C^{2+\alpha,\frac{2+\alpha}{2}\nu}(\bar{\Omega}_{T})$ is a classical solution to \eqref{6.1}. Then for any $\lambda\in[0,1],$ the following inequalities are fulfilled
\begin{equation}\label{6.2}
\|u\|_{\C([0,T],\C^{1}(\bar{\Omega}))}+\langle u\rangle_{t,\Omega_{T}}^{(\nu/2)}
\leq C(1+\|u_0\|_{W^{2,2}(\Omega)}+[\underset{t\in[0,T]}{\sup}I_{t}^{\nu}\|g\|^{2}_{L_{2}(\Omega)}]^{\frac{1}{2}}),
\end{equation}
\begin{equation}\label{6.3}
\|u\|_{\C^{2+\alpha,\frac{2+\alpha}{2}\nu}(\bar{\Omega}_{T})}+\sum_{i=1}^{M}\|\mathbf{D}_{t}^{\nu_i}u\|_{\C^{\alpha,\frac{\alpha}{2}\nu}(\bar{\Omega}_{T})}+\sum_{j=1}^{N}\|\mathbf{D}_{t}^{\mu_j}u\|_{\C^{\alpha,\frac{\alpha}{2}\nu}(\bar{\Omega}_{T})}
\leq C[1+\|u_0\|_{\C^{2+\alpha}(\bar{\Omega})}+\|g\|_{\C^{\alpha,\frac{\alpha}{2}\nu}(\bar{\Omega}_{T})}],
\end{equation}
\begin{align}\label{6.4}\notag
&\|u\|_{L_{2}((0,T),W^{2,2}(\Omega))}+\|\mathbf{D}_{t}^{\nu}u\|_{L_2(\Omega_{T})}
+\sum_{i=1}^{M}\|\mathbf{D}_{t}^{\nu_i}u\|_{L_2(\Omega_{T})}
+\sum_{j=1}^{N}\|\mathbf{D}_{t}^{\mu_j}u\|_{L_2(\Omega_{T})}\\&
\leq C[1+\|u_0\|_{W^{2,2}(\Omega)}+[\underset{t\in[0,T]}{\sup}I_{t}^{\nu}\|g\|^{2}_{L_{2}(\Omega)}]^{\frac{1}{2}}].
\end{align}
Here the positive constant $C$ is independent of $\lambda$ and depends only on $\nu,\nu_i,\mu_{j}$, $T,L,L_i$, $r$ and the corresponding norms of the coefficients $a_i,b_i,$ $\varrho_{0},$ $\varrho_{j}$, $\gamma_{j}$ and of the kernel $\mathcal{K}$.
\end{lemma}
\begin{remark}\label{r6.0}
Actually,  our arguments in Section \ref{s6.1} tells us that the term $\|u\|_{\C(\bar{\Omega}_{T})}$ is evaluated via minor norms
\[
\|u\|_{\C(\bar{\Omega}_{T})}\leq C[1+\|u_0\|_{W^{1,2}(\Omega)}+[\underset{t\in[0,T]}{\sup}I_{t}^{\nu}\|g\|^{2}_{L_{2}(\Omega)}]^{\frac{1}{2}}].
\]
\end{remark}

First, we notice that estimate \eqref{6.3} is verified with the standard Schauder approach and by means of Lemma \ref{l5.2} and bound \eqref{6.2}. Thus, to prove Lemma \ref{l6.1}, we are left to produce inequalities \eqref{6.2} and \eqref{6.4}. We preliminarily observe that verification of these estimates in the case of absence of $\mathbf{D}_{t}^{\mu_i}(\gamma_i u),$ $i=1,2,...,N,$ (i.e. $N=0$) is the simpler and repeat the main steps (with minor changes) in arguments related with the analysis of the case  $N\geq 1$. Hence, here we focus on the  case of the presence of at least one fractional derivative $\mathbf{D}_{t}^{\mu_i}(\gamma_i u)$ in the operator $\mathbf{D}_t u$. To this end, appealing to assumption \textbf{h3}, we rewrite  $\mathbf{D}_t u$ in the more suitable form to our analysis
\begin{align*}\label{6.5}\notag
\mathbf{D}_{t}u&=\,_{1}\mathbf{D}_{t}u+\,_{2}\mathbf{D}_{t}u,\quad
\,_{1}\mathbf{D}_{t}u=\mathbf{D}_{t}^{\nu}(\varrho u)+\sum_{i=1}^{M}\mathbf{D}_{t}^{\nu_{i}}(\varrho_{i}u),\\
\,_{2}\mathbf{D}_{t}u&=\sum_{j=1}^{N}[\mathbf{D}_{t}^{\nu}(\gamma_{j} u)-\mathbf{D}_{t}^{\mu_{j}}(\gamma_{j}u)]=
\sum_{j=1}^{N}\frac{\partial}{\partial t}\bigg(\mathcal{N}(t;\nu,\mu_j)*(\gamma_{j}u-\gamma_{j}(0)u_{0})\bigg)
\end{align*}
(see \eqref{5.3} for the definition of the kernel $\mathcal{N}(t;\nu,\mu_j)$).

Thanks to the positivity of the kernel  $\mathcal{N}(t;\nu,\mu_j)$ for $t\in[0,T^{*}]$ (see \textbf{h1}), we first prove estimates \eqref{6.2} and \eqref{6.4} for $t\in[0,T_0]$ where 
\begin{equation}\label{6.0}
T_{0}<\min\bigg\{T^{*},\underset{j\in\{1,2,...,N\}}{\min} \bigg(\frac{\nu\Gamma(1+\nu-\mu_j)}{\mu_{j}}\bigg)^{\frac{1}{\nu-\mu_{j}}}\bigg\}.
\end{equation}
After that, if $T\geq T_{0}$, we discuss the extension of these bounds to the interval $(T_{0},T]$. It is worth noting that this step is absent in the case of $N=0$, due to the proof of estimates \eqref{6.2} and \eqref{6.4} is carried out immediately on entire  time interval.

To verify \eqref{6.2} and \eqref{6.4} for $t\in[0,T_0]$, we will follow the strategy containing fourth main steps. In the first, we evaluate the function $u(x,t)$ in the class $\C([0,T_{0}], W^{1,2}(\Omega))\cap L_{2}((0,T_{0}),W^{2,2}(\Omega))$. Then, Sobolev embedding theorem (see, e.g., \cite[Subsection 5.4]{AF}) allows us to readily obtain the bound of $\|u\|_{\C(\bar{\Omega}_{T_0})}$ exploiting only the estimate on $\|u\|_{\C([0,T_0],W^{1,2}(\Omega))}$. On the second stage, we evaluate the term $\|\frac{\partial u}{\partial x}\|_{\C(\bar{\Omega}_{T_0})}$ via integral iteration technique adapted to the case of multi-term fractional derivatives. After that, to complete the proof of \eqref{6.2}, appealing to estimate of $\|u\|_{\C([0,T_0],\C^{1}(\bar{\Omega}))}$, we arrive at the bound of the H\"{o}lder seminorm of $u$ with respect to time. Finally, taking into account \eqref{6.2}, we achieve estimate \eqref{6.4} via evaluation of $\|\mathbf{D}_{t}^{\nu}u\|_{L_{2}(\Omega_{T_0})}$, $\|\mathbf{D}_{t}^{\nu_{i}}u\|_{L_{2}(\Omega_{T_0})}$, $\|\mathbf{D}_{t}^{\mu_{j}}u\|_{L_{2}(\Omega_{T_0})}$.

Note that assumption \textbf{h3} provides the existence of a constant $C_0$ such that 
\begin{equation}\label{6.0*}
\sum_{i=0}^{2}[\underset{\bar{\Omega}_{T}}{\sup}|a_i(x,t)|+\underset{\bar{\Omega}_{T}}{\sup}|b_i(x,t)|]+\underset{\bar{\Omega}_{T}}{\sup}\bigg|\frac{\partial a_2}{\partial x}(x,t)\bigg|+\underset{\bar{\Omega}_{T}}{\sup}\bigg|\frac{\partial b_2}{\partial x}(x,t)\bigg|
\leq C_{0}.
\end{equation}
%%%%%%%%%%%%%%%%%%%%%%%%%%%%%%%%%%%%%%%%%%%%%%%%%%%%%%%%%%%%%%%%%%%%%%
 
%%%%%%%%%%%%%%%%%%%%%%%%%%%%%%%%%%%%%%%%%%%%%%%%%%%%%%%%%%%%%%%%%%%%%%

\subsection{Estimate of $\|u\|_{\C(\bar{\Omega}_{T_0})}$}
\label{s6.1}

We begin to evaluate the norm of $u$ in the space $\C([0,T_0],L_2(\Omega))$. To this end, we multiply the equation in \eqref{6.1} by $u(x,\tau)$, and then we integrate over $\Omega$ and compute the fractional integral $I_{t}^{\nu}$. Thus, we have
\begin{equation}\label{6.6}
\sum_{j=1}^{5}\mathcal{R}_{j}(t)=0,
\end{equation}
where we put
\begin{align*}
\mathcal{R}_{1}(t)&=I_{t}^{\nu}\bigg(\int_{\Omega}\,_{1}\mathbf{D}_{\tau}u udx\bigg)(t),\quad
\mathcal{R}_{2}(t)=I_{t}^{\nu}\bigg(\int_{\Omega}\,_{2}\mathbf{D}_{\tau}u udx\bigg)(t),\quad
\mathcal{R}_{3}(t)=-I_{t}^{\nu}\bigg(\int_{\Omega}\mathcal{L}_{1}u udx\bigg)(t),\\
\mathcal{R}_{4}(t)&=-I_{t}^{\nu}\bigg(\int_{\Omega}(\mathcal{K}*\mathcal{L}_{2}u) udx\bigg)(t),\quad
\mathcal{R}_{5}(t)=I_{t}^{\nu}\bigg(\int_{\Omega}[\lambda f(u)-g] udx\bigg)(t).
\end{align*}
At this point, we estimate each term $\mathcal{R}_{j}$, separately.

\noindent $\bullet$ By Propositions \ref{p5.1} and \ref{p5.1*} and assumptions \textbf{h2} and \textbf{h3},
\begin{align*}
\mathcal{R}_{1}(t)&\geq \frac{1}{2}I_{t}^{\nu}\Big(\varrho^{-1}\int_{\Omega}\partial_{t}^{\nu}(\varrho u)^{2}dx\Big)(t)
+
\frac{1}{2}\sum_{i=1}^{M}I_{t}^{\nu}\Big(\varrho_{i}^{-1}\int_{\Omega}\partial_{t}^{\nu_{i}}(\varrho_{i} u)^{2}dx\Big)(t)\\
&
-\frac{1}{2}\int_{\Omega}u_0^{2}dx\Big[\varrho^{2}(0)I_{t}^{\nu}(\varrho^{-1}\omega_{1-\nu})+\sum_{i=1}^{M}\varrho_{i}^{2}(0)
I_{t}^{\nu}(\varrho_{i}^{-1}\omega_{1-\nu_{i}})
\Big]\\
&
\geq \frac{\delta_{1}}{2}\int_{\Omega}u^{2}(x,t)dx
-\int_{\Omega}u_0^{2}dx\Big[2\varrho(0)+\sum_{i=1}^{M}\varrho_{i}(0)
\omega_{1+\nu-\nu_{i}}(t)\Big].
\end{align*}
% Finally, according to definition \eqref{5.1} and conditions \textbf{h1}, \textbf{h3},  we end up with the bound
%\begin{align*}
%\mathcal{R}_{1}(t)&\geq \frac{1}{2}[\delta_1+C_0\underset{i\in[1,2,...,M]}{\min}\rho_i(0)]\int_{\Omega}u^{2}(x,t)dx\\
%&
%-
%\|u_0\|^{2}_{L_2(\Omega)}\Big[\rho(T)+(\max\{1,T\})^{\nu-\nu_1}\sum_{i=1}^{M}\rho_{i}(T)\Big].
%\end{align*}

\noindent $\bullet$ As for the term $\mathcal{R}_{2}(t)$, we recast the arguments leading to the bound of $\mathcal{R}_{1}(t)$. Thus, exploiting  Corollary \ref{c5.0}, Proposition \ref{p5.1*}, and conditions \textbf{h2}, \textbf{h3}, we immediately arrive at
\begin{align*}
\mathcal{R}_2(t)&\geq\frac{1}{2}\sum_{j=1}^{N}I_{t}^{\nu}\bigg(\gamma^{-1}_{j}\int_{\Omega}\frac{\partial}{\partial \tau}(\mathcal{N}(\tau;\nu,\mu_j)*(\gamma_{j}u)^{2})dx\bigg)(t)
-\frac{1}{2}\int_{\Omega}u_0^{2}dx\sum_{j=1}^{N}I_{t}^{\nu}\bigg(\mathcal{N}(\tau;\nu,\mu_{j})\frac{\gamma^{2}_{j}(0)}{\gamma_{j}(\tau)}\bigg)(t)\\
&
\geq
\frac{1}{2}\sum_{j=1}^{N}\gamma_{j}(t)\int_{\Omega}u^{2}(x,t)dx-\frac{1}{2}\sum_{j=1}^{N}I_{t}^{\nu-\mu_{j}}\bigg(\gamma_{j}\int_{\Omega}u^{2}dx\bigg)(t)
-
\frac{1}{2}\sum_{j=1}^{N}\gamma_{j}^{2}(0)\gamma_{j}^{-1}(t)\int_{\Omega}u_{0}^{2}dx\\
&
+\frac{1}{2}\sum_{j=1}^{N}
\bigg[
I_{t}^{1}\bigg(-\nu\mathcal{W}(\gamma_{j}^{-1})\int_{\Omega}(\gamma_{j}u)^{2}dx\bigg)(t)-
I_{t}^{1+\nu-\mu_{j}}\bigg(-\mu_{j}\mathcal{W}(\gamma_{j}^{-1})\int_{\Omega}(\gamma_{j}u)^{2}dx\bigg)(t)
\bigg].
\end{align*}
After that, Corollary \ref{c5.1} with $T_1=T_0$ (see restriction \eqref{6.0}) and assumption \textbf{h3} tell us about the positivity of the last sum in the right-hand side of this inequality. Thus, we have
\[
\mathcal{R}_2(t)\geq \frac{N}{2}\delta_{3}\int_{\Omega}u^{2}(x,t)dx
-\frac{1}{2}\sum_{j=1}^{N}\gamma_{j}(T)I_{t}^{\nu-\mu_j}\bigg(\int_{\Omega}u^{2}dx\bigg)(t)
-\frac{1}{2}\sum_{j=1}^{N}\gamma_{j}(0)\|u_0\|^{2}_{L_{2}(\Omega)}.
\]

\noindent $\bullet$ Integrating by parts and keeping in mind the homogeneous Dirichlet boundary condition and according to assumptions \textbf{h2-h3} and \eqref{6.0*}, we deduce
\[
\mathcal{R}_{3}(t)\geq\frac{3\delta_0}{4} I_{t}^{\nu}\bigg(\int_{\Omega}\bigg(\frac{\partial u}{\partial x}\bigg)^{2}dx\bigg)(t)
-C_1I_{t}^{\nu}\bigg(\int_{\Omega}u^{2}dx\bigg),
\]
where
\[
C_1=C_0(1+4C_{0}/\delta_0).
\]
It is worth noting that, we used the Cauchy inequality to evaluate the term $u\frac{\partial u}{\partial x}$.

\noindent $\bullet$  Coming to the term $\mathcal{R}_{4}(t)$, we integrate by parts and take advantage of the Cauchy and the Poincare inequalities and requirements \textbf{h2-h3}, \eqref{6.0*}. In summary, we obtain
\[
\mathcal{R}_{4}(t)\geq -\frac{\delta_0}{4}I_{t}^{\nu}\bigg(\int_{\Omega}\bigg(\frac{\partial u}{\partial x}\bigg)^{2}dx\bigg)(t)
-CI_{t}^{\nu}\bigg(\mathcal{K}*\int_{\Omega}\bigg(\frac{\partial u}{\partial x}\bigg)^{2}dx\bigg)(t)
-CI_{t}^{\nu}\bigg(\int_{\Omega}u^{2}dx\bigg)(t),
\]
where the positive constant $C$ is independent of $\lambda$ and $T_0$.

\noindent $\bullet$ Exploiting assumption \textbf{h6 (i)} with the Cauchy inequality and point (ii) of Proposition \ref{p5.1}, we have
\[
\mathcal{R}_{5}(t)\geq -L|\Omega|\omega_{1+\nu}(t)-(L+2)I_{t}^{\nu}\bigg(\int_{\Omega}u^{2}dx\bigg)(t)-I_{t}^{\nu}\bigg(\int_{\Omega}g^{2}dx\bigg)(t).
\]
Finally, collecting all estimates of $\mathcal{R}_{j}(t)$, we end up with
\begin{align*}
\int_{\Omega}u^{2}(x,t)dx+I_{t}^{\nu}\bigg(\int_{\Omega}\bigg(\frac{\partial u}{\partial x}\bigg)^{2}dx\bigg)(t)
&\leq C\Big\{\big(I_{t}^{\nu}+\sum_{i=1}^{N}I_{t}^{\nu-\mu_i}\big)\Big(\int_{\Omega}u^{2}dx\Big)(t)
+
I_{t}^{\nu}\Big(\mathcal{K}*\int_{\Omega}\bigg(\frac{\partial u}{\partial x}\bigg)^{2}dx\Big)(t)\\
&
+
1+\|u_0\|_{L_2(\Omega)}+I_{t}^{\nu}\bigg(\int_{\Omega}g^{2}dx\bigg)(t)
\Big\}.
\end{align*}
Appealing to associative properties of a convolution to handle the second term in the right-hand side of this estimate, we use the Gronwall-type inequality (4.3) in \cite{KPV3} and arrive at the desired bound
\begin{equation}\label{6.7}
\|u\|^{2}_{\C([0,T_0],L_2(\Omega))}+\underset{t\in[0,T_0]}{\sup}I_{t}^{\nu}\bigg(\int_{\Omega}\bigg(\frac{\partial u}{\partial x}\bigg)^{2}dx\bigg)(t)
\leq C[1+\|u_0\|_{L_2(\Omega)}+\underset{t\in[0,T]}{\sup}I_{t}^{\nu}\|g\|^{2}_{L_2(\Omega)}]
\end{equation}
with the constant $C$ being independent of $\lambda$ and $T_0$.
Moreover, inequality \eqref{6.7} provides the estimate
\begin{equation}\label{6.8}
\bigg\|\frac{\partial u}{\partial x}\bigg\|_{L_2(\Omega_{T_0})}\leq C[1+\|u_0\|^{2}_{L_{2}(\Omega)}+\underset{t\in[0,T]}{\sup}I_{t}^{\nu}\|g\|^{2}_{L_2(\Omega)}].
\end{equation}

In order to complete the estimate of $\|u\|_{\C(\bar{\Omega}_{T_0})}$, we need a similar bound for the derivative $\frac{\partial u}{\partial x}$. To this end, we multiply the equation in \eqref{6.1} by $\frac{\partial^{2} u}{\partial x^{2}}(x,\tau)$ and then we integrate over $\Omega$ and compute the fractional integral $I_{t}^{\nu}$. Taking into account \textbf{h2-h4} and \textbf{h6 (i)} and applying Lemma \ref{l5.1} with $p=2$ and \eqref{6.8}, we obtain
\begin{align}\label{6.10*}\notag
&\frac{1}{2}[\delta_1+N\delta_3]\int_{\Omega}\bigg(\frac{\partial u}{\partial x}\bigg)^{2}dx+
\sum_{j=1}^{N}(\nu I_{t}^{1}-\mu_jI_{t}^{1+\nu-\mu_j})\bigg(\gamma_{j}^{2}(-\mathcal{W}(\gamma_{j}^{-1}))\int_{\Omega}\bigg(\frac{\partial u}{\partial x}\bigg)^{2}dx\bigg)(t)\\\notag &+
\frac{\delta_0}{2}I_{t}^{\nu}\bigg(\int_{\Omega}\bigg(\frac{\partial^{2} u}{\partial x^{2}}\bigg)^{2}dx\bigg)(t)
\leq
\frac{1}{2}\sum_{j=1}^{N}\gamma_{j}(T)I_{t}^{\nu-\mu_j}\bigg(\int_{\Omega}\bigg(\frac{\partial u}{\partial x}\bigg)^{2}dx\bigg)(t)
\\
&+\frac{3\|\mathcal{K}\|_{L_1(0,T)}C_{0}^{2}}{\delta_{0}}|\mathcal{K}|*I_{t}^{\nu}\bigg(\int_{\Omega}\bigg(\frac{\partial^{2} u}{\partial x^{2}}\bigg)^{2}dx\bigg)
+
C[1+\|u_0\|^{2}_{W^{1,2}(\Omega)}+\underset{t\in[0,T]}{\sup}I_{t}^{\nu}\|g\|^{2}_{L_2(\Omega)}].
\end{align}
At last, using Corollary \ref{c5.1} with $T_1=T_0$ (see \eqref{6.0*}) to handle the second term in the left-hand side of this inequality, and then applying Gronwall-type inequality (4.3) in \cite{KPV3}, we conclude that
\begin{equation}\label{6.9*}
\underset{t\in[0,T_0]}{\sup}\int_{\Omega}\bigg(\frac{\partial u}{\partial x}\bigg)^{2}dx+
\underset{t\in[0,T_0]}{\sup} I_{t}^{\nu}\bigg(\int_{\Omega}\bigg(\frac{\partial^{2} u}{\partial x^{2}}\bigg)^{2}dx\bigg)
\leq C[1+\|u_0\|^{2}_{W^{1,2}(\Omega)}+\underset{t\in[0,T]}{\sup}I_{t}^{\nu}\|g\|^{2}_{L_2(\Omega)}]
\end{equation}
where the positive constant $C$ is independent of $\lambda$ and $T_0$.

Collecting this estimate with \eqref{6.7} and \eqref{6.8} and applying  Sobolev embedding theorem (as we wrote before) arrive at the desired estimate
\begin{align}\label{6.9}\notag
&\|u\|_{\C(\bar{\Omega}_{T_0})}+\|u\|_{L_2((0,T_0),W^{2,2}(\Omega))}+
\|u\|_{\C([0,T_0],W^{1,2}(\Omega))}\\
&
\leq
C[1+\|u_0\|_{W^{1,2}(\Omega)}+\sqrt{\underset{t\in[0,T]}{\sup}I_{t}^{\nu}\|g\|^{2}_{L_2(\Omega)}}]\equiv C\mathcal{F}(u_0,g)
\end{align}
with the positive constant $C$ being independent of $\lambda, T_0$.

Obviously, the last estimate allows us to evaluate the term $\|\mathcal{K}*\mathcal{L}_{2}u\|_{L_2(\Omega_{T_0})}$. Indeed,
the Young inequality of a convolution (see, e.g. \cite{Fo}) provides
\[
\|\mathcal{K}*\mathcal{L}_{2}u\|_{L_2(\Omega_{T_0})}\leq C\|\mathcal{K}\|_{L_{1}(0,T)}\|\mathcal{L}_{2}u\|_{L_2(\Omega_{T_0})}.
\]
To manage the term $\|\mathcal{L}_{2}u\|_{L_2(\Omega_{T_0})}$, we apply \eqref{6.9} and exploit the smoothness of the coefficients $b_j,$ $j=0,1,2.$ Hence, we have
\begin{equation}\label{6.10}
\|\mathcal{K}*\mathcal{L}_{2}u\|^{2}_{L_2(\Omega_{T_0})}\leq C\|u\|^{2}_{L_2((0,T_0),W^{2,2}(\Omega))}
\leq 
C[1+\|u_0\|^{2}_{W^{1,2}(\Omega)}+\underset{t\in[0,T]}{\sup}I_{t}^{\nu}\|g\|^{2}_{L_2(\Omega)}],
\end{equation} 
where the positive constant $C$ is independent of $\lambda$ and $T_0$.
\qed
\begin{remark}\label{r6.1}
The treatment of the term $I_{t}^{\nu}(\int_{\Omega}f(u)u_{xx}dx)(t)$ in \eqref{6.10*} in the case of \eqref{4.1} differs from the case of $f(u)$ satisfying \eqref{4.2}. Indeed, if  \eqref{4.2} holds, we first rewrite the term  $I_{t}^{\nu}(\int_{\Omega}f(u)u_{xx}dx)(t)$ in the form
\[
 I_{t}^{\nu}\bigg(\int_{\Omega}f(u)u_{xx}dx\bigg)(t)= I_{t}^{\nu}\bigg(\int_{\Omega}\bar{f}(u)u_{xx}dx\bigg)(t)+ I_{t}^{\nu}\bigg(\int_{\Omega}[f(0)-L_4u]u_{xx}dx\bigg)(t)
\]
with
\[
\bar{f}(u)=f(u)-f(0)+L_4u.
\]
It is apparent that the second term in this representation is controlled with the arguments leading to the estimate of $\mathcal{R}_5$. Coming the first term, we note that  the function $\bar{f}(u)$ meets the first three requirements in \eqref{4.2} and, besides,
\[
\bar{f}(u)=0\quad\text{on}\quad \partial\Omega_T,\quad \bar{f}'(u)\geq 0.
\] 
Thus, taking into account these relations and integrating by parts arrive at
\[
I_{t}^{\nu}\bigg(\int_{\Omega}\bar{f}(u)u_{xx}dx\bigg)(t)=I_{t}^{\nu}\bigg(\int_{\Omega}\bar{f}'(u)(u_{x})^{2}dx\bigg)(t)\geq 0,
\] 
which in turn provides \eqref{6.10*}.
\end{remark}
%%%%%%%%%%%%%%%%%%%%%%%%%%%%%%%%%%%%%%%%%%%%%%%%%%%%%%%%%%%%%%%%%%%%%%
 
%%%%%%%%%%%%%%%%%%%%%%%%%%%%%%%%%%%%%%%%%%%%%%%%%%%%%%%%%%%%%%%%%%%%%%

\subsection{The bound of $\|\frac{\partial u}{\partial x}\|_{\C(\bar{\Omega}_{T_0})}$}
\label{s6.2}
In this subsection  we aim to prove the bound
\begin{equation}\label{6.11}
\bigg\|\frac{\partial u}{\partial x}\bigg\|_{\C(\bar{\Omega}_{T_0})}\leq \mathfrak{C}[1+\|u_0\|_{W^{2,2}(\Omega)}+(\underset{t\in[0,T]}{\sup}I_{t}^{\nu}\|g\|^{2}_{L_2(\Omega)})
^{1/2}]
\end{equation}
with the positive constant $\mathfrak{C}$ is independent of $\lambda$ and $T_0$ and depends only on $T,\nu,\nu_i,\mu_j,L,$ $\|\mathcal{K}\|_{L_1(0,T)},$ and the corresponding norms of the coefficients.

For simplicity of consideration, we assume  that
\begin{equation}\label{6.12}
\bigg\|\frac{\partial u}{\partial x}\bigg\|_{\C(\bar{\Omega}_{T_0})}\geq \frac{\mathfrak{C}}{2}\mathcal{F}(u_0,g),
\end{equation}
where $\mathcal{F}(u_0,g)$ is defined in \eqref{6.9}, otherwise we immediately arrive at \eqref{6.11}.

To verify \eqref{6.11} in the case of \eqref{6.12}, we reason similarly to the case analyzed in Subsection \ref{s6.1}. Thus, we first  multiply the equation in \eqref{6.1} by $p\frac{\partial}{\partial x}\bigg(\frac{\partial u}{\partial x}\bigg)^{p-1}$ and then integrate over $\Omega$ and compute the fractional integral $I_{t}^{\nu}$
\begin{align*}
&\varrho(0)\int_{\Omega}\bigg(\frac{\partial u}{\partial x}\bigg)^{p}dx+\frac{3\delta_0}{4}I_{t}^{\nu}\bigg(\int_{\Omega}\bigg(\frac{\partial u}{\partial x}\bigg)^{p-2}\bigg(\frac{\partial^{2} u}{\partial x^{2}}\bigg)^{2}dx\bigg)(t)
\\&\leq C\bigg\{p(p-1)[1+\|u\|_{\C(\bar{\Omega}_{T_0})}]^{2}I_{t}^{\nu}\bigg(\int_{\Omega}\bigg(\frac{\partial u}{\partial x}\bigg)^{p-2}dx\bigg)(t)\\
&
+
p(p-1)I_{t}^{\nu}\bigg(\int_{\Omega}\bigg(\frac{\partial u}{\partial x}\bigg)^{p}dx\bigg)(t)
+
\int_{\Omega}\bigg(\frac{\partial u_{0}}{\partial x}\bigg)^{p}dx\\&
+
\bigg(\sum_{j=1}^{N}\gamma_{j}(T)T^{\mu_N-\mu_{j}}\bigg)I_{t}^{\nu-\mu_{N}}\bigg(\int_{\Omega}\bigg(\frac{\partial u}{\partial x}\bigg)^{p}dx\bigg)(t)\\&+
(p-1)N\Gamma(1+\nu)\underset{j\in(1,2,...,N)}{\sup}\|\gamma_{j}\|_{\C^{1}([0,T])}T^{1-\nu}
I_{t}^{\nu}\bigg(\int_{\Omega}\bigg(\frac{\partial u}{\partial x}\bigg)^{p}dx\bigg)(t)\bigg\}\\&
+
I_{t}^{\nu}\bigg(\int_{\Omega}(\mathcal{K}*\mathcal{L}_{2}u)p\frac{\partial }{\partial x}\bigg(\frac{\partial u}{\partial x}\bigg)^{p-1}dx\bigg)(t)
\end{align*}
with $C$ being independent of $\lambda$ and $T_0$.
Here we used Corollary \ref{c5.1}, Lemma \ref{l5.1} and assumptions \textbf{h1-h3} to manage the term $I_{t}^{\nu}\bigg(\int_{\Omega}\mathbf{D}_{t}up\frac{\partial }{\partial x}\bigg(\frac{\partial u}{\partial x}\bigg)^{p-1}dx\bigg)$, while to handle the term $I_{t}^{\nu}\bigg(\int_{\Omega}\mathcal{L}_{1}up\frac{\partial }{\partial x}\bigg(\frac{\partial u}{\partial x}\bigg)^{p-1}dx\bigg)$, we appealed to \textbf{h2-h3} and the Young inequality. Finally, we take advantage of \eqref{4.1} and again  the Young inequality to treat 
$I_{t}^{\nu}\bigg(\int_{\Omega}[\lambda f(u)-g]p\frac{\partial }{\partial x}\bigg(\frac{\partial u}{\partial x}\bigg)^{p-1}dx\bigg)$. After that, keeping in mind restriction \eqref{6.12} and inequality \eqref{6.9}, we conclude that
\begin{align}\label{6.13}\notag
&\varrho(0)\int_{\Omega}\bigg(\frac{\partial u}{\partial x}\bigg)^{p}dx+\frac{3\delta_0}{4}I_{t}^{\nu}\bigg(\int_{\Omega}\bigg(\frac{\partial u}{\partial x}\bigg)^{p-2}\bigg(\frac{\partial^{2} u}{\partial x^{2}}\bigg)^{2}dx\bigg)(t)\\\notag
&
\leq C\bigg\{p(p-1)I_{t}^{\nu}\bigg(\underset{\bar{\Omega}}{\sup}\bigg|\frac{\partial u}{\partial x}\bigg|^{p}\bigg)(t)
+
I_{t}^{\nu-\mu_{N}}\bigg(\int_{\Omega}\bigg(\frac{\partial u}{\partial x}\bigg)^{p}dx\bigg)(t)\\
&
+
\int_{\Omega}\bigg(\frac{\partial u_{0}}{\partial x}\bigg)^{p}dx\bigg\}
+
I_{t}^{\nu}\bigg(\int_{\Omega}(\mathcal{K}*\mathcal{L}_{2}u)p\frac{\partial }{\partial x}\bigg(\frac{\partial u}{\partial x}\bigg)^{p-1}dx\bigg)(t),
\end{align}
where the positive constant $C$ is independent of $T_0,\lambda,p$ and the corresponding norms of $u_0,g$.

Now we are left to evaluate the last term in the right-hand side of \eqref{6.13}. To this end, applying the Young inequality and restriction \eqref{6.0*} provides the estimate
\begin{align*}
\int_{\Omega}(\mathcal{K}*\mathcal{L}_{2}u)p\frac{\partial}{\partial x}\bigg(\frac{\partial u}{\partial x}\bigg)^{p-1}dx&\leq
C_{0}\frac{2}{\varepsilon_{1}}p(p-1)\underset{\bar{\Omega}}{\sup}\bigg|\frac{\partial u}{\partial x}(x,t)\bigg|^{p-2}\bigg(|\mathcal{K}|*\|u(\cdot,t)\|^{2}_{W^{2,2}(\Omega)}\bigg)(t)\\
&
+
C_{0}p(p-1)\varepsilon_1(|\mathcal{K}|*1)(t)\int_{\Omega}\bigg(\frac{\partial u}{\partial x}(x,t)\bigg)^{p-2}
\bigg(\frac{\partial^{2} u}{\partial x^{2}}(x,t)\bigg)^{2}dx
\end{align*}
where the small quantity $\varepsilon_1$ being specified later.

Exploiting assumption \textbf{h4} and estimate \eqref{6.9*} to manage the first term in the right-hand side of this inequality, we easily conclude that
\begin{align*}
&I_{t}^{\nu}\bigg(\int_{\Omega}(\mathcal{K}*\mathcal{L}_{2}u)p\frac{\partial}{\partial x}\bigg(\frac{\partial u}{\partial x}\bigg)^{p-1}dx\bigg)
\\&\leq
\frac{C p(p-1)}{\varepsilon_1}I_{t}^{\nu}\bigg(\underset{\bar{\Omega}}{\sup}\bigg|\frac{\partial u}{\partial x}(x,t)\bigg|^{p-2}\bigg)(t)\mathcal{F}^{2}(u_0,g)\\
&
+
C_{0}p(p-1)\varepsilon_1\|\mathcal{K}\|_{L_1(0,T)}I_{t}^{\nu}\bigg(\int_{\Omega}\bigg(\frac{\partial u}{\partial x}(x,t)\bigg)^{p-2}
\bigg(\frac{\partial^{2} u}{\partial x^{2}}(x,t)\bigg)^{2}dx\bigg)(t).
\end{align*}
Applying restriction \eqref{6.12} to control the first term in the right-hand side leads to the estimate
\begin{align*}
&I_{t}^{\nu}\bigg(\int_{\Omega}(\mathcal{K}*\mathcal{L}_{2}u)p\frac{\partial}{\partial x}\bigg(\frac{\partial u}{\partial x}\bigg)^{p-1}dx\bigg)\\&\leq
\frac{2C p(p-1)}{\mathfrak{C}^{2}\varepsilon_1}I_{t}^{\nu}\bigg(\underset{\bar{\Omega}}{\sup}\bigg|\frac{\partial u}{\partial x}(x,t)\bigg|^{p}\bigg)(t)\\
&
+
C_{0}p(p-1)\varepsilon_1\|\mathcal{K}\|_{L_1(0,T)}I_{t}^{\nu}\bigg(\int_{\Omega}\bigg(\frac{\partial u}{\partial x}(x,t)\bigg)^{p-2}
\bigg(\frac{\partial^{2} u}{\partial x^{2}}(x,t)\bigg)^{2}dx\bigg)(t).
\end{align*}
Choosing
\[
\varepsilon_{1}=\frac{\delta_0}{4C_0\|\mathcal{K}\|_{L_1(0,T)}}
\]
and coming to estimate \eqref{6.13}, we easily draw
\begin{align}\label{6.13*}\notag
&\varrho(0)\int_{\Omega}\bigg(\frac{\partial u}{\partial x}\bigg)^{p}dx
+\frac{\delta_0}{2}I_{t}^{\nu}\bigg(\int_{\Omega}\bigg(\frac{\partial u}{\partial x}\bigg)^{p-2}\bigg(\frac{\partial^{2} u}{\partial x^{2}}\bigg)^{2}dx\bigg)(t)\\\notag
&
\leq C_{2}\bigg\{p(p-1)I_{t}^{\nu}\bigg(\underset{\bar{\Omega}}{\sup}\bigg|\frac{\partial u}{\partial x}\bigg|^{p}\bigg)(t)
+
I_{t}^{\nu-\mu_{N}}\bigg(\int_{\Omega}\bigg(\frac{\partial u}{\partial x}\bigg)^{p}dx\bigg)(t)\\
&
+
\int_{\Omega}\bigg(\frac{\partial u_{0}}{\partial x}\bigg)^{p}dx\bigg\}
,
\end{align}
where the positive constant $C_2$ is independent of $T_0,$ $\lambda, p$ and the norms of $u_0,g$.

Finally, to handle the term $I_{t}^{\nu}\bigg(\underset{\bar{\Omega}}{\sup}\bigg|\frac{\partial u}{\partial x}(x,t)\bigg|^{p}\bigg)$, we exploit bound \eqref{i.1} with $v=\bigg(\frac{\partial u}{\partial x}(x,t)\bigg)^{p/2}$ and have
\begin{align*}
\underset{\bar{\Omega}}{\sup}\bigg|\frac{\partial u}{\partial x}(x,t)\bigg|^{p}&\leq \frac{\varepsilon^{2}p^{2}}{4}\int_{\Omega}
\bigg(\frac{\partial u}{\partial x}(x,t)\bigg)^{p-2}\bigg(\frac{\partial^{2} u}{\partial x^{2}}(x,t)\bigg)^{2}dx
+
\varepsilon^{2}\int_{\Omega}\bigg(\frac{\partial u}{\partial x}(x,t)\bigg)^{p}dx
\\&
+
\frac{C}{\varepsilon}
\bigg(\int_{\Omega}
\bigg(\frac{\partial u}{\partial x}(x,t)\bigg)^{p/2}dx\bigg)^{2}
\end{align*}
with a small positive $\varepsilon$.

Then, setting here 
\[
\varepsilon=\frac{\varepsilon_0}{ \sqrt{|\Omega|}}
\]
with some positive $\varepsilon_0<1$ and computing the fractional integral, we arrive at
\begin{align*}
I_{t}^{\nu}\bigg(\underset{\bar{\Omega}}{\sup}\bigg|\frac{\partial u}{\partial x}(x,t)\bigg|^{p}\bigg)(t)&\leq \frac{\varepsilon_{0}^{2}p^{2}}{4[1-\varepsilon^{2}]|\Omega|}I_{t}^{\nu}\bigg(\int_{\Omega}
\bigg(\frac{\partial u}{\partial x}(x,t)\bigg)^{p-2}\bigg(\frac{\partial^{2} u}{\partial x^{2}}(x,t)\bigg)^{2}dx\bigg)(t)
\\&
+
\frac{C\sqrt{|\Omega|}}{\varepsilon_{0}(1-\varepsilon_{0}^{2})}
I_{t}^{\nu}\bigg(
\bigg(\int_{\Omega}
\bigg(\frac{\partial u}{\partial x}(x,t)\bigg)^{p/2}dx\bigg)^{2}\bigg)(t),
\end{align*}
Collecting this estimate with \eqref{6.13*} and choosing
\[
\varepsilon_{0}^{2}=\frac{\delta_0|\Omega|}{\delta_0|\Omega|+p^{2}C_2},
\]
we achieve  the bound
\begin{align}\label{6.14}\notag
&\int_{\Omega}\bigg(\frac{\partial u}{\partial x}(x,t)\bigg)^{p}dx+p(p-1)I_{t}^{\nu}\bigg(\int_{\Omega}\bigg(\frac{\partial u}{\partial x}\bigg)^{p-2}\bigg(\frac{\partial^{2} u}{\partial x^{2}}\bigg)^{2}dx\bigg)(t)\\\notag
&
\leq
C\bigg\{I_{t}^{\nu-\mu_{N}}\bigg(\int_{\Omega}\bigg(\frac{\partial u}{\partial x}\bigg)^{p}dx\bigg)(t)+\int_{\Omega}\bigg(\frac{\partial u_{0}}{\partial x}\bigg)^{p}dx\\
&
+
p^{3}
I_{t}^{\nu}\bigg(\int_{\Omega}\bigg(\frac{\partial u}{\partial x}\bigg)^{p/2}dx\bigg)^{2}(t),
\end{align}
where the positive constant $C$ is independent of $T_0$, $\lambda,$ $p$ and  the norms of $u_0$ and $g$.

In order to evaluate the term $I_{t}^{\nu-\mu_{N}}\bigg(\int_{\Omega}\bigg(\frac{\partial u}{\partial x}\bigg)^{p}dx\bigg)(t)$ in the right-hand side of \eqref{6.14}, we exploit the Gronwal-type inequality \cite[Proposition 4.3]{KPV3} and obtain
\[
\int_{\Omega}\bigg(\frac{\partial u}{\partial x}(x,t)\bigg)^{p}dx\leq A E_{\nu-\mu_{N}}(Ct^{\nu-\mu_{N}})\quad\text{for any}\quad t\in[0,T_0],
\]
where we put
\[
A=C\bigg[\int_{\Omega}\bigg(\frac{\partial u_0}{\partial x}\bigg)^{p}dx+p^{3}\bigg\|\int_{\Omega}\bigg(\frac{\partial u}{\partial x}\bigg)^{p/2}dx\bigg\|^{2}_{\C([0,T_0])}\bigg]
\]
and
\[
E_{\theta}(t)=\sum_{m=0}^{+\infty}\frac{z^{m}}{\Gamma(1+m\theta)}
\]
is the classical Mittag-Leffler function of the order $\theta$ (see, e.g., its definition in \cite[(2.2.4)]{GKMR}). 

Applying this estimate to handle the first term in the right-hand side of \eqref{6.14} and then taking into account formula (3.7.44) in \cite{GKMR} to compute the fractional integral of Mittag-Leffler function produce the inequality
\begin{align*}
\int_{\Omega}\bigg(\frac{\partial u}{\partial x}\bigg)^{p}dx&\leq
CE_{\nu-\mu_{N}}(CT^{\nu-\mu_{N}})
\bigg\{1+\int_{\Omega}\bigg(\frac{\partial u_0}{\partial x}\bigg)^{p}dx\\
&
+p^{3}\bigg\|\int_{\Omega}\bigg(\frac{\partial u}{\partial x}\bigg)^{p/2}dx\bigg\|^{2}_{\C([0,T_0])}\bigg\}
\end{align*}
At last, denoting
\[
\mathcal{B}_{0}=CE_{\nu-\mu_{N}}(CT^{\nu-\mu_{N}})\qquad
\mathcal{A}_{m}=\underset{t\in [0,T_0]}{\sup}
\bigg(\int_{\Omega}\bigg(\frac{\partial u}{\partial x}\bigg)^{p}dx\bigg)^{1/p}
\]
with $p=2^{m},$ $m\geq 1$, we derive the bound
\begin{align}\label{6.15}\notag
\mathcal{A}_{m}&\leq \mathcal{B}_{0}^{2^{-m}}[1+\|u_{0}\|_{W^{1,p}(\Omega)}+(2^{3m2^{-m}})\mathcal{A}_{m-1}]\\\notag
&
\leq
\mathcal{B}_{0}^{2^{-m}}[1+\|u_0\|_{\C^{1}(\bar{\Omega})}+
(2^{3m2^{-m}})\mathcal{A}_{m-1}]\\
&
\leq
\mathcal{B}_{0}^{2^{-m}}[1+\|u_0\|_{W^{2,2}(\Omega)}]+
(2^{3m2^{-m}})\mathcal{A}_{m-1}].
\end{align}
To handle the term $\|u_0\|_{\C^{1}(\bar{\Omega})}$ in these inequalities, we used the Sobolev embedding theorem.

At this point, we analyze two possibilities:

\noindent\textbf{(i)} either $\max\bigg\{1+\|u_0\|_{W^{2,2}(\Omega)};\mathcal{A}_{m-1}\bigg\}=1+\|u_0\|_{W^{2,2}(\Omega)}$,

\noindent\textbf{(ii)} or $\max\bigg\{1+\|u_0\|_{W^{2,2}(\Omega)};\mathcal{A}_{m-1}\bigg\}=\mathcal{A}_{m-1}$.

\noindent Obviously, in the case of \textbf{(i)} passing to the limit as $m\to+\infty$ in \eqref{6.15}, we end up with the desired bound
\[
\underset{\bar{\Omega}_{T_0}}{\sup}\bigg|\frac{\partial u}{\partial x}\bigg|\leq C[1+\|u_0\|_{W^{2,2}(\Omega)}].
\]
Conversely, if \textbf{(ii)} holds then
\[
\mathcal{A}_{m}\leq \mathcal{B}_{0}^{2^{-m}}\mathcal{A}_{m-1}
\leq
C\exp\bigg\{
\sum_{m=1}^{+\infty}m2^{-m}
\bigg\}\mathcal{A}_{1}\leq C\mathcal{A}_{1}
\]
and letting $m\to+\infty$, we deduce that
\[
\underset{\bar{\Omega}_{T_0}}{\sup}\bigg|\frac{\partial u}{\partial x}\bigg|\leq C \mathcal{A}_{1}.
\]
Finally, applying \eqref{6.9*} to control the term $\mathcal{A}_{1}$, we obtain \eqref{6.11}.
\qed
%%%%%%%%%%%%%%%%%%%%%%%%%%%%%%%%%%%%%%%%%%%%%%%%%%%%%%%%%%%%%%%%%%%%%%
 
%%%%%%%%%%%%%%%%%%%%%%%%%%%%%%%%%%%%%%%%%%%%%%%%%%%%%%%%%%%%%%%%%%%%%%

\subsection{The estimate of $\langle u\rangle_{t,{\Omega}_{T_0}}^{(\nu/2)}$}
\label{s6.3}

This subsection will be devoted to obtain H\"{o}lder regularity in time of the solution $u$. Namely, we aim to achieve the bound
\[
\langle u\rangle_{t,{\Omega}_{T_0}}^{(\nu/2)}\leq C\bigg[1+\|u_0\|_{W^{2,2}(\Omega)}+\bigg(\underset{t\in[0,T]}{\sup}I_{t}^{\nu}(
\|g\|^{2}_{L_{2}(\Omega)})(t)\bigg)^{1/2}+\underset{t\in[0,T]}{\sup}
\|g\|_{L_{2}(\Omega)}\bigg].
\]
It is apparent that this estimate is a simple consequence of the inequality
\begin{align}\label{6.16}\notag
\Delta_{h}u&=|u(x,t+h)-u(y,t)|\\
&\leq C\bigg[1+\|u_0\|_{W^{2,2}(\Omega)}+\bigg(\underset{t\in[0,T]}{\sup}I_{t}^{\nu}(
\|g\|^{2}_{L_{2}(\Omega)})(t)\bigg)^{1/2}+\underset{t\in[0,T]}{\sup}
\|g\|_{L_{2}(\Omega)}\bigg]
[h^{\nu/2}+|x-y|]
\end{align}
for any $x,y\in\bar{\Omega}, $ and $h\in(0,1)$ such that $x+h^{\nu/2}\in\bar{\Omega}$ and $t\in[0,T_0]$. Indeed, substituting  $x=y$ in \eqref{6.16}, we immediately arrive at the desired bound.

First, appealing to the integral mean value theorem, we conclude that
\begin{equation}\label{6.17}
\int_{x}^{x+h^{\nu/2}}[u(z,t+h)-u(z,t)]dz=h^{\nu/2}[u(x^{*},t+h)-u(x^{*},t)]
\end{equation} 
for some $x^{*}\in[x,x+h^{\nu/2}]$ and $t\in[0,T_0]$. Taking into account this relation and \eqref{6.11}, we deduce that
\begin{align}\label{6.18}\notag
\Delta_{h}u&\leq |u(x,t+h)-u(x^{*},t+h)|+|u(x^{*},t+h)-u(x^{*},t)|+|u(x^{*},t)-u(y,t)|\\\notag
&
\leq
C\{|x-x^{*}|+|x^{*}-y|\}\underset{\bar{\Omega}_{T_{0}}}{\sup}\bigg|\frac{\partial u}{\partial x}\bigg|
+
|u(x^*,t+h)-u(x^{*},t)|\\\notag
&
\leq
C[1+\|u_0\|_{W^{2,2}(\Omega)}+\|I_{t}^{\nu}\|g\|^{2}_{L_{2}(\Omega)}\|^{1/2}_{\C(\bar{\Omega}_{T})}]
[h^{\nu/2}+|x^{*}-y|]\\
&+
|u(x^*,t+h)-u(x^{*},t)|\equiv d_{1}+d_{2}.
\end{align}
At this point, we evaluate each term $d_i$, separately.

\noindent$\bullet$ As for $d_1$, we should to evaluate $|x^*-y|$ in appropriate way. To this end, we analyze
 three possibilities   to the location of $y$:
\begin{description}
 \item[i] if $y\geq x^{*},$ then $|y-x^{*}|\leq |y-x|$; 
\item[ii] if $x\leq y<x^{*},$ then $|y-x^{*}|\leq h^{\nu/2}$;
\item[iii] if $y<x,$ then $|y-x^{*}|\leq |x-y|+h^{\nu/2}$.
\end{description}
In summary, we derive  the bound
\[
d_1\leq C[ |x-y|+h^{\nu/2}]\bigg[1+\|u_0\|_{W^{2,2}(\Omega)}+\bigg(\underset{t\in[0,T]}{\sup}I_{t}^{\nu}(
\|g\|^{2}_{L_{2}(\Omega)})(t)\bigg)^{1/2}+\underset{t\in[0,T]}{\sup}
\|g\|_{L_{2}(\Omega)}\bigg].
\]

\noindent$\bullet$ Concerning the estimate of $d_2$, we have to evaluate the left-hand side of \eqref{6.17}. To this end, exploiting (3.5.4) in \cite{KST} and \cite[Chapter 1, Corollary 2]{SKM} arrives at the easily verified relations
\[
u(z,t+h)-u(z,t)=I_{t+h}^{\nu}\mathbf{D}_{t+h}^{\nu}u(z,t+h)-I_{t}^{\nu}\mathbf{D}_{t}^{\nu}u(z,t),
\]
and, therefore,
\begin{equation}\label{6.19}
|u(x^{*},t+h)-u(x^{*},t)|\leq C h^{\nu/2}\underset{[0,T_0]}{\sup}\bigg|\int_{x}^{x+h^{\nu/2}}\mathbf{D}_{t}^{\nu}u(z,s)dz\bigg|.
\end{equation}
Thus, we are left to evaluate the term $\int_{x}^{x+h^{\nu/2}}\mathbf{D}_{t}^{\nu}u(z,s)dz$. Appealing to (i) in Proposition \ref{p5.1*} provides the representation
\begin{equation}\label{6.20}
\mathbf{D}_{t}u=\sum_{j=1}^{3}\mathfrak{D}_{j}u+\varrho_{0}\mathbf{D}_{t}^{\nu}u,
\end{equation}
where we set
\begin{align*}
\mathfrak{D}_{1}u&=\sum_{i=1}^{M}\varrho_{i}\mathbf{D}_{t}^{\nu_i}u-\sum_{j=1}^{N}\gamma_{j}\mathbf{D}_{t}^{\mu_j}u,\quad
\mathfrak{D}_{2}u=u_0\{\mathbf{D}_{t}^{\nu}\varrho_{0}+\sum_{i=1}^{M}\mathbf{D}_{t}^{\nu_{i}}\varrho_{i}-\sum_{j=1}^{N}\mathbf{D}_{t}^{\mu_j}\gamma_{j}\},\\
\mathfrak{D}_{3}u&=\frac{\nu}{\Gamma(1-\nu)}\int_{0}^{t}(t-s)^{-1-\nu}[\varrho_{0}(t)-\varrho_{0}(s)][u(x,s)-u_0(x)]ds\\
&
+
\sum_{i=1}^{M}\frac{\nu_{i}}{\Gamma(1-\nu_{i})}\int_{0}^{t}(t-s)^{-1-\nu_{i}}[\varrho_{i}(t)-\varrho_{i}(s)][u(x,s)-u_0(x)]ds\\
&
-
\sum_{j=1}^{N}\frac{\mu_{j}}{\Gamma(1-\mu_{j})}\int_{0}^{t}(t-s)^{-1-\mu_{j}}[\gamma_{j}(t)-\gamma_{j}(s)][u(x,s)-u_0(x)]ds.
\end{align*}
Using representation \eqref{6.20} and the equation in \eqref{6.1}, we end up with
\begin{equation}\label{6.21}
\bigg|\int_{x}^{x+h^{\nu/2}}\mathbf{D}_{s}^{\nu}u(z,s)dz\bigg|\leq\sum_{j=1}^{6}|\mathcal{G}_{j}(u)|,
\end{equation}
where
\begin{align*}
\mathcal{G}_{1}u&=\varrho_{0}^{-1}\int_{x}^{x+h^{\nu/2}}\mathcal{L}_{1}u(z,s)dz,\qquad \qquad
\mathcal{G}_{2}u=\varrho_{0}^{-1}\int_{x}^{x+h^{\nu/2}}(\mathcal{K}*\mathcal{L}_{2}u)(z,s)dz,\\
\mathcal{G}_{3}u&=\varrho_{0}^{-1}\int_{x}^{x+h^{\nu/2}}[\lambda f(u)-g(z,s)]dz,\quad 
\mathcal{G}_{4}u=\varrho_{0}^{-1}\int_{x}^{x+h^{\nu/2}}\mathfrak{D}_{3}u(z,s)dz,\\
\mathcal{G}_{5}u&=\varrho_{0}^{-1}\int_{x}^{x+h^{\nu/2}}\mathfrak{D}_{2}u(z,s)dz,\qquad\qquad 
\mathcal{G}_{6}u=\varrho_{0}^{-1}\int_{x}^{x+h^{\nu/2}}\mathfrak{D}_{1}u(z,s)dz.
\end{align*}
Hence, we are left to treat each $\mathcal{G}_{i}u$.

\noindent$\bullet$ Taking into account \eqref{6.9} and \eqref{6.11}, assumptions \textbf{h2-h3}, and easily verified relation
\[
\int_{x}^{x+h^{\nu/2}}a_{2}(z,s)\frac{\partial^{2}u}{\partial z^{2}}(z,s)dz=
\int_{x}^{x+h^{\nu/2}}\frac{\partial }{\partial z}\bigg[a_2(z,s)\frac{\partial u}{\partial z}(z,s)\bigg]dz
-
\int_{x}^{x+h^{\nu/2}}\frac{\partial u}{\partial z}(z,s)\frac{\partial a_{2}}{\partial z}(z,s)dz,
\]
we immediately deduce that
\[
|\mathcal{G}_{1}u|\leq \delta_{1}^{-1}[1+h^{\nu/2}]C_{0}\bigg[1+\|u_0\|_{W^{2,2}(\Omega)}+\bigg(\underset{t\in[0,T]}{\sup}I_{t}^{\nu}(
\|g\|^{2}_{L_{2}(\Omega)})(t)\bigg)^{1/2}+\underset{t\in[0,T]}{\sup}
\|g\|_{L_{2}(\Omega)}\bigg].
\]

\noindent$\bullet$  It is worth noting that, using the similar arguments and making  assumption \textbf{h4} on the kernel $\mathcal{K}$, we get  the inequality
\[
|\mathcal{G}_{2}u|\leq\frac{C_0(1+h^{\nu/4})\|\mathcal{K}\|_{L_1(0,T)}}{\delta_{1}}\bigg[1+\|u_0\|_{W^{2,2}(\Omega)}+\bigg(\underset{t\in[0,T]}{\sup}I_{t}^{\nu}(
\|g\|^{2}_{L_{2}(\Omega)})(t)\bigg)^{1/2}+\underset{t\in[0,T]}{\sup}
\|g\|_{L_{2}(\Omega)}\bigg].
\]

\noindent$\bullet$ Assumptions \eqref{4.1} on $f(u)$, \textbf{h5} on $g$ and estimate \eqref{6.9} arrive at
\[
|\mathcal{G}_{3}u|\leq  \delta_{1}^{-1}h^{\nu/2}[1+L]\bigg[1+\|u_0\|_{W^{2,2}(\Omega)}+\bigg(\underset{t\in[0,T]}{\sup}I_{t}^{\nu}(
\|g\|^{2}_{L_{2}(\Omega)})(t)\bigg)^{1/2}+\underset{t\in[0,T]}{\sup}
\|g\|_{L_{2}(\Omega)}\bigg].
\]

\noindent$\bullet$ In light of the regularity of the coefficients $\varrho_0,\varrho_{i},\gamma_{j} $ and the function $u_0$ (see \textbf{h3} and \textbf{h5}), we use \eqref{6.9} and obtain the bound
\[
|\mathcal{G}_{4}u|+|\mathcal{G}_{5}u|\leq  C\delta_{1}^{-1}h^{\nu/2}\bigg[1+\|u_0\|_{W^{2,2}(\Omega)}+\bigg(\underset{t\in[0,T]}{\sup}I_{t}^{\nu}(
\|g\|^{2}_{L_{2}(\Omega)})(t)\bigg)^{1/2}+\underset{t\in[0,T]}{\sup}
\|g\|_{L_{2}(\Omega)}\bigg]
\]
with the constant $C$ depending only on $T,\nu,\mu_{j},\nu_{i},N,M$ and the norms of the coefficients.

\noindent$\bullet$ Concerning $\mathcal{G}_{6}u,$ we exploit the representation (10.34) in \cite{KPSV} and have
\[
|\mathcal{G}_{6}u|\leq \sum_{i=1}^{M}\varrho_{i}I_{s}^{\nu-\nu_{i}}\bigg|\int_{x}^{x+h^{\nu/2}}\mathbf{D}_{s}^{\nu}u(z,s)dz\bigg|
+
\sum_{j=1}^{N}\gamma_{j}I_{s}^{\nu-\mu_{j}}\bigg|\int_{x}^{x+h^{\nu/2}}\mathbf{D}_{s}^{\nu}u(z,s)dz\bigg|.
\]
Finally, gathering  these estimates with \eqref{6.21}, we obtain
\begin{align*}
\bigg|\int_{x}^{x+h^{\nu/2}}\mathbf{D}_{s}^{\nu}u(z,s)dz\bigg|&\leq C[1+h^{\frac{\nu}{2}}]\bigg[1+\|u_0\|_{W^{2,2}(\Omega)}+\bigg(\underset{t\in[0,T]}{\sup}I_{t}^{\nu}(
\|g\|^{2}_{L_{2}(\Omega)})(t)\bigg)^{1/2}+\underset{t\in[0,T]}{\sup}
\|g\|_{L_{2}(\Omega)}\bigg]\\
&
+\bigg[\sum_{i=1}^{M}\varrho_{i}I_{s}^{\nu-\nu_{i}}
+
\sum_{j=1}^{N}\gamma_{j}I_{s}^{\nu-\mu_{j}}\bigg]\bigg|\int_{x}^{x+h^{\nu/2}}\mathbf{D}_{s}^{\nu}u(z,s)dz\bigg|.
\end{align*}
Then, Gronwall-type inequality (4.4) \cite{KPV3} tells us that
\begin{equation}\label{6.22}
\bigg|\int_{x}^{x+h^{\nu/2}}\mathbf{D}_{s}^{\nu}u(z,s)dz\bigg|\leq C[1+h^{\frac{\nu}{2}}]\bigg[1+\|u_0\|_{W^{2,2}(\Omega)}+\bigg(\underset{t\in[0,T]}{\sup}I_{t}^{\nu}(
\|g\|^{2}_{L_{2}(\Omega)})(t)\bigg)^{1/2}+\underset{t\in[0,T]}{\sup}
\|g\|_{L_{2}(\Omega)}\bigg],
\end{equation}
where $C$ is independent  of $T_0,$ $\lambda, h$ and depends only on $\nu,\mu_{i},\nu_{j},N,M,$ $\|\mathcal{K}\|_{L_1(0,T)}$, and the corresponding norms of the coefficients.

As a result, \eqref{6.22} and \eqref{6.19} lead to
\begin{align*}
d_2&=|u(x^{*},t+h)-u(x^{*},t)|
\\
&\leq Ch^{\frac{\nu}{2}}[1+h^{\frac{\nu}{2}}]\bigg[1+\|u_0\|_{W^{2,2}(\Omega)}+\bigg(\underset{t\in[0,T]}{\sup}I_{t}^{\nu}(
\|g\|^{2}_{L_{2}(\Omega)})(t)\bigg)^{1/2}+\underset{t\in[0,T]}{\sup}
\|g\|_{L_{2}(\Omega)}\bigg],
\end{align*} 
or coming to \eqref{6.18} and, taking into account the estimate of $d_1$, we achieve
\[
\Delta_{h}u\leq C\{h^{\nu}+h^{\nu/2}+|x-y|\} 
\bigg[1+\|u_0\|_{W^{2,2}(\Omega)}+\bigg(\underset{t\in[0,T]}{\sup}I_{t}^{\nu}(
\|g\|^{2}_{L_{2}(\Omega)})(t)\bigg)^{1/2}+\underset{t\in[0,T]}{\sup}
\|g\|_{L_{2}(\Omega)}\bigg].
\]
It is apparent that (as we wrote above) this inequality provides  the desired estimate.

In summary, this bound completes the proof of \eqref{6.2}, and therefore, the verification \eqref{6.3} is finished.
\qed

%%%%%%%%%%%%%%%%%%%%%%%%%%%%%%%%%%%%%%%%%%%%%%%%%%%%%%%%%%%%%%%%%%%%%%
 
%%%%%%%%%%%%%%%%%%%%%%%%%%%%%%%%%%%%%%%%%%%%%%%%%%%%%%%%%%%%%%%%%%%%%%

\subsection{Verification of \eqref{6.4} if $t\in[0,T_{0}]$}
\label{s6.4}
In light of estimates \eqref{6.9} and \eqref{6.10}, we conclude that \eqref{6.4} with $T=T_{0}$ will be proved if we obtain the inequality
\begin{align}\label{6.23}\notag
&\|\mathbf{D}_{t}^{\nu}u\|_{L_{2}(\Omega_{T_0})}+\sum_{i=1}^{M}\|\mathbf{D}_{t}^{\nu_{i}}u\|_{L_{2}(\Omega_{T_0})}+
\sum_{j=1}^{N}\|\mathbf{D}_{t}^{\mu_{j}}u\|_{L_{2}(\Omega_{T_0})}
\\
&\leq C\bigg[1+\|u_0\|_{W^{2,2}(\Omega)}+\bigg(\underset{t\in[0,T]}{\sup}I_{t}^{\nu}(
\|g\|^{2}_{L_{2}(\Omega)})(t)\bigg)^{1/2}\bigg].
\end{align}
To this end, we use again representation \eqref{6.20} and the easily verified estimates
\begin{align}\label{6.24}\notag
\|\mathbf{D}_{t}u\|_{L_{2}(\Omega_{T_0})}&\leq C \bigg[1+\|u_0\|_{W^{2,2}(\Omega)}+\bigg(\underset{t\in[0,T]}{\sup}I_{t}^{\nu}(
\|g\|^{2}_{L_{2}(\Omega)})(t)\bigg)^{1/2}\bigg],\\\notag
\|\mathbf{D}^{\nu}_{t}u\|_{L_{2}(\Omega_{T_0})}&\leq\frac{C}{\delta_{1}}\|\mathbf{D}_{t}u-\sum_{j=1}^{3}\mathfrak{D}_{j}u\|_{L_{2}(\Omega_{T_0})}
\leq\frac{C}{\delta_{1}}[\|\mathbf{D}_{t}u\|_{L_{2}(\Omega_{T_0})}+\sum_{j=1}^{3}\|\mathfrak{D}_{j}u\|_{L_{2}(\Omega_{T_0})}],\\\notag
\|\mathcal{D}_{2}u\|_{L_{2}(\Omega_{T_0})}+\|\mathfrak{D}_{3}u\|_{L_{2}(\Omega_{T_0})}&\leq C \bigg[1+\|u_0\|_{W^{2,2}(\Omega)}+\bigg(\underset{t\in[0,T]}{\sup}I_{t}^{\nu}(
\|g\|^{2}_{L_{2}(\Omega)})(t)\bigg)^{1/2}\bigg]\\
&
\times
[\|\varrho_{0}\|_{\C^{1}([0,T_0])}+\sum_{i=1}^{M}\|\varrho_{i}\|_{\C^{1}([0,T_0])}
+\sum_{j=1}^{N}\|\gamma_{j}\|_{\C^{1}([0,T_0])}
]
\end{align}
with the constants being independent of $\lambda$ and $T_0$.
We remark that the last inequality in \eqref{6.24} is a simple consequence of assumption \textbf{h3} and estimates \eqref{6.9}
 and \eqref{6.11}.

Besides, collecting all these estimates provides the bound
\begin{align}\label{6.25}\notag
\|\mathbf{D}^{\nu}_{t}u\|_{L_{2}(\Omega_{T_0})}&\leq C\bigg[1+\|u_0\|_{W^{2,2}(\Omega)}+\bigg(\underset{t\in[0,T]}{\sup}I_{t}^{\nu}(
\|g\|^{2}_{L_{2}(\Omega)})(t)\bigg)^{1/2}
\bigg]\\
&+
C_{3}[\sum_{i=1}^{M}\|\mathbf{D}_{t}^{\nu_i}u\|_{L_{2}(\Omega_{T_0})}+\sum_{j=1}^{N}\|\mathbf{D}_{t}^{\mu_j}u\|_{L_{2}(\Omega_{T_0})}].
\end{align}
As a result, in virtue of the last estimate, we are left to evaluate   $\|\mathbf{D}_{t}^{\nu_i}u\|_{L_{2}(\Omega_{T_0})}$ and $\|\mathbf{D}_{t}^{\mu_j}u\|_{L_{2}(\Omega_{T_0})}$. To this end, we are exploit \cite[Theorem 2.2]{Y} (taking into account Definition \ref{d2.2} and  estimate \eqref{2.1})  and embedding Theorem 1.4.3.3 in \cite{Gr}. Indeed, setting
\[
\mu_0=\frac{1}{4}\min\{\mu_1;\nu-\mu_{N}\}\quad\text{and}\quad \nu_0=\frac{1}{4}\min\{\nu_1;\nu-\nu_{M}\},
\]
and choosing $\bar{\nu}$ satisfying inequalities
\[
0<\bar{\nu}<\min\{\nu-\mu_{N}-\mu_{0};\nu-\nu_{M}-\nu_{0}\},
\]
we take advantage of  \cite[Theorem 2.2]{Y}, \cite[Propositions 3 and 7]{Y2} and then embedding Theorem 1.4.3.3 in \cite{Gr} to conclude that
\begin{align}\label{6.26}\notag
&C_{\nu}\|u-u_0\|_{W^{\nu-\bar{\nu}}((0,T_0),L_2(\Omega))}\leq \|\mathbf{D}_{t}^{\nu}u\|_{L_2(\Omega_{T_0})},\\\notag
&\sum_{i=1}^{M}
\|\mathbf{D}_{t}^{\nu_i}u\|_{L_{2}(\Omega_{T_0})}+\sum_{j=1}^{N}\|\mathbf{D}_{t}^{\mu_j}u\|_{L_{2}(\Omega_{T_0})}\leq
\sum_{i=1}^{M}
C_{\nu_{i}}\|u-u_0\|_{W^{\nu_{i}+\nu_{0}}((0,T_0),L_2(\Omega))}\\\notag&
+\sum_{j=1}^{N}C_{\mu_{j}}\|u-u_0\|_{W^{\mu_{j}+\mu_{0}}((0,T_0),L_2(\Omega))}\\\notag
&
\leq
\varepsilon[\sum_{i=1}^{M}
C_{\nu_{i}}+\sum_{j=1}^{N}C_{\mu_{j}}]\|u-u_0\|_{W^{\nu-\bar{\nu}}((0,T_0),L_2(\Omega))}\\
&+
C\bigg\{\sum_{j=1}^{N}C_{\mu_{j}}\varepsilon^{-\frac{\mu_{j}+\mu_0}{\nu-\bar{\nu}-\mu_{j}-\mu_0}}
+
\sum_{i=1}^{M}C_{\nu_{i}}\varepsilon^{-\frac{\nu_{i}+\nu_0}{\nu-\bar{\nu}-\nu_{i}-\nu_0}}
\bigg\}\|u-u_{0}\|_{L_{2}(\Omega_{T_0})},
\end{align}
where $C_{\nu}=C(\nu,\bar{\nu}),$ $C_{\nu_{i}}=C(\nu_i,\nu_0),$ $C_{\mu_{j}}=C(\mu_{j},\mu_{0}),$ are positive constants defined in (i) of Theorem 2.2 \cite{Y}.

Gathering \eqref{6.25} and \eqref{6.26} and setting
\[
\varepsilon=\frac{C_\nu}{4C_3[\sum_{j=1}^{N}C_{\mu_{j}}+\sum_{i=1}^{M}C_{\nu_{i}}]},
\]
we end up with the bound
\begin{equation}\label{6.27}
\|u-u_{0}\|_{W^{\nu-\bar{\nu}}((0,T_0),L_{2}(\Omega))}\leq C \bigg[1+\|u_0\|_{W^{2,2}(\Omega)}+\bigg(\underset{t\in[0,T]}{\sup}I_{t}^{\nu}(
\|g\|^{2}_{L_{2}(\Omega)})(t)\bigg)^{1/2}\bigg]
\end{equation}
with the positive quantity $C$ being independent of $\lambda$ and $T_0$. We notice that, we exploited \eqref{6.9} to control the term $\|u-u_{0}\|_{L_{2}(\Omega_{T_0})}$ in the right-hand side of the second inequality in \eqref{6.26}.

After that, \eqref{6.27} together with the second inequality in \eqref{6.26} arrive at the estimate
\[
\sum_{i=1}^{M}\|\mathbf{D}_{t}^{\nu_i}u\|_{L_{2}(\Omega_{T_0})}+\sum_{j=1}^{N}\|\mathbf{D}_{t}^{\mu_j}u\|_{L_{2}(\Omega_{T_0})}\leq
C \bigg[1+\|u_0\|_{W^{2,2}(\Omega)}+\bigg(\underset{t\in[0,T]}{\sup}I_{t}^{\nu}(
\|g\|^{2}_{L_{2}(\Omega)})(t)\bigg)^{1/2}\bigg],
\]
which in turn (see \eqref{6.25}) provides the desired bound 
\[
\|D_{t}^{\nu}u\|_{L_{2}(\Omega_{T_0})}\leq C  \bigg[1+\|u_0\|_{W^{2,2}(\Omega)}+\bigg(\underset{t\in[0,T]}{\sup}I_{t}^{\nu}(
\|g\|^{2}_{L_{2}(\Omega)})(t)\bigg)^{1/2}\bigg].
\]
This finishes the proof of \eqref{6.23} and, hence, \eqref{6.4} with $T=T_{0}$.
\qed

%%%%%%%%%%%%%%%%%%%%%%%%%%%%%%%%%%%%%%%%%%%%%%%%%%%%%%%%%%%%%%%%%%%%%%
 
%%%%%%%%%%%%%%%%%%%%%%%%%%%%%%%%%%%%%%%%%%%%%%%%%%%%%%%%%%%%%%%%%%%%%%

\subsection{Conclusion of the proof of Lemma \ref{l6.1}}
\label{s6.5}
In light of the results described in Sections \ref{s6.1}-\ref{s6.4}, we are left to extend estimates \eqref{6.2}-\eqref{6.4} on the whole time interval $[0,T]$. To this end, we first discuss the technique which allows us to extend these estimates to the interval $[0,3T_0/2]$. Then, recasting this procedure a finite number of times until the entire $[0,T]$ is exhausted.

First, we need in new function
\begin{equation}\label{6.28}
\mathcal{U}(x,t)=\xi(t)u(x,t),
\end{equation}
where $\xi(t)$ is defined in \eqref{5.2***} with $T_3=T_{0}$, and $u(x,t)$ is a  solution of \eqref{6.1}.

Next statement describes the main properties of this function.
\begin{corollary}\label{c6.1}
The function $\mathcal{U}(x,t)$ solves the problem \eqref{6.1} in $\bar{\Omega}_{T_0/2}$ and satisfies estimates:
\begin{align*}
&\|\mathcal{U}\|_{\C^{2+\alpha,\frac{2+\alpha}{2}\nu}(\bar{\Omega}_{T})}+\sum_{j=1}^{N}\|\mathbf{D}_{t}^{\mu_j}\mathcal{U}\|_{\C^{\alpha,\frac{\alpha\nu}{2}}(\bar{\Omega}_{T})}+\sum_{i=1}^{M}\|\mathbf{D}_{t}^{\nu_i}\mathcal{U}\|_{\C^{\alpha,\frac{\alpha\nu}{2}}(\bar{\Omega}_{T})}
\leq C[1+\|u_0\|_{\C^{2+\alpha}(\bar{\Omega})}+\|g\|_{\C^{\alpha,\frac{\nu\alpha}{2}}(\bar{\Omega}_{T})}],\\
&\|\mathcal{U}\|_{\C([0,T],\C^{1}(\bar{\Omega}))}+\langle\mathcal{U}\rangle_{t,\Omega_{T}}^{(\nu/2)}\leq
C\bigg[1+\bigg(\underset{t\in[0,T]}{\sup}I_{t}^{\nu}(
\|g\|^{2}_{L_{2}(\Omega)})(t)\bigg)^{1/2}+\underset{t\in[0,T]}{\sup}
\|g\|_{L_{2}(\Omega)}+\|u_0\|_{W^{2,2}(\Omega)}\bigg],\\
&\|\mathcal{U}\|_{L_{2}((0,T),W^{2,2}(\Omega))}+\|\mathbf{D}_{t}^{\nu}\mathcal{U}\|_{L_{2}(\Omega_{T})}+\sum_{j=1}^{N}\|\mathbf{D}_{t}^{\mu_j}\mathcal{U}\|_{L_{2}(\Omega_{T})}+\sum_{i=1}^{M}\|\mathbf{D}_{t}^{\nu_i}\mathcal{U}\|_{L_{2}(\Omega_{T})}
+\bigg(\underset{t\in[0,T]}{\sup}I_{t}^{\nu}\|\mathcal{U}_{xx}\|^{2}_{L_2(\Omega)}\bigg)^{\frac{1}{2}}
\\
&
\leq
C\bigg[1+\bigg(\underset{t\in[0,T]}{\sup}I_{t}^{\nu}(
\|g\|^{2}_{L_{2}(\Omega)})(t)\bigg)^{1/2}+\|u_0\|_{W^{2,2}(\Omega)}\bigg],\\
&
\|\mathbf{D}_{t}\mathcal{U}-\mathcal{L}_{1}\mathcal{U}-\mathcal{K}*\mathcal{L}_{2}\mathcal{U}\|_{\C^{\alpha,\frac{\alpha\nu}{2}}(\bar{\Omega}_{T})}\leq 
C\bigg[1+\bigg(\underset{t\in[0,T]}{\sup}I_{t}^{\nu}(
\|g\|^{2}_{L_{2}(\Omega)})(t)\bigg)^{\frac{1}{2}}+\|u_0\|_{W^{2,2}(\Omega)}+\|g\|_{\C^{\alpha,\frac{\nu\alpha}{2}}(\bar{\Omega}_{T})}\bigg],\\
&
\underset{t\in[0,T]}{\sup}I_{t}^{\nu}(
\|\mathbf{D}_{t}\mathcal{U}-\mathcal{L}_{1}\mathcal{U}-\mathcal{K}*\mathcal{L}_{2}\mathcal{U}\|^{2}_{L_{2}(\Omega)})\leq
C\bigg[1+\underset{t\in[0,T]}{\sup}I_{t}^{\nu}(
\|g\|^{2}_{L_{2}(\Omega)})(t)+\|u_0\|^{2}_{W^{2,2}(\Omega)}\bigg]
\end{align*}
Here the positive constant $C$ depends only on $T,\mu_{j},\nu_{i},\nu$ and the corresponding norms of the coefficients of the operators $\mathcal{L}_{1},$ $\mathcal{L}_{2}$ and $\mathbf{D}_{t}$.
\end{corollary}
\begin{proof}
Clearly,  the first three inequalities are simple consequences of definition \eqref{6.28}, Lemma \ref{l5.0} (where $T_3=T_0$) and estimates \eqref{6.2}-\eqref{6.4} with $T=T_0$, which are proved in Sections \ref{s6.1}-\ref{s6.4}.

As for the fourth bound,
taking into account that $u$ solves \eqref{6.1}, we have the representation
\[
\mathbf{D}_{t}\mathcal{U}-\mathcal{L}_{1}\mathcal{U}-\mathcal{K}*\mathcal{L}_{2}\mathcal{U}=\sum_{j=1}^{3}s_{j},
\]
where
\[
s_1=\xi[g(x,t)-\lambda f(u)],\quad
s_2=\mathbf{D}_{t}(\xi u)-\xi\mathbf{D}_{t}u,\quad
s_3=\xi(\mathcal{K}*\mathcal{L}_2u)(t)-(\mathcal{K}*\xi\mathcal{L}_{2}u)(t).
\]
At this point, we evaluate each $s_j$, separately. 

\noindent $\bullet$ Appealing to \textbf{h5} and \textbf{h6 (i)}, we derive
\[
\|s_1\|_{\C^{\alpha,\alpha\nu/2}(\bar{\Omega}_{T})}\leq C[1+\|g\|_{\C^{\alpha,\alpha\nu/2}(\bar{\Omega}_{T})}
+\|u\|_{\C^{\alpha,\alpha\nu/2}(\bar{\Omega}_{T_0})}],
\]
where the constant $C$ depends only on $\|\xi\|_{\C^{1}([0,T])}$, $L,C_{\rho}$. 
Then exploiting inequality \eqref{6.2}, we end up with
\[
\|s_1\|_{\C^{\alpha,\alpha\nu/2}(\bar{\Omega}_{T})}\leq C[1+\|g\|_{\C^{\alpha,\alpha\nu/2}(\bar{\Omega}_{T})}
+
\|u_0\|_{W^{2,2}(\Omega)}
].
\]

\noindent $\bullet$ Collecting the representation of the operator $\mathbf{D}_{t}$ with statement (v) in Lemma \ref{c5.0} and the bound \eqref{6.2} with $T=T_0$ provides the estimate
\begin{align*}
\|s_2\|_{\C^{\alpha,\alpha\nu/2}(\bar{\Omega}_{T})}&\leq C[1+\|u_0\|_{\C^{\alpha}(\bar{\Omega})}
+\|u\|_{\C^{\alpha,\nu/2}(\bar{\Omega}_{T_0})}]\\
&
\leq
C\bigg[\|u_0\|_{W^{2,2}(\Omega)}+1+\bigg(\underset{t\in[0,T]}{\sup}I_{t}^{\nu}(
\|g\|^{2}_{L_{2}(\Omega)})(t)\bigg)^{\frac{1}{2}}\bigg]
\end{align*}
with the constant $C$ being independent of $\lambda$ and $T_0$.

\noindent $\bullet$ Standard calculations allow us to rewrite $s_3$ in the form
\[
s_3=
\begin{cases}
\xi(t)(\mathcal{K}*\mathcal{L}_{2}u)(t)-(\mathcal{K}*\xi\mathcal{L}_{2}u)(t),\qquad\qquad t\in(0,3T_{0}/2],\\
\,\\
\int_{0}^{3T_0/2}\mathcal{K}(t-\tau)\xi(\tau)\mathcal{L}_{2}u(x,\tau)d\tau,\quad t>3T_{0}/2.
\end{cases}
\]
After that, \cite[Lemma 4.1]{KPV1}, Lemma \ref{l5.0} and assumptions of \textbf{h2, h4} lead to the estimate
\[
\|s_3\|_{\C^{\alpha,\alpha\nu/2}(\bar{\Omega}_{T})}
\leq
C[\|u_0\|_{W^{2,2}(\Omega)}+1+\|I_{t}^{\nu}\|g\|^{2}_{L_{2}(\Omega)}\|^{1/2}_{\C([0,T])}].
\]
As a result, gathering all estimates of $s_{i}$ leads to the searched estimate.

Finally, we left remark that the verification of the last inequality in this corollary is carried out with the similar arguments and with exploiting the second and the third inequalities of this claim.
This completes the proof of Corollary \ref{c6.1}.
\end{proof}
Now we introduce new unknown function
\begin{equation}\label{6.29}
\mathcal{V}=u-\mathcal{U}
\end{equation}
which solves the problem
\[
\begin{cases}
\mathbf{D}_{t}\mathcal{V}-\mathcal{L}_1\mathcal{V}-\mathcal{K}*\mathcal{L}_2\mathcal{V}=g^{*}(x,t)-\lambda f^{*}(\mathcal{V})\quad \text{in}\quad \Omega_{3T_0/2},\\
\mathcal{V}(x,0)=0\quad\text{in}\quad\bar{\Omega},\\
\mathcal{V}(x,t)=0\quad\text{on}\quad \partial\Omega_{3T_0/2},
\end{cases}
\]
where we put
\[
g^{*}(x,t)=g(x,t)-\mathbf{D}_{t}\mathcal{U}+\mathcal{L}_1\mathcal{U}+\mathcal{K}*\mathcal{L}_2\mathcal{U},\quad
f^{*}(\mathcal{V})=f(\mathcal{V}+\mathcal{U}).
\]
The definition of the function $\mathcal{U}$ and Corollary \ref{c6.1} readily yield
\begin{align*}
\|g^{*}\|_{\C^{\alpha,\alpha\nu/2}(\bar{\Omega}_{3T_0/2})}&\leq C[1+\|u_0\|_{W^{2,2}(\Omega)}+\|g\|_{\C^{\alpha,\alpha\nu/2}(\bar{\Omega}_{T})}],\\
g^{*}(x,t)-\lambda f^{*}(\mathcal{V})&=0\quad\text{if}\quad t\in[0,T_0/2],\quad x\in\bar{\Omega},
\end{align*}
and $ f^{*}(\mathcal{V})$ meets requirements \textbf{h6 (i)}. Besides, the last equality here means that compatibility conditions hold. Finally, we introduce the new time-variable
\[
\sigma=t-\frac{T_{0}}{2},\quad \sigma\in[-T_0/2,T_0]
\]
and recast arguments of the end of Section 6.3 in \cite{PSV6}. Thus, we deduce
\begin{equation}\label{6.30}
\begin{cases}
\bar{\mathbf{D}}_{\sigma}\bar{\mathcal{V}}-\bar{\mathcal{L}}_1\bar{\mathcal{V}}-\mathcal{K}*\bar{\mathcal{L}}_2\bar{\mathcal{V}}=\bar{g}^{*}(x,\sigma)-\lambda \bar{f}^{*}(\bar{\mathcal{V}})\quad \text{in}\quad \Omega_{T_0},\\
\bar{\mathcal{V}}(x,0)=0\quad\text{in}\quad\bar{\Omega},\\
\bar{\mathcal{V}}(x,\sigma)=0\quad\text{on}\quad \partial\Omega_{T_0},
\end{cases}
\end{equation}
and $\bar{\mathcal{V}}(x,\sigma)=0$ if $\sigma\in[-T_0/2,0]$.
Here we denote
\[
\bar{\mathcal{V}}(x,\sigma)=\mathcal{V}(x,\sigma+T_0/2),\quad \bar{g}^{*}(x,\sigma)=g^{*}(x,\sigma+T_0/2)\quad
\bar{f}^{*}(\bar{\mathcal{V}})=f^{*}(\mathcal{V})|_{t=\sigma+T_0/2},
\]
and we call $\bar{\mathcal{L}}_{i}$, $\bar{\mathbf{D}}_{\sigma}$ the operators $\mathcal{L}_{i}$ and $\mathbf{D}_{\sigma}$, respectively, with the bar coefficients. It is easy to verify that the coefficients $\bar{\mathcal{L}}_{i}$, $\bar{\mathbf{D}}_{\sigma}$  and the function $\bar{g}^{*}$, $\bar{f}^{*}$ meet the requirements of Lemma \ref{l6.1}.

Then, we repeat the arguments of Sections \ref{s6.1}-\ref{s6.4} in the case of problem \eqref{6.30} and obtain estimates \eqref{6.2}-\eqref{6.4} to the function $\bar{\mathcal{V}}$. In fine, taking into account representation \eqref{6.29} and Corollary \ref{c6.1}, we extend these estimates to the interval $[0,3T_0/2]$. Therefore, after repeating this procedure finite times, we get \eqref{6.2}-\eqref{6.4} for all $t\in[0,T]$.
\qed

%%%%%%%%%%%%%%%%%%%%%%%%%%%%%%%%%%%%%%%%%%%%%%%%%%%%%%%%%%%%%%%%%%%%%%
 
%%%%%%%%%%%%%%%%%%%%%%%%%%%%%%%%%%%%%%%%%%%%%%%%%%%%%%%%%%%%%%%%%%%%%%

\section{Proof of Theorem \ref{t4.1}}
\label{s7}

\noindent Here we proceed with a detailed proof of this theorem in the case of \textbf{DBC} \eqref{1.2} and $f(u)$ satisfying \textbf{h6 (i)}.
Another cases are analyzed with the similar arguments and left to the readers.

First of all, we reduce problem \eqref{1.1}, \eqref{1.2}, \eqref{1.3} to the problem with homogeneous initial and boundary conditions. To this end, we apply \cite[Remark 3.1]{KPV3} and Lemma \ref{l5.2} (in this art) to the linear problem for the unknown function $\mathfrak{U}=\mathfrak{U}(x,t):\Omega_{T}\to\R:$
\[
\begin{cases}
\mathbf{D}_{t}\mathfrak{U}-\mathcal{L}_1\mathfrak{U}-\mathcal{K}*\mathcal{L}_2\mathfrak{U}=g(x,t)-f(u_0)\quad \text{in}\, \Omega_{T},\\
\mathfrak{U}(x,0)=u_0(x)\quad\text{in}\quad\bar{\Omega},\\
\mathfrak{U}(x,t)=\psi(x,t)\quad\text{on}\quad \partial\Omega_{T},
\end{cases}
\]
and obtain the one-valued global classical solvability of this problem $\mathfrak{U}\in\C^{2+\alpha,\frac{2+\alpha}{2}\nu}(\bar{\Omega}_{T})$ satisfying the bound
\begin{align*}
&\|\mathfrak{U}\|_{\C^{2+\alpha,\frac{2+\alpha}{2}\nu}(\bar{\Omega}_{T})}+\sum_{i=1}^{M}\|\mathbf{D}_{t}^{\nu_{i}}\mathfrak{U}\|_{\C^{\alpha,\frac{\nu\alpha}{2}}(\bar{\Omega}_{T})}
+\sum_{j=1}^{N}\|\mathbf{D}_{t}^{\mu_{j}}\mathfrak{U}\|_{\C^{\alpha,\frac{\nu\alpha}{2}}(\bar{\Omega}_{T})}\\
&\leq
C[1+\|u_0\|_{\C^{2+\alpha}(\bar{\Omega})}+
\|g\|_{\C^{\alpha,\frac{\nu\alpha}{2}}(\bar{\Omega}_{T})}+
\|\psi\|_{\C^{2+\alpha,\frac{2+\alpha}{2}\nu}(\partial\Omega_{T})}
]\\
&\equiv C\mathfrak{G}(u_0,f,\psi).
\end{align*}
Here we used assumption \textbf{h6 (i)} and \cite[Remark 3.1]{KPV3} to control the term $\|f(u_0)\|_{\C^{\alpha,\frac{\nu\alpha}{2}}(\bar{\Omega}_{T})}$. Then we look for a solution of the original problem \eqref{1.1}, \eqref{1.2}, \eqref{1.3} in the form
\[
u(x,t)=v(x,t)+\mathfrak{U}(x,t),
\]
where the unknown function $v$ solves the problem
\begin{equation}\label{7.1}
\begin{cases}
\mathbf{D}_{t}v-\mathcal{L}_1v-\mathcal{K}*\mathcal{L}_2v=G(x,t)- F(v)\quad \text{in}\quad \Omega_{T},\\
v(x,0)=0\quad\text{in}\quad\bar{\Omega},\\
v(x,t)=0\quad\text{on}\quad \partial\Omega_{T}.
\end{cases}
\end{equation}
Here we set
\[
G(x,t)=f(u_0)-f(\mathfrak{U}),\qquad F(v)=f(v+\mathfrak{U})-f(\mathfrak{U}).
\]
\begin{remark}\label{r7.1} Assumptions \textbf{h6 (i)} and the estimate of $\mathfrak{U}$ readily ensure the following relations to the functions $F$ and $G$:
\[
\|G\|_{\C^{\alpha,\frac{\nu\alpha}{2}}(\bar{\Omega}_{T})}\leq C\mathfrak{G}(u_0,f,\psi),
\]
and for all $v_{i}\in[-\rho,\rho]$ and $v\in\R$ there holds
\[
|F(v_1)-F(v_2)|\leq C_{\rho}|v_1-v_2|,\quad |F(v)|\leq L^{*}(1+|v|),
\]
with
\[
L^{*}=L(1+2\underset{\bar{\Omega}_{T}}{\sup}|\mathfrak{U}|)\leq L[1+2C\mathfrak{G}(u_0,f,\psi)].
\]
Moreover, the direct calculations and the properties of the function $\mathfrak{U}$ arrive at the equalities:
\[
G(x,0)=0\quad\text{for}\quad x\in\bar{\Omega},\quad F(0)=0\quad\text{for each}\quad (x,t)\in\bar{\Omega}_{T}.
\]
\end{remark}
Thus, we easily conclude that $F$ and $G$ meet requirements of Theorem \ref{t4.1} and, hence, we are left to prove this theorem in the case of problem \eqref{7.1}.

To this end, we rely on the so-called continuation arguments, in analogy to the case of semilinear problem to the subdiffusion equations with a single-term fractional derivative described in \cite[Section 5.2]{KPV3}. This approach deals with the analysis of the family of problems for $\lambda\in[0,1]$,
\begin{equation}\label{7.2}
\begin{cases}
\mathbf{D}_{t}v-\mathcal{L}_1v-\mathcal{K}*\mathcal{L}_2v=G(x,t)-\lambda F(v)\quad \text{in}\, \Omega_{T},\\
v(x,0)=0\quad\text{in}\quad\bar{\Omega},\\
v(x,t)=0\quad\text{on}\quad \partial\Omega_{T}.
\end{cases}
\end{equation}
Let $\Lambda$ be the set of those $\lambda$ for which \eqref{7.2} is solvable on $[0,T]$. For $\lambda=0$, \eqref{7.2} is a linear problem studied in detailed in \cite{PSV6} (see  also Lemma \ref{l5.2} here). Hence, keeping in mind assumptions \textbf{h1}-\textbf{h5}, \textbf{h7} and Remark \ref{r7.1}, we can apply Lemma \ref{l5.2} to \eqref{7.2} with $\lambda=0$ and obtain the global classical solvability in the corresponding classes. Therefore, $0\in\Lambda$. The next step demonstrates that $\Lambda$ is open and closed at the same time. To this end, we repeat step-by-step the arguments given in \cite[Section 5.2]{KPV3} and, exploiting Lemma \ref{l6.1} (where $f=F$, $g=G$, $u_0=0$), we complete the proof of Theorem \ref{t4.1}.\qed 

%%%%%%%%%%%%%%%%%%%%%%%%%%%%%%%%%%%%%%%%%%%%%%%%%%%%%%%%%%%%%%%%%%%%%%
 
%%%%%%%%%%%%%%%%%%%%%%%%%%%%%%%%%%%%%%%%%%%%%%%%%%%%%%%%%%%%%%%%%%%%%%

\section{Proof of Theorem \ref{t4.2}}
\label{s8}

\noindent Here we focus on the proof of this theorem in the case of \textbf{DBC} \eqref{1.2}. The case of \textbf{NBC} \eqref{1.2*} is treated with the similar arguments. We will follow the strategy consisting in two steps. In the first, we assume the existence of functions $U_{0,n}$  which approximate the initial data $u_{0}\in \overset{0}{W}\,^{1,2}(\Omega)\cap W^{2,2}(\Omega)$ and satisfy assumptions \textbf{h5 (i)}  and \textbf{h7}. Other words, the function $U_{0,n}\in\C_{0}^{\infty}(\bar{\Omega})$ has the properties:
\begin{equation}\label{8.2}
 U_{0,n}\quad\text{converges to}\,\, u_{0}\quad\text{in}\quad W^{2,2}(\Omega);\quad
\begin{cases}
 U_{0,n}(a)= U_{0,n}(b)=0,\\
a_2(a,0)\frac{\partial^{2}U_{0,n}}{\partial x^{2}}(a)+a_1(a,0)\frac{\partial U_{0,n}}{\partial x}(a)+f(0)=0,\\
a_2(b,0)\frac{\partial^{2}U_{0,n}}{\partial x^{2}}(b)+a_1(b,0)\frac{\partial U_{0,n}}{\partial x}(b)+f(0)=0.
\end{cases}
\end{equation}
We discuss subsequently the technique of the construction of $U_{0,n}$.

Appealing to Theorem \ref{t4.1}, we build smooth solutions $u_{n}=u_n(x,t):\Omega_{T}\to\R$
 to the approximated problems
\begin{equation}\label{8.1}
\mathbf{D}_{t}u_n-\mathcal{L}_1u_n-\mathcal{K}*\mathcal{L}_2u_n+f(u_n)=0\quad \text{in}\quad \Omega_{T},\quad
u_n(x,0)=U_{0,n}\quad\text{in}\quad\bar{\Omega},\quad
u_n(x,t)=0\quad\text{on}\quad \partial\Omega_{T}.
\end{equation}
Then we extract a convergent subsequence and pass to the limit in the equation. Using the Banach-Alaoglu Theorem and estimates stated in
 Lemma \ref{l6.1} (with $\lambda=1,\, g\equiv 0$), we can apply a standard arguments to choose a subsequence of $u_n$ (which we relabel) such that, for any fixed time $T>0$
\begin{align*}
u_n&\rightharpoonup u\quad\text{weakly in}\quad L_{2}((0,T),W^{2,2}(\Omega)),\\
K*u_n&\rightharpoonup K*u\quad\text{weakly in}\quad L_{2}((0,T),W^{2,2}(\Omega)),\\
u_n&\rightarrow u\quad\text{uniformly}\quad \C^{\alpha^*,\frac{\alpha^*\nu}{2}}(\bar{\Omega}_{T})\quad\text{for any}\quad \alpha^*\in(0,1),\\
f(u_n)&\rightarrow f(u) \quad\text{uniformly}\quad \C(\bar{\Omega}_{T}),\\
\mathbf{D}_{t}^{\nu}u_{n}&\rightharpoonup \mathbf{D}_{t}^{\nu} u_{n} \quad\text{weakly in}\quad L_{2}(\Omega_{T}),\\
\mathbf{D}_{t}^{\nu_{i}}u_{n}&\rightharpoonup \mathbf{D}_{t}^{\nu_{i}} u_{n},\, i=1,2,...,M, \quad\text{weakly in}\quad L_{2}(\Omega_{T}),\\
\mathbf{D}_{t}^{\mu_{j}}u_{n}&\rightharpoonup \mathbf{D}_{t}^{\mu_{j}} u_{n},\, j=1,2,...,N, \quad\text{weakly in}\quad L_{2}(\Omega_{T}),
\end{align*}
for some $u$ belonging to all the spaces above.

Therefore, such $u$ satisfies estimate \eqref{6.4} (with $g\equiv 0$) and
\[
\|u\|_{\C^{\alpha*,\frac{\alpha*\nu}{2}}(\bar{\Omega}_{T})}\leq C[1+\|u_0\|_{W^{2,2}(\Omega)}],
\]
besides, $u$ satisfies  equation \eqref{1.1} a.e. in $\Omega_{T}$, along with  \eqref{1.3} and  homogeneous boundary condition \eqref{1.2}. 

As for the uniqueness this solution, this fact is proved with the standard arguments. Namely, assuming the existence of two solutions $u$ and $u^*$ satisfying the same data, we consider homogenous problem \eqref{1.1}, \eqref{1.2} and \eqref{1.3} for the difference $U^*=u-u^*$, and recasting the proof leading to estimate \eqref{6.9*} (see also Section \ref{s6.5} concerning this estimate to whole time interval), we get
\[
\|U^*\|_{\C([0,T],L_{2}(\Omega))}\leq 0\quad  \Rightarrow\quad u\equiv u^*.
\]

In summary, in order to complete the proof of the existence of a unique strong solution to  \eqref{1.1}, \eqref{1.2} and \eqref{1.3} satisfying the regularity required by Theorem \ref{t4.2}, we are left to build the function $U_{0,n}$
 satisfying \eqref{8.2}.
To this end, we introduce new function
\[
U_{0}(x)=u_0(x)-P_{u_{0}}(x),
\]
where $P_{u_{0}}(x)$ is the $5^{\text{th}}$ degrees polynomial satisfying equalities
\[
\begin{cases}
P_{u_{0}}(a)=P_{u_{0}}(b)=0,\\
P'_{u_{0}}(a)=u'_{0}(a),\\
P'_{u_{0}}(b)=u'_{0}(b),\\
P''_{u_{0}}(a)=\frac{1}{a_2(a,0)}[-f(0)-a_{1}(a,0)u'_{0}(a)],\\
P''_{u_{0}}(b)=\frac{1}{a_2(b,0)}[-f(0)-a_{1}(b,0)u'_{0}(b)].
\end{cases}
\]
Here we used the smoothness of $u_0$ and embedding theorem which provides existence (in the classical sense) of $u'_{0}(a)$ and $u'_{0}(b)$.

We easily conclude that $U_0(x)\in\overset{0}{W}\,^{2,2}(\Omega)$ and, hence, we can approximate $U_{0}(x)$ by the functions $\bar{U}_{0,n}\in\C_{0}^{\infty}(\bar{\Omega})$ in the norm of $W^{2,2}(\Omega)$. Then, setting
\[
U_{0,n}(x)=\bar{U}_{0,n}(x)+P_{u_{0}}(x),
\]
and performing standard calculations, we conclude that $U_{0,n}(x)\in\C^{2+\alpha}(\bar{\Omega})$ and satisfies conditions \eqref{8.2}.
This means that $U_{0,n}$ satisfies assumptions of Theorem 4.1 and approximate the initial data $u_0$ in the corresponding classes.
Finally, coming to \eqref{8.1}, we  carry out the procedure described above  and complete the proof of Theorem \ref{t4.2}. \qed.

%%%%%%%%%%%%%%%%%%%%%%%%%%%%%%%%%%%%%%%%%%%%%%%%%%%%%%%%%%%%%%%%%%%%%%
 
%%%%%%%%%%%%%%%%%%%%%%%%%%%%%%%%%%%%%%%%%%%%%%%%%%%%%%%%%%%%%%%%%%%%%%

\section{Numerical Simulations}
\label{s9}

\noindent We finally presents some numerical tests aimed to illustrate our theoretical results (Theorems \ref{t4.1} and \ref{t4.2}) and demonstrate some effects of fractional derivatives which do not have yet analytical proofs. Namely, if all derivatives in the operator $\mathbf{D}_{t}$ (see \eqref{1.4})  are fractional, then we observe rapid changes in the solutions for small time, and  the solutions slow down as time goes up 
(see Example~\ref{e.2}). It differs from the case  $\nu=1$ in \eqref{1.4}, where this effect is negligible (and is diminishing as fractional orders grow, as we observe in Example~\ref{e.2}).

For simplicity, we focus on the initial-boundary value problem in the
one-dimensional domain $\Omega=(0,1)$ (multi-dimensional generalization of the proposed finite-difference scheme is straightforward and boils down to adding new terms that approximate the occurring new spatial partial derivatives in a completely similar manner):
\begin{equation}
\label{9.1_}
\begin{cases}
\mathbf{D}_{t}^{\nu}(\varrho_0 u)+\mathbf{D}_{t}^{\nu_1}(\varrho_1 u)-\mathbf{D}_{t}^{\mu_1}(\gamma_1 u) -\mathfrak{a}\frac{\partial^{2} u}{\partial x^{2}}+\mathfrak{d}
\frac{\partial u}{\partial x}\ -(\mathcal{K}* b\frac{\partial^{2} u}{\partial x^{2}}) = f(x,t,u)+g(x,t)\quad\text{in }\Omega_T,\\
\noalign{\vskip.3mm}
u(x,0)=u_0(x),\qquad\qquad\qquad\quad x\in[0, 1],\\
\noalign{\vskip2mm}
\mathfrak{c}_{1}\frac{\partial u}{\partial x}(0,t)+\mathfrak{c}_{2}u(0,t)=\varphi_{1}(t), \quad t\in[0, T],\\
\noalign{\vskip2mm}
\mathfrak{c}_{3}\frac{\partial u}{\partial x}(1,t)+\mathfrak{c}_{4}u(1,t)=\varphi_{2}(t), \quad t\in[0, T].
\end{cases}
\end{equation}
We introduce the space-time mesh with nodes
$$x_{k}=kh,\quad \sigma_{j}=j\sigma,\quad k=0,1,\ldots, K,
\quad j=0,1,\ldots,  J, \quad h = L/ K, \quad \sigma = T/J.$$
For these examples, we actually take $L=1$. Denoting the finite-difference approximation of the solution $u$ at the point $(x_k, \sigma_j)$ by $u^j_k$ and calling
\begin{gather*}
\mathfrak{a}^{j+1}_k = \mathfrak{a}(x_k, \sigma_{j+1}),\qquad   \mathfrak{d}^{j+1}_k  =   \mathfrak{d}(x_k, \sigma_{j+1}),\qquad
b^j_k = b(x_k, \sigma_j),\\
 \mathcal{K}_{m,j} = \int_{\sigma_m}^{\sigma_{m+1}}{\mathcal{K}}(\sigma_{j+1}-s) d s,\qquad
 \rho_m = (-1)^m\binom{\nu}{m},\qquad
 \Tilde\rho_m  = (-1)^m\binom{\nu_1}{m}, \qquad \Bar\rho_m  = (-1)^m\binom{\mu_1}{m}, \\
 \varrho_{0,k}^{j+1} = \varrho_0(x_k,\sigma_{j+1}),\qquad \varrho^{j+1}_{1,k} = \varrho_{1}(x_k, \sigma_{j+1}), \qquad \gamma^{j+1}_{1,k} = \gamma_{1}(x_k, \sigma_{j+1}),
\end{gather*}
we approximate the differential equation in \eqref{9.1_} at each time level $\sigma_{j+1}$ and spatial point $x_k$, so to obtain the
finite-difference scheme
\begin{align*}
&  \sigma^{-\nu} \sum\limits_{m=0}^{j+1} (\varrho_{0,k}^{j+1-m} u^{j+1-m}_k - \varrho_{0,k}^0 u_0(x_k))\rho_m + \sigma^{-\nu_1} \sum\limits_{m=0}^{j+1} (\varrho_{1,k}^{j+1-m} u^{j+1-m}_k - \varrho_{1,k}^{0} u_0(x_k))\Tilde\rho_m \\
& -\sigma^{-\mu_1} \sum\limits_{m=0}^{j+1} (\gamma_{1,k}^{j+1-m} u^{j+1-m}_k - \gamma_{1,k}^{0} u_0(x_k))\Bar\rho_m   - \frac{\mathfrak{a}^{j+1}_k}{h^2}(u^{j+1}_{k-1}-2u^{j+1}_{k}+u^{j+1}_{k+1})
+ \frac{  \mathfrak{d}^{j+1}_k}{2 h}(u^{j+1}_{k+1}-u^{j+1}_{k-1})\\
& = \sum_{m=0}^{j}\left(b^m_k \frac{u^{m}_{k-1}-2 u^{m}_{k}+u^{m}_{k+1}}{h^2}
+ b^{m+1}_k \frac{u^{m+1}_{k-1}-2 u^{m+1}_{k}+u^{m+1}_{k+1}}{h^2}\right)\!\frac{\mathcal{K}_{m,j}}{2}
+f(x_k,\sigma_{j},u_k^j) +g(x_k,\sigma_{j+1}),
\end{align*}
for
$$k = 1,\ldots,  K-1\qquad\text{and}\qquad j = 0,1,\ldots,  J-1.$$
Here, the derivatives $u_x$ and $u_{x x}$ are approximated by
the standard central finite-difference formulas, the trapezoid rule is employed to approximate the integrals in the sum 
$$\sum_{m=0}^j\int^{\sigma_{m+1}}_{\sigma_m} {\mathcal{K}}(\sigma_{j+1}-s)b(x,s)u_{xx}(x,s) d s$$
 representing the convolution term in \eqref{9.1_}, and the Gr\"unwald-Letnikov (GL) formula \cite{KST, DFFL, JinLazZ} is applied to approximate the fractional derivatives $\mathbf{D}_{t}^{\nu}(\varrho_0 u)$,  $\mathbf{D}_{t}^{\nu_{1}}(\varrho_1 u)$ and $\mathbf{D}_{t}^{\mu_1}(\gamma_1 u)$. Alternatively, here one might also use the ``Leibniz rule'' for Caputo derivatives (see~\cite[Corollary~3.1]{KPSV4}) to reduce the treating of fractional derivatives terms in \eqref{9.1_} to the approximation of fractional derivatives of the segregated function $u$ only (but not its products with $\varrho_i$ or $\gamma_i$) by the cost of adding integral extra terms (which can be approximated using the same approach as it was done for the convolution term in \eqref{9.1_}) and  making slight changes in the right part~$g$. However we will not pursue this way further here. Also, in order to achieve an improvement in the temporal discretization accuracy (primarily, owing to the approximation of the fractional derivatives) we apply the Richardson extrapolation (see~\cite{DFFL}). Finally, two fictitious mesh points outside the spatial domain to approximate the derivatives in the boundary conditions with the second order of accuracy are exploited~\cite{PSV6,KPSV}. Further improvement in the accuracy of calculations may be reached by resorting to finite element methods~\cite{JinLazZ, SVjcp}, albeit we do not have the possibility to pursue this direction further here.
\begin{exAPP}\label{e.1}
%\begin{example}\label{e.1}
Consider problem \eqref{9.1_} with  $T=1$ and
\begin{align*}
\mathcal{K}(t)  &= t^{-1/3}, \quad  \mathfrak{a}(x,t) = \cos(\pi x / 4) +t,\quad
  \mathfrak{d}(x,t) = x+t,\quad  b(x,t) = t^{1/3}+\sin(\pi x),\quad u_{0}(x) = \cos(\pi x),\\
  \varrho_0(t) &= 1+t,\quad \varrho_1=1/2,\quad \gamma_1(t)=(1+t^2)/2, \quad \mathfrak{c}_{1} = \mathfrak{c}_{3}=1,\quad \mathfrak{c}_{2}=\mathfrak{c}_{4}=0,\quad
 \varphi_{1}(t)=\varphi_{2}(t)=0, \\
  g(x,t)&= \pi^2 \Bigl(\cos\frac{\pi x}{4} + t+ \frac{3t^{2/3}\sin(\pi x) }{2} + \frac{t \pi }{3\sin(\pi/3)}\Bigr)\cos(\pi x) - x t \sin((\cos(\pi x)+t^{\nu}/\Gamma(1+\nu))^2) \\
  & - (x+t)\pi\sin(\pi x) + 1 + \frac{t^{1-\nu}\cos(\pi x)}{\Gamma(2-\nu)} + (1+\nu) t + \frac{t^{\nu-\nu_1}}{2\Gamma(1+\nu-\nu_1)} - \frac{1}{2}\Bigl( \frac{t^{\nu-\mu_1}}{\Gamma(1+\nu-\mu_1)} \\ 
  & +\frac{2 t^{2-\mu_1} \cos(\pi x)}{\Gamma(3-\mu_1)} + \frac{(2+\nu)(1+\nu) t^{2+\nu-\mu_1}}{\Gamma(3+\nu-\mu_1)} \Bigr), \quad f(x,t,u) = x t \sin(u^2).
\end{align*}
It is easy to verify that the function $$u(x,t)=\cos(\pi x)+\frac{t^{\nu}}{\Gamma(1+\nu)}$$ solves initial-boundary value problem \eqref{9.1_} with the parameters specified above. 

The outcomes of this example (the absolute error $\gimel = \max|u-u_{\mathsf{N}}|$ between $u$ and the numerical solution $u_{\mathsf{N}}$, where the maximum is taken over all the grid points in the space-time mesh) are listed in Table~\ref{Example1:table1}. One can observe from Table~\ref{Example1:table1} the rapid decaying of errors as mesh refines and that relatively coarse meshes provide quite small errors (as compared to the exact solution magnitude); one can also observe that decreasing fractional orders leads to a decrease in the computation accuracy, which is in line with the asymptotic estimates obtained in \cite{CHS_} for the accuracy of GL approximations (asserting that, in general, their accuracy degrades with decreasing a fractional order in case of weakly singular solutions and close to the singularity points, see \cite[Theorems 3.1, 4.1]{CHS_} for details).
\begin{table}[htbp]
  \begin{center}
    \caption{Values of $\gimel$ in Example~\ref{e.1}; $\nu_1=\nu/3$, $\mu_1=\nu/2$.}
    \label{Example1:table1}
    \begin{tabular}{c|c|c|c}
      \hline
     $\nu$ & \multicolumn{3}{c}{$\gimel$}\\ % $\gimel$&$\gimel$&$\gimel$\\
		 \hline
		$0.1$& 2.4443e-02 & 9.0291e-03 & 5.4896e-03 \\
		$0.2$& 2.4006e-02 & 7.8603e-03 & 5.0437e-03 \\
		$0.3$& 2.3429e-02 & 7.2815e-03 & 4.9720e-03 \\
		$0.4$& 2.2749e-02 & 6.4123e-03 & 4.4865e-03 \\
		$0.5$& 2.1996e-02 & 5.6072e-03 & 3.5239e-03 \\
		$0.6$& 2.1195e-02 & 5.4235e-03 & 2.4460e-03 \\
		$0.7$& 2.0369e-02 & 5.2311e-03 & 2.3659e-03 \\
		$0.8$& 1.9538e-02 & 5.0338e-03 & 2.2819e-03 \\
		$0.9$& 1.8722e-02 & 4.8379e-03 & 2.1973e-03 \\ \hline
		     & $K=J=10$ & $K=J=20$ & $K=J=30$
	\end{tabular}
  \end{center}
\end{table}
\end{exAPP}
%\end{example}

\begin{exAPP}\label{e.2}
%\begin{example}\label{e.2}
Consider problem~\eqref{9.1_} with the constant coefficients $\varrho_0 = 1$, $\varrho_1=1/2$, $\gamma_1 = 1/2$, $g=0$, while the remaining coefficients being as in Example~\ref{e.1}. As for the function $f$, we test here two options:
\begin{itemize}
\item[(i)] $f(x,t,u) = 0$ (linear problem),
\item[(ii)]$f(x,t,u) = x t \cos(u^2)$.
\end{itemize}
Solutions to this example are drawn in Figures \ref{fig:e2:linear} and \ref{fig:e2:nonlinear} for different fractional orders $\nu$ (with $\nu_1=\nu/3$ and $\mu_1=\nu/2$) and time points $t$; the steps $\sigma=h=10^{-3}$ are employed. One can observe in these figures that adding the non-linearity can noticeably change the solution behavior (especially for low values of fractional orders).
\begin{figure}
\includegraphics[scale=0.49]{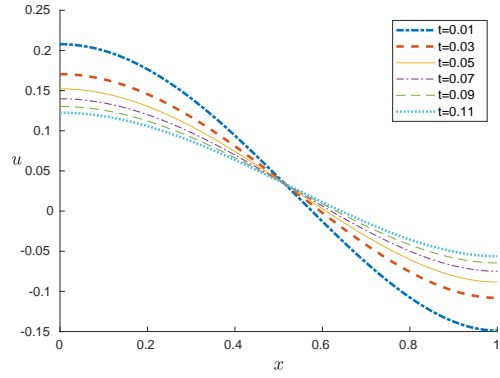}
\centerline{\text{(a)} $\nu=0.2$}
\label{lin_u_alpha=0.2}
\includegraphics[scale=0.49]{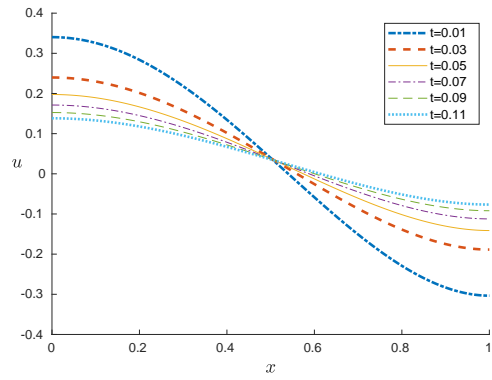}
\centerline{\text{(b)} $\nu=0.4$}
\label{lin_u_alpha=0.4}
\includegraphics[scale=0.49]{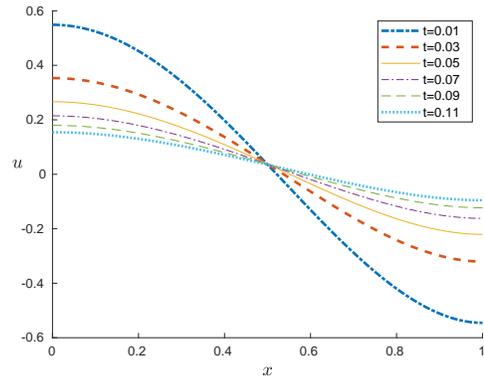}
\centerline{\text{(c)} $\nu=0.6$}
\label{lin_u_alpha=0.6}
\includegraphics[scale=0.49]{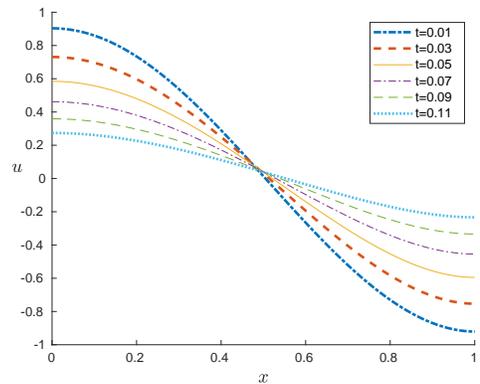}
\centerline{\text{(d)} $\nu=0.8$}
\label{lin_u_alpha=0.8}
\caption{Solutions to Example~\ref{e.2} with $f(x,t,u) = 0$, $\nu_1=\nu/3$, $\mu_1=\nu/2$.}
    \label{fig:e2:linear}
\end{figure}
\begin{figure}
\includegraphics[scale=0.484]{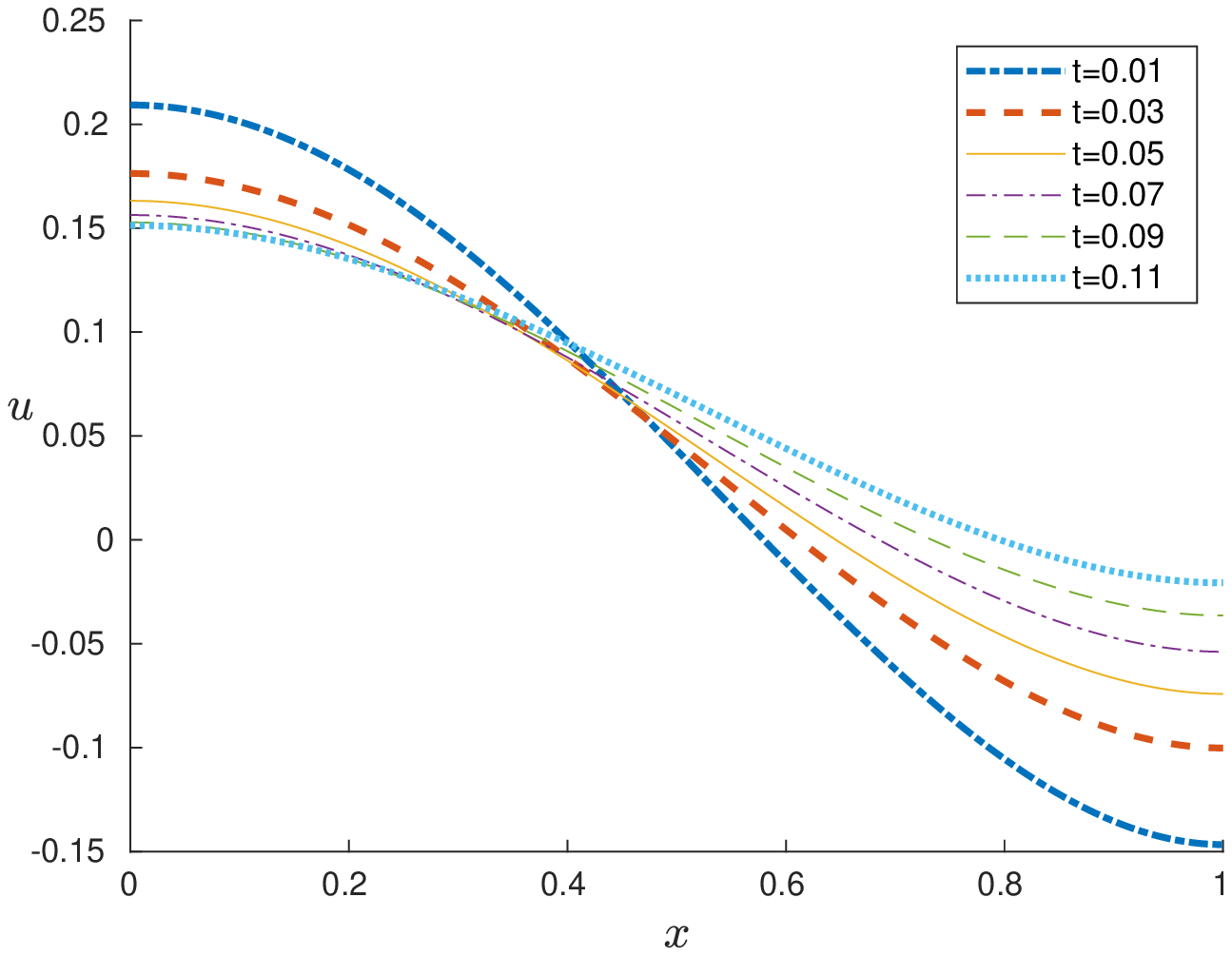}
\centerline{\text{(a)} $\nu=0.2$}
\label{nonlin_u_alpha=0.2}
\includegraphics[scale=0.484]{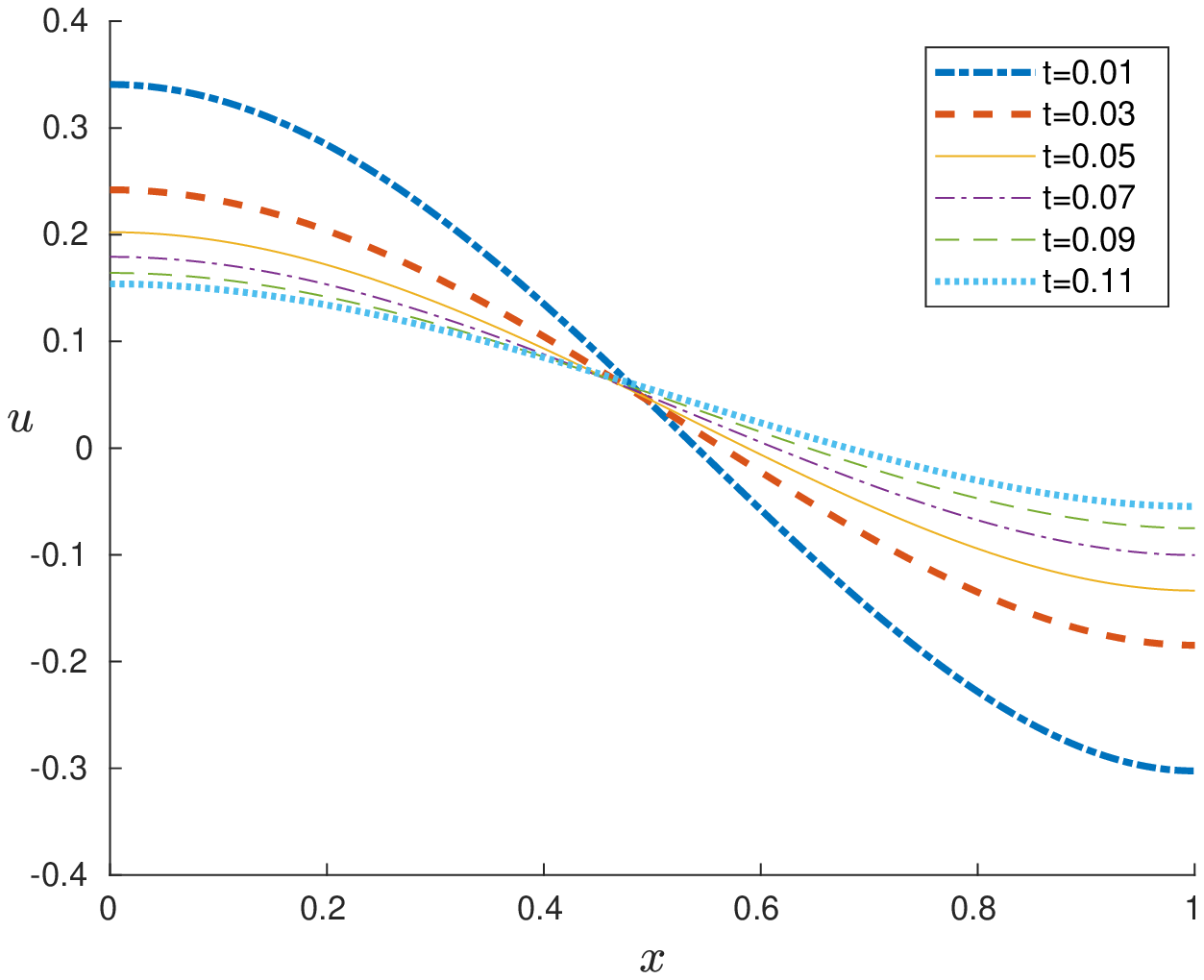}
\centerline{\text{(b)} $\nu=0.4$}
\label{nonlin_u_alpha=0.4}
\includegraphics[scale=0.484]{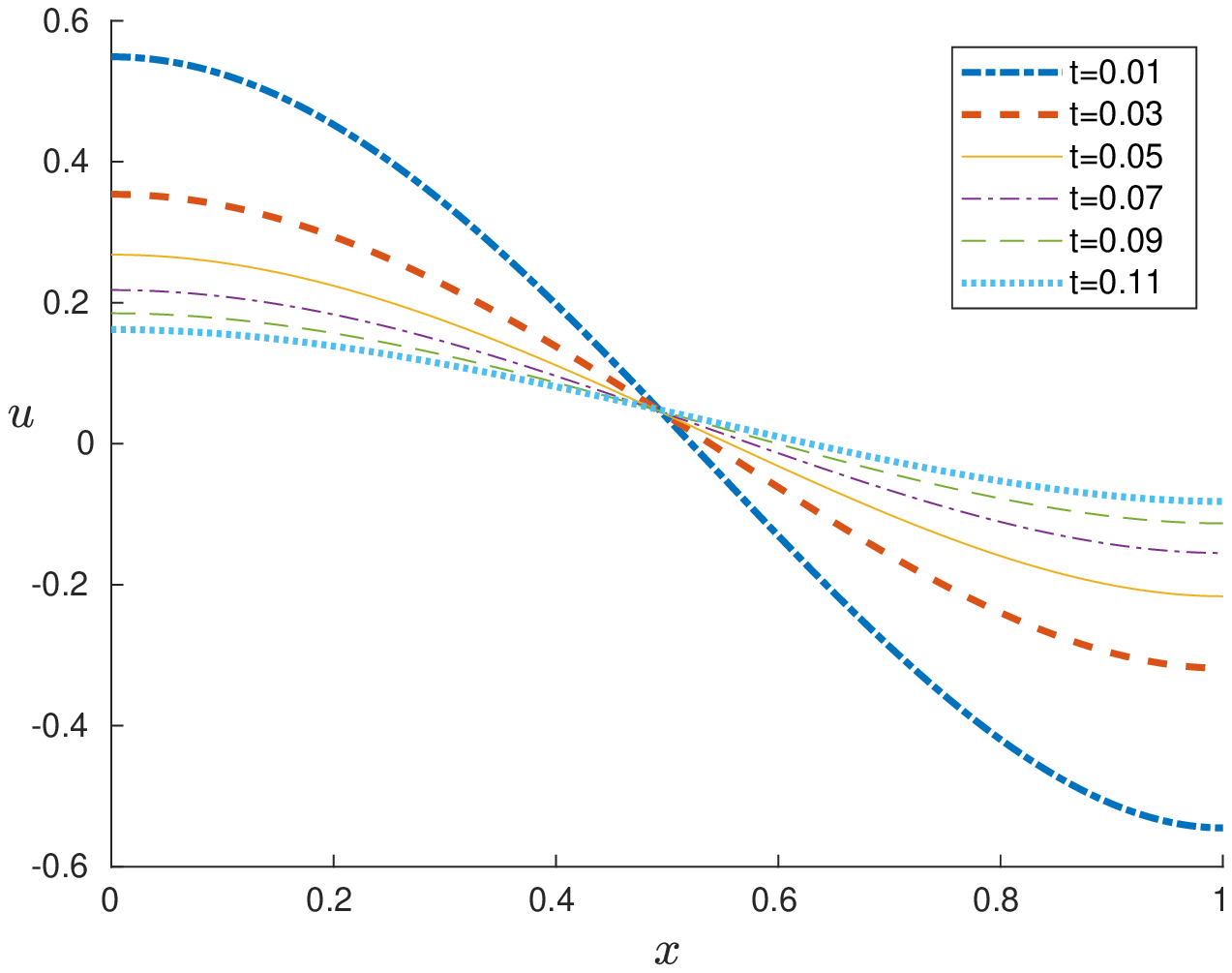}
\centerline{\text{(c)} $\nu=0.6$}
\label{nonlin_u_alpha=0.6}
\includegraphics[scale=0.484]{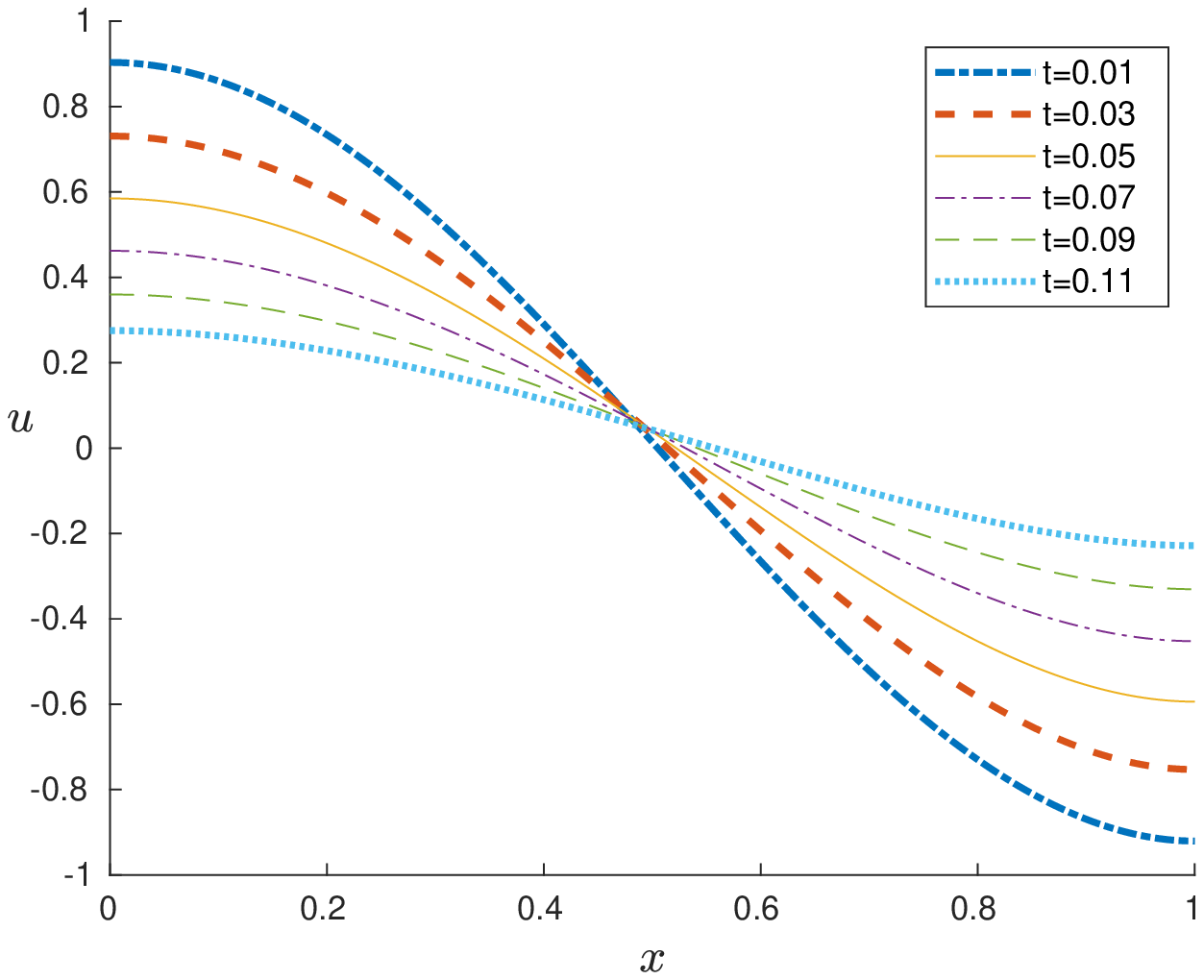}
\centerline{\text{(d)} $\nu=0.8$}
\label{nonlin_u_alpha=0.8}
\caption{Solutions to Example~\ref{e.2} with $f(x,t,u) = x t \cos(u^2)$, $\nu_1=\nu/3$, $\mu_1=\nu/2$.}
    \label{fig:e2:nonlinear}
\end{figure}
\end{exAPP}
%\end{example}

%%%%%%%%%%%%%%%%%%%%%%%%%%%%%%%%%%%%%%%%%%%%%%%%%%%%%%%%%%%%%%%%%%%%%%

%%%%%%%%%%%%%%%%%%%%%%%%%%%%%%%%%%%%%%%%%%%%%%%%%%%%%%%%%%%%%%%%%%%%%%
\subsection*{Acknowledgments} The second author was partially  supported by the Foundation of  The European Federation of Academy of Sciences and Humanities (ALLEA),
the Grant EFDS-FL2-08.
%%%%%%%%%%%%%%%%%%%%%%%%%%%%%%%%%%%%%%%%%%%%%%%%%%%%%%%%%%%%%%%%%%%%%%

%%%%%%%%%%%%%%%%%%%%%%%%%%%%%%%%%%%%%%%%%%%%%%%%%%%%%%%%%%%%%%%%%%%%%
 
%%%%%%%%%%%%%%%%%%%%%%%%%%%%%%%%%%%%%%%%%%%%%%%%%%%%%%%%%%%%%%%%%%%%%%

\end{document}